\documentclass[a4paper,10pt]{amsart}

\usepackage[T1]{fontenc}
\usepackage{lmodern}
\usepackage[utf8]{inputenc}
\usepackage[english]{babel}
\usepackage{geometry}

\usepackage{booktabs}
\usepackage{color}
\usepackage{amsthm}
\usepackage{amsmath,amssymb}
\usepackage{graphicx}
\usepackage{listings}
\usepackage{emptypage}
\usepackage{newlfont}
\usepackage{verbatim}
\usepackage{amssymb}
\usepackage{array}
\usepackage{mathtools}
\usepackage{afterpage}
\usepackage{csquotes}
\usepackage{tabularx}
\usepackage{xcolor}
\usepackage{soul}
\usepackage{etoolbox}
\usepackage{hyperref}

\newcommand{\numberset}{\mathbb}
\newcommand{\N}{\numberset{N}}

\newcommand{\R}{\numberset{R}}

\newcommand{\ud}{\mathrm{d}}

\theoremstyle{plain} 
\newtheorem{thm}{Theorem} 
\newtheorem{cor}[thm]{Corollary}
\newtheorem{lem}[thm]{Lemma}  
\newtheorem{prop}[thm]{Proposition}
\newtheorem*{theorem*}{Teorema}

\theoremstyle{definition} 
\newtheorem{defn}[thm]{Definition}

\newtheorem{rem}[thm]{Remark} 
\newtheorem{notation}{Notation}

\usepackage{setspace}
\usepackage{fancyhdr}
\usepackage{pdfpages}
\usepackage{standalone}
\usepackage{appendix}

\numberwithin{equation}{section}
\numberwithin{thm}{section}

\title[Hamilton-Jacobi equations for the $N$-body problem ]{\textbf{On the regularity of solutions to the Hamilton-Jacobi equations for the $N$-body problem}}
\date{}


\thanks{All authors are affiliated to INDAM-GNAMPA research group. This work is partially supported by the PRIN 2022 project 20227HX33Z -- \emph{Pattern formation in nonlinear phenomena}.  D. B. was supported by MUR - M4C2 1.5 of PNRR with grant no. ECS00000036.}

\author{Diego Berti}
\address{Dipartimento di Matematica “Giuseppe Peano”\\Università degli Studi di Torino}
\email{diego.berti@unito.it}

\author{Davide Polimeni}
\address{Dipartimento di Matematica “Giuseppe Peano”\\Università degli Studi di Torino}
\email{davide.polimeni@unito.it}

\author{Susanna Terracini}
\address{Dipartimento di Matematica “Giuseppe Peano”\\Università degli Studi di Torino}
\email{susanna.terracini@unito.it}

\begin{document}

\begin{abstract}

We prove that certain renormalized value functions associated with the $d$-dimensional ($d\geq2$) $N$-body problem corresponding to different limiting shapes of expanding solutions, under the assumption that the center of mass is at the origin, are viscosity solutions of the associated Hamilton-Jacobi equation. We analyze their singularities, defined as the initial configurations for which the minimizer of the associated variational problem is not unique. Moreover, we estimate the size of the closure of the singular set by proving its $\mathcal{H}^{d(N-1)-1}$-rectifiability, and we provide an upper bound on the Hausdorff dimension of the set of regular conjugate points.

\end{abstract}

\maketitle

\tableofcontents

\section{Introduction}The Hamilton-Jacobi equation plays a fundamental role in the analysis of dynamical systems, particularly in Celestial Mechanics. The Newtonian $N$-body problem -- describing the motion of $N$ point masses under mutual gravitational attraction -- provides a rich framework for studying viscosity solutions of the equation. Over the last two decades, substantial progress has been made in characterizing such solutions and connecting them to the geometric and variational properties of the $N$-body system (cf., e.g., \cite{MadernaVenturelli_GloballyMinimizingParabolic,MadernaVenturelli_HyperbolicMotions,Maderna2012, PercinoSanchez2014}).

We consider the classical Newtonian $N$-body problem in the Euclidean space $\mathbb{R}^d$, with $d \geq 2$, where $N$ point masses $m_i > 0$ occupy positions $r_i \in \mathbb{R}^d$. Newton's equation of motion for the $i$-th body can be written as:
\begin{displaymath}\label{eq:newton}
    m_i\ddot{r}_i = -\sum_{\substack{j=1 \\ j\neq i}}^{N} m_i m_j \frac{r_i - r_j}{|r_i - r_j|^3}.
\end{displaymath}
Since Newton's equations are invariant under translation, we may fix the origin of our inertial frame at the center of mass of the system and consider the configuration space with zero barycenter:
\[
    \mathcal{X} = \left\{x=(r_1,\ldots,r_N)\in\mathbb{R}^{dN} : \sum_{i=1}^{N} m_i r_i = 0\right\},
\]
and denote by
\[
    \Omega = \left\{x\in \mathcal{X} \,\middle|\, r_i \neq r_j \text{ for all } i \neq j\right\}
\]
the set of collision-free configurations, which is open and dense in $\mathcal{X}$. We write $\Delta = \mathcal{X} \setminus \Omega$ for the collision set.

The Newtonian potential is given by:
\[
    U(x) = \sum_{1 \leq i < j \leq N} \frac{m_i m_j}{|r_i - r_j|},
\]
so that the equations of motion take the more compact form:
\begin{equation}\label{eq_newton}
    \mathcal{M}\ddot{x} = \nabla U(x),
\end{equation}
where $\mathcal{M} = \text{diag}(m_1 I_d, \ldots, m_N I_d)$. The associated Lagrangian $L : T\Omega \to \mathbb{R} \cup \{+\infty\}$ and Hamiltonian $H : T^*\Omega \to \mathbb{R} \cup \{-\infty\}$ are defined by
\[
    L(x,v) = \frac{1}{2} \|v\|_\mathcal{M}^2 + U(x), \qquad
    H(x,p) = \frac{1}{2} \|p\|_{\mathcal{M}^{-1}}^2 - U(x),
\]
{where $\|\cdot\|_\mathcal{M}$ is the norm induced by the mass inner product on $\mathcal{X}$, given by
\[
    \langle x, y \rangle_\mathcal{M} = \sum_{i=1}^{N} m_i \langle r_i, s_i \rangle,
    \qquad x=(r_1,\ldots,r_N),\; y=(s_1,\ldots,s_N)\in \mathcal{X},
\]
and $\|\cdot\|_{\mathcal{M}^{-1}}$ denotes the dual norm on $T^*\Omega$, induced by the inverse mass metric. Here $\langle \cdot,\cdot \rangle$ is the standard Euclidean scalar product in $\mathbb{R}^d$.}

We are interested in global solutions $v : \Omega \to \mathbb{R}$ of the stationary Hamilton-Jacobi equation at a fixed energy level $h \geq 0$:
\begin{equation}\label{eq:HJ}
   H(x,\nabla v(x))= \frac12 \|\nabla v(x)\|_{\mathcal M^{-1}}^2 - U(x)=h,
\end{equation}
where, here and throughout the paper, $\nabla$ denotes the Euclidean gradient, having fixed the canonical basis of $T\Omega$.

 Our goal is to address existence and regularity questions concerning viscosity solutions associated with expansive solutions of the $N$-body problem, obtained via minimization of a suitably renormalized action functional with prescribed initial data.

\begin{defn}
    A motion $\gamma:[0,+\infty) \rightarrow \Omega$ is said to be \emph{expansive} if all mutual distances diverge, that is, $|r_i(t)-r_j(t)| \rightarrow +\infty$ as $t \rightarrow +\infty$ for all $i < j$. Equivalently, $\gamma$ is expansive if $U(\gamma(t)) \rightarrow 0$ as $t \rightarrow +\infty$.
\end{defn}

From energy conservation, it follows that expansive motions can only exist at nonnegative energy levels, since the condition $U(\gamma(t)) \to 0$ as $t \to +\infty$ implies $\|\dot{\gamma}(t)\|_\mathcal{M}^2 \to 2h$. To classify expansive motions, we define the minimum and maximum separation between bodies along a motion $\gamma$ at time $t$:
\[
    r(t) = \min_{i<j} |r_i(t) - r_j(t)|, \qquad R(t) = \max_{i<j} |r_i(t) - r_j(t)|.
\]

The following classical results, dating back to the 1970s, describe the asymptotic behavior of expansive motions. For positive functions $f$ and $g$, we write $f \approx g$ if there exist constants $\alpha, \beta > 0$ such that $\alpha \leq \frac{f}{g} \leq \beta$.

\begin{thm}\label{PollardTheorem}
\begin{itemize}
    \item[$(i)$] (Pollard, 1967 \cite{Pollard_BehaviorOfGravitationalSystems}) Let $\gamma$ be a motion defined for all $t > t_0$. If $r(t)$ is bounded away from zero, then $R(t) = O(t)$ as $t \rightarrow +\infty$. Moreover, $R(t)/t \rightarrow +\infty$ if and only if $r(t) \rightarrow 0$.
    
    \item[$(ii)$] (Marchal-Saari, 1976 \cite{MarchalSaari_FinalEvolution}) Let $\gamma$ be a motion defined for all $t > t_0$. Then either $R(t)/t \to +\infty$ and $r(t) \to 0$, or there exists a configuration $a \in \mathcal{X}$ such that $\gamma(t) = at + O(t^{2/3})$. In particular, for \emph{superhyperbolic} motions -- i.e., motions such that $\limsup_{t \to +\infty} R(t)/t = +\infty$ -- the quotient $R(t)/t$ diverges.
    
    \item[$(iii)$] (Marchal-Saari, 1976 \cite{MarchalSaari_FinalEvolution}) Suppose $\gamma(t) = at + O(t^{2/3})$ for some $a \in \mathcal{X}$, and that the motion is expansive. Then for each pair $i < j$ with $a_i = a_j$, we have $|r_i(t)-r_j(t)| \approx t^{2/3}$.
\end{itemize}
\end{thm}

As a consequence, expansive motions cannot be superhyperbolic, and therefore must satisfy $\gamma(t) = at + O(t^{2/3})$ for some asymptotic velocity $a \in \mathcal{X}$. Chazy (\cite{Chazy}) provided a complete classification of expansive motions with zero barycenter, based on the asymptotic growth rate of mutual distances. They fall into the following categories:
\begin{itemize}
    \item[$(H)$] \textit{Hyperbolic}: $a \in \Omega$, with $|r_i(t) - r_j(t)| \approx t$ for all $i < j$;
    \item[$(P)$] \textit{Completely parabolic}: $a = 0$, with $|r_i(t) - r_j(t)| \approx t^{2/3}$ for all $i < j$;
    \item[$(HP)$] \textit{Hyperbolic-parabolic}: $a \in \Delta$, $a \neq 0$.
\end{itemize}

To further refine this classification, we introduce the notion of \emph{limit shape}, where we consider the diagonal action $S(t)\gamma = (S(t)r_1, \ldots, S(t)r_N)$. 

\begin{defn}
    A motion $\gamma(t)$ is said to have a \emph{limit shape} if there exists a family of similarity transformations $S(t)$ of $\mathbb{R}^d$ such that $S(t)\gamma(t) \to a \neq 0$ as $t \to +\infty$.
\end{defn}

For (half) hyperbolic motions, the limit shape is simply the asymptotic velocity $a = \lim_{t\to\infty} \frac{\gamma(t)}{t}$; for (half) parabolic motions, if a limit shape exists, it must be a central configuration.

\begin{defn}\label{def:central_configuration}
    A configuration $b \in \mathcal{X}$ is called a \emph{central configuration} if it is a critical point of $U$ restricted to the inertia ellipsoid
    \[
        \mathcal{E} = \left\{x \in \mathcal{X} : \langle \mathcal{M} x, x \rangle = 1\right\}.
    \]
    A central configuration $b_m \in \mathcal{E}$ is said to be \emph{minimal} if it minimizes $U$ over $\mathcal{E}$:
    \[
        U(b_m) = \min_{b \in \mathcal{E}} U(b).
    \]
\end{defn}

Polimeni and Terracini \cite{PolimeniTerracini} generalized and refined earlier results by Maderna and Venturelli \cite{MadernaVenturelli_GloballyMinimizingParabolic,MadernaVenturelli_HyperbolicMotions}, proving the existence of half-entire expansive solutions of all three types $(H)$, $(P)$, and $(HP)$ via a unified variational framework. Their approach relies on minimizing a renormalized Lagrangian action functional using the direct method in the calculus of variations.

\begin{thm}[Maderna and Venturelli, 2020 \cite{MadernaVenturelli_HyperbolicMotions}]\label{thm_hyperbolic}
    Given $d \geq 2$, for the Newtonian $N$-body problem in $\mathbb{R}^d$, there exists a hyperbolic solution $\gamma : [1,+\infty) \to \mathcal{X}$ of the form
    \[
        \gamma(t) = at - \log(t){\nabla_{\mathcal M}} U(a) + O(1), \quad \text{as } t \to +\infty,
    \]
    for any initial condition $x = \gamma(1) \in \mathcal{X}$ and any collision-free configuration $a \in \Omega$. {Here, $\nabla_{\mathcal M}$ denotes the gradient with respect to $\langle \cdot, \cdot \rangle_{\mathcal M}$.}
\end{thm}

The method in \cite{MadernaVenturelli_HyperbolicMotions} constructs global viscosity solutions to the Hamilton-Jacobi equation \eqref{eq:HJ} at positive energy, associated with a given collision-free asymptotic velocity $a$. These solutions are shown to be fixed points of the associated Lax-Oleinik semigroup.

For the parabolic case, the result of \cite{MadernaVenturelli_GloballyMinimizingParabolic} was improved in \cite{PolimeniTerracini}, providing sharper asymptotic estimates:

\begin{thm}[Maderna and Venturelli, 2009 \cite{MadernaVenturelli_GloballyMinimizingParabolic}; Polimeni and Terracini, 2024 \cite{PolimeniTerracini}]\label{thm_parabolic}
    Given $d \geq 2$, for the Newtonian $N$-body problem in $\mathbb{R}^d$, there exists a parabolic solution $\gamma : [1,+\infty) \to \mathcal{X}$ of the form
    \begin{equation}\label{eq:parabolic}
        \gamma(t) = \beta b_m t^{2/3} + o(t^{1/3^+}), \quad \text{as } t \to +\infty,
    \end{equation}
    for any initial configuration $x = \gamma(1) \in \mathcal{X}$, any minimal normalized central configuration $b_m$, and $\beta = \sqrt[3]{\frac{9}{2}U(b_m)}$.
\end{thm}

We point out that in Theorem \ref{thm_parabolic}, as well as in the following Theorem \ref{thm_partially_hyperbolic}, the condition that the central configuration is minimal is somehow not exhaustive, and the possibility of the existence of geodesic rays with other configurations cannot be excluded.

This variational framework also led to the first general existence results for hyperbolic-parabolic motions, where the asymptotic velocity $a \in \Delta \setminus \{0\}$ lies in the collision set, extending earlier results in \cite{Burgos_PartiallyHyperbolic}. Following \cite{BurgosMaderna_GeodesicRays, PolimeniTerracini}, we can define an \emph{$a$-cluster partition} via the equivalence relation:
\begin{equation}\label{eq:equivalence_relation}
    i \sim j \Longleftrightarrow a_i = a_j.
\end{equation}
Given a cluster $K$, they define the \emph{partial potential} $U_K$ (the restriction of $U$ to cluster $K$) and the \emph{$a$-clustered potential} $U_a$ (the sum over all cluster potentials $U_K)$. Notice that, in \cite{BurgosMaderna_GeodesicRays}, it is shown the key property that each of these partitions gives rise to an orthogonal decomposition for the inner product of the masses.

\begin{thm}[Polimeni and Terracini, 2024 \cite{PolimeniTerracini}]\label{thm_partially_hyperbolic}
    Given $d \geq 2$, for the Newtonian $N$-body problem in $\mathbb{R}^d$, there exists a hyperbolic-parabolic motion $\gamma : [1,+\infty) \to \mathcal{X}$ of the form
    \begin{equation}\label{eq:hyperbolic_parabolic}
           \gamma(t) = at + \beta b_m t^{2/3} + o(t^{1/3^+}), \quad \text{as } t \to +\infty,
    \end{equation}
    
    for any initial configuration $x = \gamma(1) \in \mathcal{X}$, any collision configuration $a \in \Delta$, any normalized minimal central configuration $b_m$ of the $a$-clustered potential, and any $h > 0$.
\end{thm}

This decomposition {reveals} that among hyperbolic-parabolic motions there are those consisting of clusters of bodies whose centers of mass diverge linearly, while internal distances within each cluster grow like $t^{2/3}$, converging to a prescribed limit shape given by a minimal configuration of $U_K$. It is worth mentioning that the existence of hyperbolic-parabolic motions whose limit shape is a non-minimal central configurations is an open problem. The action functional decomposes accordingly, separating global cluster dynamics from internal cluster structure.

\begin{cor}[Polimeni and Terracini, 2024 \cite{PolimeniTerracini}]
    The motions $\gamma(t)$ in Theorems \ref{thm_hyperbolic}, \ref{thm_parabolic}, and \ref{thm_partially_hyperbolic} are continuous at $t = 1$ and collision-free for $t > 1$. Moreover, they are free-time minimizers of the action at their respective energy levels.
\end{cor}

We are now ready to state the main result of this paper.

\begin{thm}
    Let $a \in \Omega$ (type $(H)$), or let $b_m$ be a minimal central configuration of $U$ (type $(P)$), or let $a \in \Delta$ and $b_m$ be a normalized minimal central configuration of the $a$-clustered potential (type $(HP)$). Then, there exists a viscosity solution to the $N$-body Hamilton-Jacobi equation \eqref{eq:HJ}. The singular set of such a solution is a countably $\mathcal{H}^{d(N-1)-1}$-rectifiable subset of the configuration space $\mathcal{X}$. Moreover, we have
    \[
\mathrm{dim}_{\mathcal H}(\Gamma \setminus \Sigma) \le d(N-1)-2, 
\]
where $\Gamma$ and $\Sigma$ denote respectively the conjugate and irregular sets of points (Definitions \ref{def:irregular} and \ref{def:conjugate}).

\end{thm}

\begin{rem}
    Our viscosity solutions to the Hamilton-Jacobi equation extend continuously to collision configurations, although they are no longer Lipschitz continuous at those points. Nevertheless, the results in \cite{Maderna2012, MadernaVenturelli_HyperbolicMotions} imply that they remain globally $1/2$-H\"older continuous on $\mathcal X$, including at collision configurations.
    \end{rem}

    \begin{rem}
        
In the case $h=0$, in \cite{Maderna2012}, explicit Busemann-type global solutions of are constructed for the two-body (Kepler) problem. In particular, setting $N=2$, $d=1$ and $m_1=m_2=1$ the planar Kepler problem can be seen as a fixed-center problem, with configuration variable $x=(x_1,x_2)\in\mathbb R^2$, equation \eqref{eq:HJ} reads as $|\nabla_x u(x)|^2 = 2/|x| $. Besides radial solutions, whose calibrated curves (straight lines) are parabolic homothetic motions, this equation also admits an explicit Busemann-type solution, which is not differentiable along a half-line and whose calibrated curves are half-parabolic trajectories (see \cite[Section 4]{Maderna2012}).
    \end{rem}
\subsection{The Variational Setting and the Renormalized Action Principle}

We now outline the variational approach used in \cite{PolimeniTerracini} to prove the existence of hyperbolic, parabolic, and hyperbolic-parabolic motions, relying on the optimization of a suitably renormalized action functional.

For the $N$-body problem, the Hamiltonian $H$ is defined on $\Omega \times \mathbb{R}^{dN}$ by
\begin{equation}\label{eq:hamiltonian}
    H(x, p) = \frac{1}{2} \|p\|_{\mathcal{M}^{-1}}^2 - U(x),
\end{equation}
where $\|\cdot\|_{\mathcal{M}^{-1}}$ denotes the dual norm with respect to the mass inner product. The Lagrangian is defined analogously by
\begin{equation}\label{eq:lagrangian}
    L(x, v) = \frac{1}{2} \|v\|_\mathcal{M}^2 + U(x).
\end{equation}

Given two configurations $x, y \in \mathcal{X}$ and $T > 0$, we define the set of admissible paths:
\[
\mathcal{C}(x, y, T) = \left\{ \gamma : [a, b] \rightarrow \mathcal{X} \,\middle|\, \gamma \text{ is absolutely continuous},\, \gamma(a) = x,\, \gamma(b) = y,\, b - a = T \right\},
\]
and the free-time path space:
\[
\mathcal{C}(x, y) = \bigcup_{T > 0} \mathcal{C}(x, y, T).
\]

The Lagrangian action of a curve $\gamma \in \mathcal{C}(x, y, T)$ is defined as
\[
    \mathcal{A}_L(\gamma) = \int_{a}^{b} L(\gamma, \dot{\gamma})\, \mathrm{d}t = \int_{a}^{b} \left( \frac{1}{2} \|\dot{\gamma}\|_\mathcal{M}^2 + U(\gamma) \right) \mathrm{d}t.
\]

According to Hamilton's principle of least action, if a curve $\gamma \in \mathcal{C}(x, y, T)$ minimizes the Lagrangian action, then it satisfies Newton's equations at every instant $t \in [a, b]$ for which $\gamma(t)$ is collision-free. However, due to the existence of curves with isolated collisions and finite action (see \cite{Poincare_SolutionsPeriodiques}), minimizing paths not always correspond to genuine solutions. Marchal's Principle (Theorem \ref{thm_marchal}) addresses this issue by ensuring that action-minimizing curves are free of collisions in the interior of the interval, thereby enabling the use of variational methods.

The idea of proving Marchal's Principle via averaged variations originates from Marchal \cite{Marchal_MethodOfMinimization}, with more complete arguments provided by Chenciner \cite{Chenciner}, and Ferrario and Terracini \cite{FerrarioTerracini}.

\begin{thm}[Marchal \cite{Marchal_MethodOfMinimization}, Chenciner \cite{Chenciner}, Ferrario and Terracini \cite{FerrarioTerracini}]\label{thm_marchal}
Assume $d\geq 2$, let $x, y \in \mathcal{X}$, and let $\gamma \in \mathcal{C}(x, y)$ be defined on some interval $[a, b]$. If
\[
    \mathcal{A}_L(\gamma) = \min \left\{ \mathcal{A}_L(\sigma)\ \middle|\ \sigma \in \mathcal{C}(x, y, b-a) \right\},
\]
then $\gamma(t) \in \Omega$ for all $t \in (a, b)$.
\end{thm}

To state the Renormalized Action Principle, we introduce the following notion:

\begin{defn}\label{def:free-time_minimizers}
A curve $\gamma: I \to \mathcal{X}$ is a \emph{free-time minimizer} for the Lagrangian action at energy $h$ if, for all intervals $[a, b], [a', b'] \subset I$ and all curves $\sigma: [a', b'] \to \mathcal{X}$ with $\gamma(a) = \sigma(a')$ and $\gamma(b) = \sigma(b')$, we have:
\[
    \int_a^b L(\gamma, \dot{\gamma})\, \mathrm{d}t + h(b - a) \leq \int_{a'}^{b'} L(\sigma, \dot{\sigma})\, \mathrm{d}t + h(b' - a').
\]
\end{defn}

The main strategy is to construct solutions to the $N$-body problem of the form:
\[
    \gamma(t) = r_0(t) + \varphi(t) + x - r_0(1),
\]
where $r_0$ is a fixed reference path (linear in the hyperbolic case, parabolic self-similar in the parabolic case, or a hybrid in the hyperbolic-parabolic case), $\varphi$ is a perturbation in a suitable function space, and $x$ is the initial configuration. Specifically, we work with the space
\begin{displaymath}\label{eq:space}
 \mathcal D=
    \left\{ \varphi \in H^1_{\mathrm{loc}}([1, +\infty), \mathcal{X}) : \varphi(1) = 0,\ \int_1^{+\infty} \|\dot{\varphi}(t)\|_\mathcal{M}^2\, \mathrm{d}t < +\infty \right\},
\end{displaymath}
endowed with the norm
\[
    \|\varphi\|_{\mathcal{D}} = \left( \int_1^{+\infty} \|\dot{\varphi}(t)\|_\mathcal{M}^2\, \mathrm{d}t \right)^{1/2}.
\]

This is the space in which the minimization of the renormalized Lagrangian action takes place. {From \cite{BDFT} (where $\mathcal D$ is denoted by $\mathcal{D}_0^{1,2}(1, +\infty)$),} let us recall two fundamental inequalities in the space $\mathcal D$.

\begin{prop}[Cfr. \cite{BDFT}]
The space $\mathcal D$ is a Hilbert space, and $C_c^\infty(1, +\infty)$ is dense in it.
\end{prop}

\begin{prop}[\textit{Hardy inequality}, Cf. \cite{BDFT}]\label{dis_hardy}
For every $\varphi \in \mathcal{D}$, the following inequalities hold:
\begin{displaymath}
    \int_1^{+\infty} \frac{\|\varphi(t)\|_\mathcal{M}^2}{t^2}\, \mathrm{d}t \leq 4 \int_1^{+\infty} \|\dot{\varphi}(t)\|_\mathcal{M}^2\, \mathrm{d}t,
\end{displaymath}
\begin{equation}\label{dis_space_D012}
    \sup_{t \in [1, +\infty)} \frac{\|\varphi(t)\|_\mathcal{M}^2}{t - 1} \leq \int_1^{+\infty} \|\dot{\varphi}(t)\|_\mathcal{M}^2\, \mathrm{d}t.
\end{equation}
\end{prop}

In order to cope with the nonintegrability of the Lagrangian of an expansive motion on the half-line $[1, +\infty)$, we introduce the following renormalization:

\begin{defn}[Renormalized Lagrangian Action]\label{def:ren_action}
Given an initial configuration $x \in \mathcal{X}$ and a reference path $r_0$, we define the \emph{Renormalized Lagrangian Action} as the functional $\mathcal{A}_x: \mathcal{D} \to \mathbb R{\cup\{+\infty\}}$ given by:
\begin{displaymath}\label{def:renormalized_action}
    \mathcal{A}_x(\varphi) = \int_1^{+\infty} \left[ \frac{1}{2} \|\dot{\varphi}(t)\|_\mathcal{M}^2 + U(\varphi(t) + r_0(t) + x - r_0(1)) - U(r_0(t)) - \langle \mathcal{M} \ddot{r}_0(t), \varphi(t) \rangle \right]\, \mathrm{d}t.
\end{displaymath}
\end{defn}

By applying the direct method in the calculus of variations to this functional, one obtains the existence of minimizers. The conclusion then follows from the following result, which relies on Marchal's Theorem on the absence of collisions for minimal trajectories:

\begin{thm}[Renormalized Action Principle, Polimeni and Terracini, 2024 \cite{PolimeniTerracini}]\label{th:ren_act_pr}
Let $x \in \mathcal{X}$ be an initial configuration, and let $r_0$ be a reference path. If $\varphi^{\min} \in \mathcal D$ is a minimizer of the renormalized Lagrangian action, then the corresponding expansive motion
\[
    \gamma(t) = r_0(t) + \varphi^{\min}(t) + x - r_0(1)
\]
is a free-time minimizer of the Lagrangian action and, in particular, solves Newton's equations \eqref{eq_newton} for all $t \in (1, +\infty)$ (or for all $t \in [1, +\infty)$ if $x \in \Omega$).
\end{thm}

One advantage of this variational framework is that it allows one to define a value function directly dependent on the initial configuration. As noted in \cite{PolimeniTerracini} with a heuristic argument, a linear correction of this value function provides a viscosity solution to the Hamilton-Jacobi equation. In this paper, we rigorously prove this assertion and investigate the size and structure of the set of singular points of the value function.

\subsection{Viscosity solutions of the Hamilton-Jacobi equation}

Given the Hamiltonian \eqref{eq:hamiltonian}, we study the supercritical Hamilton-Jacobi equation
\begin{equation}\label{eq:hj}
    H(x,  \nabla v(x)) = h, \quad x \in \Omega,
\end{equation}
which, in the context of the $N$-body problem, takes the form
\[
H(x,\nabla v(x)) = \sum_{i=1}^{N}\frac{1}{2 m_i} \left| \frac{\partial v}{\partial r_i} \right|^2 - \sum_{i<j} \frac{m_i m_j}{|r_i - r_j|} = h.
\]

In general, classical solutions of \eqref{eq:hj} may fail to exist due to the lack of compactness and regularity, or the development of singularities, especially in the presence of singular potentials such as in the $N$-body problem (cfr \cite{FatMad2007,Maderna2012}). For this reason, it is natural to consider a weaker notion of solution -- namely, that of viscosity solutions -- which allows for meaningful interpretation of solutions even when the function $v$ is not differentiable.

We now recall the definition of viscosity solutions used in this paper.

\begin{defn}\label{def:viscosity}
     For a given $x\in\mathcal{X}$, define the sets
    \begin{displaymath}\begin{split}
        &D^- v(x) = \left\{ p \in \mathbb{R}^n : \liminf_{y\rightarrow x}\frac{v(y)-v(x) - \langle p, y-x \rangle_\mathcal{M}}{\|y-x\|_\mathcal{M}} \geq 0 \right\},\\
        &D^+ v(x) = \left\{ p \in \mathbb{R}^n : \limsup_{y\rightarrow x}\frac{v(y)-v(x) - \langle p, y-x \rangle_\mathcal{M}}{\|y-x\|_\mathcal{M}} \leq 0 \right\},
    \end{split}\end{displaymath}
    which are referred to as the Fréchet superdifferential and subdifferential of $v$ at $x$, respectively.

    A function $v \in C(\mathcal{X})$ is called a \emph{viscosity supersolution} of \eqref{eq:hj} if, for every $x \in \mathcal{X}$, it holds that
    \[
    H(x, \mathcal M p) \geq h, \quad \forall p \in D^- v(x).
    \]
    Similarly, $v$ is a \emph{viscosity subsolution} of \eqref{eq:hj} if, for every $x \in \mathcal{X}$, it holds that
    \[
    H(x, \mathcal M p) \leq h, \quad \forall p \in D^+ v(x).
    \]
    A function $v$ is a \emph{viscosity solution} of \eqref{eq:hj} if it is both a viscosity supersolution and a viscosity subsolution.
\end{defn}

\subsection{Regularity of the value function in the finite horizon problem}
Although in this paper we will mainly deal with stationary solutions of the Hamilton-Jacobi equation, it is worth mentioning a few salient facts concerning the time-dependent case in a finite horizon setting. In optimal control and Hamilton-Jacobi theory, the value function plays a central role as it encodes the minimal cost to reach a given state. Understanding its regularity properties is crucial for both theoretical and numerical purposes. In particular, non-smoothness of the value function often corresponds to the presence of multiple minimizers or singularities in the dynamics. To address this, we rely on the foundational results of Cannarsa and Sinestrari.

In \cite{MR2041617}, within a more general framework, Cannarsa and Sinestrari established several fundamental results on the regularity of viscosity solutions to Hamilton-Jacobi equations. Fix $t > 0$, and let $\mathrm{AC}([0,t],\,  \Omega)$ denote the set of absolutely continuous arcs $\xi:[0,t] \to  \Omega$. For a fixed $x \in \Omega$, define the functional
\begin{equation}\label{eq:value_function_CS}
    \mathcal{J}_t (\xi)=\int_{0}^{t}L(\xi(s),\dot \xi(s))\,\ud s + u_0(\xi(0)), \quad \xi \in \mathcal{B},
\end{equation}
where the space $\mathcal B$ depends on $t$ and $x$, as follows:
\[
\mathcal{B}=\mathcal B_{t,x} = \{\xi \in \mathrm{AC}([0,t], \Omega):\, \xi(t)=x\}.
\]
The associated value function is given by
\[
u(t,x)=\min_{\xi \in \mathcal{B}}\mathcal{J}_t(\xi).
\]

In \cite[Theorem 6.4.3]{MR2041617}, it is proved that if $u_0$ is continuous on $\Omega$, then $u(t,\cdot)$ is locally semiconcave with linear modulus on $\Omega$ for every $t > 0$. Moreover, given $T > 0$, it is shown in \cite[Theorem 6.4.5]{MR2041617} that $u$ is a viscosity solution of the Hamilton-Jacobi Cauchy problem
\begin{displaymath}
    \begin{cases}
\partial_t u(t,x)+H(x,\nabla u(t,x))=0, & \text{in } [0,T]\times \Omega, \\
u(0,x)=u_0(x), & \text{in } \Omega.
    \end{cases}
\end{displaymath}

\begin{rem}
    In \cite{MR2041617}, the final endpoint of the trajectory is fixed. In contrast, in the present paper, we consider a reversed-time formulation, so that the configuration is fixed at the initial time. This reversal corresponds to the study of the backward Hamilton-Jacobi equation. Although the variational structure remains the same, the time orientation affects the interpretation of minimizers and the evolution of singularities.
\end{rem}

\begin{rem}
    Another important difference lies in the structure of the Lagrangian. In our setting, we consider the classical $N$-body problem, where the Lagrangian is of the form
    \[
    L(x,v) = \sum_{i=1}^N \frac{1}{2} m_i |v_i|^2 + \sum_{i<j} \frac{m_i m_j}{|r_i - r_j|},
    \]
    which is quadratic in the velocities but includes singular potentials due to mutual interactions. While \cite{MR2041617} allows for general smooth Lagrangians, we work within a singular setting where collision configurations may occur. However, in the domain $\Omega \subset \mathcal{X}$ we exclude collisions, so that $L$ remains smooth and the results from \cite{MR2041617} can be adapted.
\end{rem}

Cannarsa and Sinestrari aimed to investigate the singularities of the value function, which correspond to points where the minimizing arc of the associated variational problem is not unique. Furthermore, they showed that the value function inherits the same regularity as the data of the problem outside the closure of its singular set.

We now recall some relevant definitions, following \cite{MR2041617}, to study the regularity of the value function in the (compact) {\em finite-horizon} case formulated in \eqref{eq:value_function_CS}. To this end, consider $u_0$ in \eqref{eq:value_function_CS} of class $C^{R+1}$, for some $R\ge 1$. For any $z \in \mathbb R^n$, we denote by $(\xi(\cdot,z),\eta(\cdot,z))$ the solution of 
\begin{equation}
\label{e:CS_system}
\begin{cases}
\dot \xi(s)=\partial_p H(\xi(s), \eta(s))
\\
\dot \eta (s)=-\partial_x H(\xi(s),\eta(s)) \ \ s\ge 0,
\end{cases}
\end{equation}
with initial data $(\xi(0),\eta(0))=(z, \nabla u_0(z))$. It turns out that $\xi$, $\eta$ and their time derivatives are of class $C^R$ in both arguments. Moreover, it is proved that an arc is an extremal for problem \eqref{eq:value_function_CS} if and only if it is of the form $\xi=\xi(\cdot, z)$ for some $z \in \mathbb R^n$. Denote by $(\xi_z,\eta_z)$ the derivative with respect to $z$ of $(\xi,\eta)$. By differentiating  \eqref{e:CS_system}, one can see that $(\xi_z, \eta_z)$ satisfies the linearized system of  \eqref{e:CS_system}, with initial conditions $(\xi_z(0), \eta_z(0))=(I_n, \nabla^2 u_0(z))$ (see \cite[(6.15)--(6.18)]{MR2041617} and their relative comments for more insights).

\begin{defn}
\label{def:CS_irregular_points}
A point $(t,x)\in[0,T]\times\mathbb{R}^n$ is called \emph{regular} if there exists a unique minimizer of \eqref{eq:value_function_CS}. All other points are called \emph{irregular}.
\end{defn}

\begin{defn}
\label{def:CS_conjugate_points}
A point $(t,x)\in[0,T]\times\mathbb{R}^n$ is called \emph{conjugate} if there exists $z\in\mathbb{R}^n$ such that $\xi(t,z)=x$, the arc $\xi(\cdot,z)$ is a minimizer of \eqref{eq:value_function_CS}, and $\det(\xi_z(t,z)) = 0$.
\end{defn}

We denote by $\Sigma$ the set of irregular points and by $\Gamma$ the set of conjugate points.

Cannarsa and Sinestrari provided a detailed description of the topological and measure-theoretic properties of the singular set $\Sigma$. In particular, they showed that its closure $\bar{\Sigma}$ is also countably $\mathcal{H}^n$-rectifiable, meaning it has the same rectifiability as $\Sigma$. Additionally, they obtained the sharper estimate
\[
\mathcal{H}^{n-1+\frac{2}{k-1}}(\Gamma\setminus\Sigma) = 0,
\]
where $k \geq 3$ denotes the differentiability class of the data. This implies that the value function is as smooth as the problem data outside a closed rectifiable set of codimension one.

We now state their main result, adapted to our setting.

\begin{thm}[Cannarsa and Sinestrari, 2004 \cite{MR2041617}]\label{thm:CS_6.6.4}
    Let $n > 1$, and suppose that $L \in C^{R+1}([0,T]\times\mathbb{R}^n \times\mathbb{R}^n)$ is a positive Lagrangian, quadratic in the velocities, as in \eqref{eq:lagrangian}, and that $u_0 \in C^{R+1}(\mathbb{R}^n)$ for some $R \geq 1$. Then the set of conjugate points $\Gamma$ is countably $\mathcal{H}^n$-rectifiable, and
    \[
    \mathcal{H}^{n-1+\frac{2}{R}}(\Gamma\setminus\Sigma)=0.
    \]
    In particular, if $L \in C^\infty([0,T]\times\mathbb{R}^n\times\mathbb{R}^n)$ and $u_0 \in C^\infty(\mathbb{R}^n)$, then the Hausdorff dimension satisfies $\dim_{\mathcal{H}}(\Gamma\setminus\Sigma)\leq n-1$.
\end{thm}

Their proofs rely on a version of Sard's theorem, stated as follows (\cite[Theorem 3.4.3]{Federer}).

\begin{thm}\label{thm:CS_Sard}
    Let $F:\mathbb{R}^N\rightarrow\mathbb{R}^M$ be a $C^R$ map for some $R \geq 1$. For each $k \in \{0,1,\dots,N-1\}$, define
    \[
    A_k = \{x\in\mathbb{R}^N\ |\ \mathrm{rk}(DF(x))\leq k\},
    \]
    where $\mathrm{rk}(DF)$ denotes the rank of the Jacobian matrix $DF$. Then
    \[
    \mathcal{H}^{k + \frac{N-k}{R}}(F(A_k)) = 0.
    \]
\end{thm}

In our setting, we aim to apply the same Sard-type theorem, carefully accounting for the differences between our problem and that considered by Cannarsa and Sinestrari.

\subsection{Busemann Functions, Horofunctions, and the Gromov Boundary in Hamilton-Jacobi Theory}

As already highlighted in \cite{MadernaVenturelli_HyperbolicMotions}, hyperbolic solutions fit naturally into the framework of the Gromov boundary and other compactifications of the configuration space of the $N$-body problem, and for this reason have attracted a growing attention in the recent literature (see e.g. \cite{BurgosMaderna_GeodesicRays,MR4121133,zbMATH06723100,MR4331257,MR3868425,arXiv:2408.00877, zbMATH07765044,MR4600219,PolimeniTerracini} and references therein). In particular, consider the noncollision configuration space $\Omega$ endowed with the Jacobi-Maupertuis' metric associated with the $N$-body problem (a position-dependent rescaling of the Euclidean metric). For a given energy level $h>0$, this metric $j_h$ can be written in coordinates as n
\[ j_{ik}(x) \;=\; 2\,\big(h + U(x)\big)\,\delta_{ik}\,, \] 
which is proportional to the standard Euclidean (mass-weighted) metric $g_m$. Free-time action minimizers are then precisely the minimal geodesics of $j_h$, and in particular our expanding solutions correspond to geodesic rays (i.e. geodesics that extend infinitely in one direction while remaining minimizing at every finite time-section). Each such geodesic ray can be thought of as converging to a point at infinity, and the collection of all these \emph{endpoints at infinity} is the \emph{ideal boundary} (or \emph{Gromov boundary}) of $(\Omega, j_h)$. To describe these boundary points more concretely, we recall the notion of \emph{horofunction}, which is a convenient tool in metric geometry for capturing directions at infinity.

A \emph{horofunction} is a type of function that arises in the study of the asymptotic geometry of metric spaces. It encodes how distances grow at infinity and plays a key role in defining boundaries at infinity (such as the Gromov boundary). More precisely, let $(X,d)$ be a proper metric space (meaning all closed balls are compact), and fix a base point $x_0 \in X$. For each point $x \in X$, consider the function 
\[ 
f_x(y) \;:=\; d(y,x)\;-\;d(x_0,x)\,. 
\] 
This $f_x(y)$ measures the distance from $y$ to $x$ relative to the base point $x_0$ (note that $f_x(x_0)=0$ by definition). As $x$ \emph{goes to infinity} in $X$ (i.e. $x$ diverges without bound), the functions $f_x$ may sub-converge (locally uniformly) to a limiting function. Any such limit 
\[ 
b(y) \;=\; \lim_{n\to\infty} f_{x_n}(y) 
\] 
(for some sequence $x_n \to \infty$ in $X$) is by definition a horofunction on $X$. The space of all horofunction limits (modulo additive constants) is called the \emph{horofunction boundary} of $X$. In many cases, these horofunction boundaries coincide with the more abstract Gromov boundaries defined via geodesic rays. Intuitively, a horofunction $b$ behaves like an \emph{affine distance} from a point at infinity, and its sublevel sets are analogous to horospheres (level sets of distance from infinity).

{\em Horofunctions in Hamilton-Jacobi theory.} In the context of the Jacobi-Maupertuis' metric $j_h = 2(h+U(x))\,g_m$ on the configuration space $\Omega \subset E^N$ (with $g_m$ the mass inner product), horofunctions arise naturally as limits of rescaled action functionals. Let $\phi_h(x,y)$  the Jacobi-Maupertuis' distance between two configurations $x,y \in \Omega$ at energy $h>0$, which  is the minimal action in free time for the Lagrangian $L+h$.

To describe the construction of horofunctions associated with hyperbolic motions, let us fix  a non-colliding configuration \( a \in \Omega \), and a reference point \( x_0 \in \Omega \). For any sequence \( (x_n) \subset \Omega \) such that \( x_n = \lambda_n a + o(\lambda_n) \) with \( \lambda_n \to +\infty \), consider the normalized action potentials
\[
b_{x_n}(y) := \phi_h(y, x_n) - \phi_h(x_0, x_n), \qquad y \in \Omega.
\]
Any locally uniform limit of a subsequence of \( (b_{x_n}) \) is called a horofunction directed by \( a \). Geometrically, $b_{x_n}(y)$ is the extra action (or distance) required to go from $y$ to $x_n=\lambda_n a $ beyond that required from the reference point $x_0$ to $x_n$. Such limits are dominated by \( L + h \), and, as shown in \cite{MadernaVenturelli_HyperbolicMotions}, in the hyperbolic regime they arise as fixed points of the Lax-Oleinik semigroup. In particular, they are global viscosity solutions of the stationary Hamilton-Jacobi equation \( H(x, \nabla u) = h \). Moreover, for every \( x \in \Omega \), any such solution admits calibrating curves through \( x \), and  each calibrating curve is a hyperbolic action-minimizing trajectory (equivalently, a geodesic ray for the Jacobi-Maupertuis metric) with asymptotic shape \( a \).

It is important, however, not to identify \emph{a priori} a single function \( b_a \) with the direction \( a \). For a fixed \( a \in \Omega \), different sequences \( x_n = \lambda_n a + o(\lambda_n) \) may produce different locally uniform limits. Even when one considers a geodesic ray \( \gamma \) and defines the associated Busemann function
\[
b_\gamma(y) := \lim_{t \to +\infty} \left( \phi_h(y, \gamma(t)) - \phi_h(x_0, \gamma(t)) \right),
\]
the resulting function depends on the reference point \( x_0 \), up to an additive constant. Therefore, without further hypotheses, one must speak of the \emph{set} of horofunctions directed by \( a \), rather than a unique representative.

The elimination of this indeterminacy is one of the main results of \cite{MadernaVenturelli_2026}: if \( a \in \Omega \) is non-colliding and \( h = \|a\|^2 / 2 > 0 \), then every directed horofunction is a hyperbolic viscosity solution whose calibrating curves all share the same limit shape \( a \), and any two such solutions differ only by an additive constant. Upon fixing a normalization (e.g., \( u(x_0) = 0 \)), one thus obtains a uniquely defined horofunction \( b_a \), also interpretable as a unique Busemann function. The argument in \cite{MadernaVenturelli_2026}, see also \cite{PercinoSanchez2014}, relies on identifying a cone region in which the geodesic ray is unique and all solutions are differentiable, then propagating the equality via comparison along calibrated curves.

{\em Geometric and dynamical meaning.} From a dynamical perspective, \emph{in the framework of hyperbolic solutions}, horofunctions (and in particular the functions $b_a$ defined above) have several important interpretations:
\begin{itemize}
\item They provide canonical global viscosity solutions to the stationary Hamilton-Jacobi equation (selection principles often single out these specific solutions).
\item They encode the asymptotic behavior of action-minimizing trajectories or geodesic rays, effectively capturing the \emph{direction} in which the orbit proceeds as $t \to +\infty$.
\item They coincide with the classical \emph{Busemann functions} when one focuses on a particular geodesic ray. For instance, if $\gamma: [0,+\infty) \to \Omega$ is a geodesic ray (e.g. the trajectory of an hyperbolic solution), one defines 
\[ 
b_\gamma(y) \;=\; \lim_{t\to+\infty}\Big(\phi_h\big(y,\gamma(t)\big)\;-\;\phi_h\big(x_0,\gamma(t)\big)\Big)\,. 
\] 
This $b_\gamma$ is exactly the Busemann function associated with the ray $\gamma$, and it equals (up to an additive constant) the directed horofunction $b_a$ where $a$ is the asymptotic configuration (limit shape) of $\gamma(t)$ as $t\to\infty$.
\end{itemize}
Thus, in the hyperbolic case, horofunctions can be viewed as describing \emph{directions at infinity} in the configuration space, characterizing how the $N$-body system behaves asymptotically as time $t \to \infty$. In particular, each hyperbolic solution (geodesic ray) determines a unique horofunction (up to constant), and different asymptotic configurations $a$ give rise to different points on the boundary at infinity of $(\Omega,j_h)$.

 It is worth noticing that the uniqueness result in \cite{MadernaVenturelli_2026} ensures that, in the hyperbolic case, our value function $v$, associated with the renormalized action minimal value,  is a Busemann function. Moreover, it may be interesting that the linear correction in \eqref{eq:def_v} is itself the Busemann function of the free particle.

Up to this point, our focus has been on the hyperbolic case. We now turn to the parabolic and mixed hyperbolic-parabolic settings, where the notions of horofunction and ideal boundary require further refinement. The parabolic case is particularly subtle, as the Jacobi--Maupertuis' metric degenerates at infinity, and parabolic minimal rays converge only to central configurations. In this context, the ideal boundary must include all minimal central configurations, each represented by parabolic solutions of the form \eqref{eq:parabolic}.

As established in \cite{BPT2026}, the second term in the asymptotic expansion \eqref{eq:parabolic} can be improved to reach precisely the order $1/3$. Thus, two distinct parabolic solutions sharing the same limit shape remain at bounded distance (with respect to the Jacobi-Maupertuis' metric) from each other at infinity.

The mixed hyperbolic-parabolic case is even more complex. Here, the associated geodesic rays exhibit a two-term expansion at infinity as described in \eqref{eq:hyperbolic_parabolic}. Moreover, in the vicinity of such expansive solutions, the Jacobi-Maupertuis' metric becomes equivalent to the Euclidean one. As a result, different choices of the minimal $a$-clustered central configuration $b_m$ lead to diverging geodesic rays -- although even the same choice of $b_m$ may result in unbounded asymptotic behaviors. Therefore, any meaningful definition of the ideal boundary must be fine-tuned to distinguish among these cases.


\section{Preliminary estimates}

In this section, we establish some preliminary estimates that will be used later in the proof. We start by defining the space, for a given $T\in(1,+\infty]$,
\[
\mathcal{D}^T= \{\varphi\in H^{1,2}([1,T])\ :\ \varphi(1)=0,\ \int_{1}^{T}\|\dot\varphi(t)\|^2_\mathcal{M}\ \ud t<+\infty\},
\]
and the norm
\[
    \|\varphi\|_{\mathcal{D}^T} = \bigg( \int_{1}^{T} \|\dot{\varphi}(t)\|_\mathcal{M}^2\ \ud t \bigg)^{1/2},\quad \forall \varphi\in H^1([1,T]),\ \varphi(1)=0,
\]
which satisfies
\begin{displaymath}\label{dis:norm_T}
    \|\varphi(t)\|_\mathcal{M} \leq \|\varphi\|_{\mathcal{D}^T} \sqrt{t},\quad \forall \varphi\in \mathcal{D}^T.
\end{displaymath}

The results of Proposition \ref{dis_hardy} can be extended as follows, showing that $\mathcal{D}^T$ is compactly embedded in the space $L^2(1,T)$ equipped with proper weights. 
To this end, for every $T\in (1,+\infty]$ and $\varepsilon\ge 0$, we define $L^2(1,T;\ud t/t^{2+\varepsilon})$ as the space of functions $\varphi$ such that
\[
\int_1^T \frac{\|\varphi(t)\|_{\mathcal{M}}^2}{t^{2+\varepsilon}}\, \ud t < +\infty.
\]

\begin{prop}
\label{prop:compact_emb}
    Let $T\in(1,+\infty]$. Then, for all $\varepsilon \ge0$ and for all $\varphi\in\mathcal{D}^T$, the following Hardy-type inequality holds
    \begin{equation}\label{hardy_compact}
        \int_{1}^{T} \frac{\|\varphi(t)\|_\mathcal{M}^2}{t^{2+\varepsilon}}\ \ud t \leq \frac{4}{(1+\varepsilon)^2} \int_{1}^{T} \|\dot{\varphi}(t)\|_\mathcal{M}^2\ \ud t,
    \end{equation}
    that is, the space $\mathcal{D}^T$ is continuously embedded in the space $L^2(1,T;\ud t/t^{2+\varepsilon})$. Besides, $\mathcal{D}^T$ is compactly embedded in the space $L^2(1,T; \ud t/t^{2+\varepsilon})$ for all $\varepsilon>0$.
\end{prop}
\begin{proof}
    First, we prove that the embedding is continuous. For all fixed $T\in(1,+\infty)$ and $\varepsilon\ge0$, this follows from
    \[
    \begin{split}
        \int_{1}^{T} \frac{\|\varphi(t)\|_\mathcal{M}^2}{t^{2+\varepsilon}}\ \ud t & = -\int_{1}^{T} \frac{\ud}{\ud t}\frac{1}{t^{1+\varepsilon}}\|\varphi(t)\|_\mathcal{M}^2\ \ud t \\
        & = -\frac{1}{1+\varepsilon}\frac{\|\varphi(T)\|_\mathcal{M}^2}{T^{1+\varepsilon}} + \frac{2}{1+\varepsilon}\int_{1}^{T} \frac{\langle \varphi(t),\dot{\varphi}(t)\rangle_\mathcal{M}}{t^{1+\varepsilon}}\ \ud t\\
        &\leq \frac{2}{1+\varepsilon}\int_{1}^{T} \frac{\langle \varphi(t),\dot{\varphi}(t)\rangle_\mathcal{M}}{t^{1+\varepsilon}}\ \ud t\\
        &\leq \frac{2}{1+\varepsilon}\bigg( \int_{1}^{T} \frac{\|\varphi(t)\|_\mathcal{M}^2}{t^{2+2\varepsilon}}\ \ud t \bigg)^{1/2} \bigg( \int_{1}^{T} \|\dot{\varphi}(t)\|_\mathcal{M}^2\ \ud t \bigg)^{1/2}\\
        &\leq \frac{2}{1+\varepsilon}\bigg( \int_{1}^{T} \frac{\|\varphi(t)\|_\mathcal{M}^2}{t^{2+\varepsilon}}\ \ud t \bigg)^{1/2} \bigg( \int_{1}^{T} \|\dot{\varphi}(t)\|_\mathcal{M}^2\ \ud t \bigg)^{1/2},
    \end{split}
    \]
    which gives the Hardy-type inequality \eqref{hardy_compact}. For $T=+\infty$, the thesis follows from the fact that the space $C_c^\infty(1,+\infty)$ is dense in $\mathcal{D}$.

    Now, let $\varepsilon>0$. For $T<+\infty$, the compactness follows from Rellich-Kondrachov Theorem. For $T=+\infty$, let $(\varphi_n)_n$ be a bounded sequence in $\mathcal{D}$, i.e. there exists $C>0$ such that $\|\varphi_n\|_\mathcal{D}\le C$, for every $n \in \mathbb N$. Then, there exists a subsequence $(\varphi_{n_k})_k$ converging pointwise to a function $\bar\varphi$ on $[1,+\infty)$. Since, by \eqref{dis_space_D012},
    \[
    \frac{\|\varphi_{n_k}(t) - \bar\varphi(t)\|_\mathcal{M}^2}{t^{2+\varepsilon}} \leq \frac{\|\varphi_{n_k}-\bar\varphi\|_\mathcal{D}^2}{t^{1+\varepsilon}} \le \frac{C}{t^{1+\varepsilon}}\in L^1(1,+\infty),
    \]
    we can use the Dominated Convergence Theorem to say that
    \[
    \lim_{k\rightarrow+\infty}\int_{1}^{+\infty} \frac{\|\varphi_{n_k}(t) - \bar\varphi(t)\|_\mathcal{M}^2}{t^{2+\varepsilon}}\ \ud t = 0.
    \]
    This concludes the proof.
\end{proof}

\begin{rem}
    Clearly, for $T=+\infty$ and $\varepsilon=0$, we obtain the same results of Proposition \ref{dis_hardy}.
\end{rem}

\subsection{Uniform coercivity estimates}\label{sec:coercivity_estimates}

The aim of this subsection is to prove that the renormalized Lagrangian action satisfies uniform coercivity estimates with respect to $T$ and $x$ when the initial configuration belongs to a compact set. This will be crucial in the proof of the main theorems and will be used several times in the following arguments.

\begin{notation}
Fix $T>1$. 
    For a function $\varphi\in \mathcal{D}^T$, we will write
    \[
    \mathcal{A}_{x}^T(\varphi) = \int_{1}^{T} \frac{1}{2}\|\dot{\varphi}(t)\|_\mathcal{M}^2 + U(\varphi(t) + r_0(t) + x - r_0(1)) - U(r_0(t)) - \langle \mathcal{M}\ddot{r}_0(t),\varphi(t)\rangle\ \ud t.
    \]
    To keep notations compact, in the case $T=+\infty$, with $\mathcal{A}_x^{+\infty}$ we mean the functional $\mathcal{A}_x$, defined on $\mathcal{D}$.
\end{notation}

\begin{lem}\label{lem:coercivity_estimates}
    Given a compact set $\mathcal{K}\subset\mathcal{X}$, fix $x\in\mathcal{K}$ and $T\in (1,+\infty]$. Then, there exist positive constants $A,B,C\in\R$, which do not depend on $x$ and $T$, such that for all $\varphi$ belonging to a bounded subset of $\mathcal{D}^T$, it holds  
    \begin{displaymath}\label{eq:unif_coerc_est}
    \mathcal{A}_{x}^T(\varphi) \geq A\|\varphi\|_{\mathcal{D}^T}^2 - B \|\varphi\|_{\mathcal{D}^T} + C.
    \end{displaymath}
\end{lem}

\begin{proof}
For fixed $x\in\mathcal{X}$ and $T=+\infty$, coercivity estimates have been proved in \cite{PolimeniTerracini}. 

Now, we fix $x\in \mathcal K$, with $\mathcal K$ a compact subset of $\mathcal{X}$, and $T\in(1,+\infty]$. Let $\bar k\in\R$ be such that $\|x\|_\mathcal{M}\leq \bar k$ for all $x\in \mathcal K$. The goal is to prove that the renormalized Lagrangian action satisfies the coercivity estimates uniformly with respect to both $x$ and $T$.

As usual, we study the three cases separately.

\smallskip
{\bf Hyperbolic case.}
We write the action as
\[
\mathcal{A}_{x}^T(\varphi) = \sum_{i<j}m_im_j\mathcal{A}_{x}^T(\varphi)_{ij},
\]
where 
\[
\mathcal{A}_{x}^T(\varphi)_{ij} = \int_{1}^{T}\frac{1}{2M}|\dot\varphi_{ij}(t)|^2 + 
\frac1{|\varphi_{ij}(t)+a_{ij}t+x_{ij}-a_{ij}|}-\frac1{|a_{ij}t|}\,\ud t,
\]
$M=\sum_{i=1,\dots,N}m_i$, and 
\[
\varphi_{ij}=\varphi_i-\varphi_j, \ a_{ij}=a_i-a_j, \ x_{ij}=x_i-x_j.
\]
We then show uniform coercivity estimates of each term $\mathcal{A}_{x}^T(\varphi)_{ij}$.

\smallskip
We have
\begin{displaymath}
\begin{split} 
    \int_{1}^{T} \frac1{|\varphi_{ij}(t)+a_{ij}t+x_{ij}-a_{ij}|}-\frac1{|a_{ij}t|}\,\ud t \ge \int_{1}^{T} \frac{1}{|\varphi_{ij}(t)| +|a_{ij}| t +  |x_{ij} - a_{ij}|} - \frac{1}{|a_{ij}|t}\, \mathrm{d}t.
    \end{split} 
    \end{displaymath} 

    Then, by the convexity of $\frac1t$, we can use Taylor's expansion and then Hardy inequality to obtain
    \[
    \begin{split}
    &\int_{1}^{T} \frac{1}{|\varphi_{ij}(t)| +  |a_{ij}| t +|x_{ij} - a_{ij}|} - \frac{1}{|a_{ij}|t}\, \mathrm{d}t \\
    &\ge \int_{1}^{T} -\frac{|\varphi_{ij}(t)|}{|a_{ij}| t^2 } - \frac{|x_{ij} - a_{ij}|}{|a_{ij}|t^2}\, \mathrm{d}t\\
    &\ge -\frac{1}{|a_{ij}|}\bigg(\int_{1}^{T} \frac{|\varphi_{ij}(t)|^2}{t^2}\ \ud t\bigg)^{1/2}\bigg(\int_{1}^{T}\frac{1}{t^2}\ \ud t\bigg)^{1/2} -\int_{1}^{T} \frac{|x_{ij} - a_{ij}|}{|a_{ij}|t^2}\ \ud t\\
    &\ge -\frac{2}{|a_{ij}|}\|\varphi_{ij}\|_{\mathcal{D}^T}\bigg(\int_{1}^{+\infty}\frac{1}{t^2}\ \ud t\bigg)^{1/2} - \int_{1}^{+\infty} \frac{|x_{ij} - a_{ij}|}{|a_{ij}|t^2}\ \ud t\\
    & = -\frac{2}{|a_{ij}|}\|\varphi_{ij}\|_{\mathcal{D}^T} + \frac{|x_{ij} - a_{ij}|}{|a_{ij}|}.
    \end{split}
    \]

    We have thus proved that
\[
\mathcal{A}_{x}^T(\varphi)_{ij} \ge \frac{1}{2M}\|\varphi_{ij}\|_{\mathcal{D}^T}^2-\frac{2}{|a_{ij}|}\|\varphi_{ij}\|_{\mathcal{D}^T}+\frac{|x_{ij} - a_{ij}|}{|a_{ij}|},
\]
from which the uniform coercivity with respect to $T$ and $x$ trivially follows.

\smallskip
{\bf Parabolic case.} In this case, it holds
\[
\nabla U(r_0) = \nabla U(\beta b_m t^{2/3})= -\frac{U_{min}}{\beta^2t^{4/3}}\mathcal{M}b_m,
 \]
 with $U_{min}=U(b_m)$.
 
Following the same computations made in \cite[Section 4.1]{PolimeniTerracini}, we get
\[
\begin{split}
\mathcal{A}_{x}^T(\varphi) &= \frac{1}{2}\|\varphi\|^2_{\mathcal{D}^T}+ \int_{1}^{T} U(\varphi(t) + \beta b_m t^{2/3} + \tilde x) - U(\beta b_m t^{2/3}) - \langle \nabla U(\beta b_m t^{2/3}),\varphi(t)\rangle\ \ud t
    \\
    &\ge \frac{1}{2}\|\varphi\|^2_{\mathcal{D}^T}+ U_{min}\int_1^T  \frac{1}{\|\varphi(t)\|_{\mathcal{M}}+\beta t^{2/3}\|b_m\|_\mathcal{M}+\|\tilde x\|_\mathcal{M}}-\frac{1}{\beta t^{2/3}\|b_m\|_\mathcal{M}}+\frac{\langle \mathcal{M}b_m,\varphi(t)\rangle}{\beta^2 t^{4/3}}\,\ud t
    \\
    & \ge \frac{1}{2}\|\varphi\|^2_{\mathcal{D}^T}-U_{min}\int_1^T\frac{\|\varphi(t)\|_\mathcal{M}^2}{2\beta^3\|b_m\|_\mathcal{M}^3t^2}+\frac{\langle \mathcal{M}\varphi(t),\tilde x\rangle}{\beta^3\|b_m\|_\mathcal{M}^3t^2}\ \ud t -U_{min}\int_1^T \frac{\langle\mathcal{M} b_m,\tilde x\rangle \beta t^{2/3}+1/2\|\tilde x\|_\mathcal{M}}{\beta^3\|b_m\|_\mathcal{M}^3t^2}\ \ud t,
    \end{split}
\]
where we set $\tilde x=x-\beta b_m$. 

The following inequalities hold:
\begin{itemize}
    \item \[-U_{min}\int_1^T\frac{\|\varphi(t)\|_\mathcal{M}^2}{2\beta^3\|b_m\|_\mathcal{M}^3t^2}\ \ud t \ge -\frac{8}{18}\|\varphi\|_{\mathcal{D}^T}^2.
    \]
    \item \[\begin{split}-U_{min}\int_1^T\frac{\langle \mathcal{M}\varphi(t),\tilde x\rangle}{\beta^3\|b_m\|_\mathcal{M}^3t^2}\ \ud t&\ge-\frac{2}{9}\int_{1}^{T} \frac{\|\varphi(t)\|_\mathcal{M}\|\tilde x\|_\mathcal{M}}{t^2}\ \ud t\\
    &\ge -\frac{2}{9}\bigg(\int_{1}^{T}\frac{||\varphi(t)\|_\mathcal{M}^2}{t^2}\ \ud t\bigg)^{1/2}\bigg(\int_{1}^{T}\frac{\|\tilde x\|_\mathcal{M}^2}{t^2}\ \ud t\bigg)^{1/2}\\
    &\ge-\frac{4}{9}\|\varphi\|_{\mathcal{D}^T}\|\tilde x\|_\mathcal{M}. \end{split}\]
    \item \[
    \begin{split}
        -U_{min}\int_1^T \frac{\langle\mathcal{M} b_m,\tilde x\rangle \beta t^{2/3}+1/2\|\tilde x\|_\mathcal{M}}{\beta^3\|b_m\|_\mathcal{M}^3t^2}\ \ud t& \ge -\frac{2}{9}\int_{1}^{T}\frac{\|\tilde x\|_\mathcal{M}\beta t^{2/3} + 1/2\|\tilde x\|_\mathcal{M}}{t^2}\ \ud t\\
        & \ge -\frac{2}{9}\int_{1}^{+\infty}\frac{\|\tilde x\|_\mathcal{M}\beta t^{2/3} + 1/2\|\tilde x\|_\mathcal{M}}{t^2}\ \ud t\\
        &= \frac{\|\tilde x\|_\mathcal{M}}{9}(2\beta-1).
    \end{split}
    \]
\end{itemize}

It follows
\[
\mathcal{A}_{x}^T(\varphi) \geq \frac{1}{18}\|\varphi\|_{\mathcal{D}^T}^2 - \frac{4}{9}\|\varphi\|_{\mathcal{D}^T}\|\tilde x\|_\mathcal{M} + \frac{\|\tilde x\|_\mathcal{M}}{9}(2\beta-1),
\]
and hence the proof for the parabolic case is completed.

\smallskip
{\bf Hyperbolic-Parabolic case.} As in \cite[Section 5]{PolimeniTerracini}, considering the cluster partition of the bodies determined by the equivalence relation \eqref{eq:equivalence_relation}, the renormalized Lagrangian action can be written as the sum of two terms as follows: for every $\varphi\in \mathcal{D}^T$, we define
\begin{equation}\label{def:decomposed_lagrangian_action}
\begin{split}
    \mathcal{A}_{x}^T(\varphi) & = \sum_{K\in\mathcal{P}} \mathcal{A}_{K}(\varphi) + \sum_{K_1,K_2\in\mathcal{P},\ K_1\neq K_2} \mathcal{A}_{K_1,K_2}(\varphi)\\
    &  = \sum_{K\in\mathcal{P}}\bigg(\sum_{i,j\in K,\ i<j} m_i m_j \mathcal{A}_{K}^{ij}(\varphi)\bigg) + \frac{1}{2}\sum_{K_1,K_2\in\mathcal{P},\ K_1\neq K_2}\bigg(\sum_{i\in K_1,\ j\in K_2} m_i m_j \mathcal{A}_{K_1,K_2}^{ij}(\varphi)\bigg),
\end{split}
\end{equation}
where
\begin{displaymath}\label{action_partially_hyperbolic_inside}
\begin{split}
    \mathcal{A}_{K}^{ij}(\varphi) = & \int_{1}^{T} \frac{1}{2M} |\dot{\varphi}_{ij}(t)|^2 + \frac{1}{|\varphi_{ij}(t)+\beta_K b^K_{ij}t^{2/3}+\Tilde{x}_{ij}|} - \frac{1}{|\beta_K b^K_{ij}t^{2/3}|} \\
    &+ \frac{2}{9}\frac{\beta_K}{M} \frac{\langle b^K_{ij},\varphi_{ij}(t)\rangle}{t^{4/3}}\ \ud t,
\end{split}
\end{displaymath}
\begin{displaymath}\label{action_partially_hyperbolic_outside}
\begin{split}
    \mathcal{A}_{K_1,K_2}^{ij}(\varphi) = &\int_{1}^{T}\frac{1}{2M} |\dot{\varphi}_{ij}(t)|^2+ \frac{1}{|\varphi_{ij}(t)+a_{ij}t+\beta_{K_{1,2}}b_{ij}^{K_{1,2}}t^{2/3}+\Tilde{x}_{ij}|} - \frac{1}{|a_{ij}t|}\ \ud t,
\end{split}
\end{displaymath}
where $\beta_K >0$, $b^K$ is a minimal central configuration in the cluster $K$, $\beta_{K_{1,2}}b^{K_{1,2}}=(\beta_{K_1}b^{K_1},\beta_{K_2}b^{K_2})$ with $b^{K_i}$ minimal central configuration for the cluster $K_i$, $i=1,2$, and $\beta_{K_i}\in\R$, $i=1,2$ (see \cite[Section 5]{PolimeniTerracini}).

We start with the interaction term $\mathcal{A}_{K_1,K_2}^{ij}$, which behaves like the renormalized Lagrangian action in the hyperbolic case, with the difference that here, in the denominator of the potential, there is also a term of order 2/3. By the triangular inequality, we have
\[
\begin{split}
    \mathcal{A}_{K_1,K_2}^{ij}(\varphi) = &\int_{1}^{T}\frac{1}{2M} |\dot{\varphi}_{ij}(t)|^2+ \frac{1}{|\varphi_{ij}(t)+a_{ij}t+\beta_{K_{1,2}}b_{ij}^{K_{1,2}}t^{2/3}+\Tilde{x}_{ij}|} - \frac{1}{|a_{ij}t|}\ \ud t\\
    &\geq \frac{1}{2M} \|\varphi_{ij}\|_{\mathcal{D}^T}^2 + \int_{1}^{T} \frac{1}{|\varphi_{ij}(t)|+|a_{ij}|t+|\beta_{K_{1,2}}b_{ij}^{K_{1,2}}|t^{2/3}+|\Tilde{x}_{ij}|} - \frac{1}{|a_{ij}|t}\ \ud t.
\end{split}
\]
By the convexity of the function $\frac{1}{t}$, it holds
\[
\begin{split}
    \int_{1}^{T} \frac{1}{|\varphi_{ij}(t)|+|a_{ij}|t+|\beta_{K_{1,2}}b_{ij}^{K_{1,2}}|t^{2/3}+|\Tilde{x}_{ij}|}\ \ud t &\geq -\frac{1}{|a_{ij}|^2}\int_{1}^{T} \frac{|\varphi_{ij}(t)|+|\beta_{K_{1,2}}b_{ij}^{K_{1,2}}|t^{2/3}+|\Tilde{x}_{ij}|}{t^2}\ \ud t \\
    &\ge-\frac{1}{|a_{ij}|^2}(3|\beta_{K_{1,2}}b_{ij}^{K_{1,2}}| + 4\|\varphi_{ij}\|_{\mathcal{D}^T}+|\Tilde{x}_{ij}|).
\end{split}
\]

So, it holds
\[
\mathcal{A}_{K_1,K_2}^{ij}(\varphi) > \frac{1}{2M}\|\varphi_{ij}\|_{\mathcal{D}^T}^2 - \frac{4}{|a_{ij}|^2} \|\varphi_{ij}\|_{\mathcal{D}^T} + -\frac{1}{|a_{ij}|^2}(3|\beta_{K_{1,2}}b_{ij}^{K_{1,2}}| + |\Tilde{x}_{ij}|).
\]

To study the uniform coercivity of the term $\mathcal{A}_{K}^{ij}(\varphi)$, we notice that it has a similar expression to the renormalized action in the parabolic case. Thus, the same computations lead us to the conclusion 
\[
    \mathcal{A}_{K}^{ij}(\varphi) \geq \frac{1}{18 M_K}\|\varphi_{ij}\|_{\mathcal{D}^T}^2 - \frac{4}{9}\|\varphi_{ij}\|_{\mathcal{D}^T}\|\tilde x_{ij}\|_\mathcal{M} + \frac{\|\tilde x_{ij}\|_\mathcal{M}}{9}(2\beta_K-1).
\]
This concludes the proof.

\end{proof}

\subsection{Distance of minimal solutions from collisions}\label{sec:collisionless_minimal_solutions}

As a consequence of the previous uniform coercivity estimates, we will prove that any expansive motion $\gamma(t)$ with initial configuration in $\Omega$, given by the Renormalized Action Principle Theorem \ref{th:ren_act_pr}, is far from collisions for all $t\in[1,+\infty)$, locally uniformly with respect to $x$.

\begin{lem}\label{lemma_collisionless_minimal}
    For any $\bar x\in\Omega$, there is a neighborhood $\mathcal{U}(\bar x)$ such that $\mathcal{U}(\bar x)\cap\Delta=\emptyset$ and there is a real constant $C>0$ such that, for any $x\in \mathcal{U}(\bar x)$, the motions $\gamma(t)$ given by Theorems \ref{thm_hyperbolic}, \ref{thm_parabolic} and \ref{thm_partially_hyperbolic} with initial configuration $x$ satisfy
    \begin{equation}\label{collisionless_minimal_solutions}
        d(\gamma(t),\Delta)\geq C\quad\forall t\in[1,+\infty),\ \text{uniformly with respect to }x.    
    \end{equation}
    
\end{lem}
\begin{proof}
    By Marchal's Principle, we already know that the minimal curve $\gamma(t) = r_0(t) + \varphi^{\bar x}(t) + \bar x + r_0(1)$ is free of collisions for $t\in[1,+\infty)$.

    From uniform coercivity estimates Lemma \ref{lem:coercivity_estimates} it follows that for any constant $R>0$ there is a constant $C_R>0$ such that for all $x\in B_R(0)$, $\|\varphi^{x}\|_\mathcal{D}\leq C_R$. Thus, we know that for any $R>0$ there is $C_R>0$ such that
    \begin{displaymath}
        \begin{split}
            |\gamma_i(t) - \gamma_j(t)| &\geq |(a_i-a_j)t + (\beta b_i - \beta b_j)t^{2/3}| - |\varphi_i^{ x}(t) - \varphi_j^{ x}(t)| - |\Tilde{ x}_i - \Tilde{ x}_j| \\
            & \geq C' t^{2/3} - C_R t^{1/2} - C'' \\
            & \geq 1
        \end{split}
    \end{displaymath}
    for some $t\geq \tau=\tau(R)\geq1$ and for proper $C',C''>0$.
    
    The fact that $|\gamma_i(t) - \gamma_j(t)| \geq1$ for all $i<j$ and for $t\in[\tau,+\infty)$ implies that there is a constant $C>1$ such that $d(\gamma(t),\Delta)\geq C$. Indeed, the collision set $\Delta$ can be represented as a union of hyperplanes:
    \[
    \Delta = \bigcup_{i<j} \Delta_{ij},
    \]
    with $\Delta_{ij}:=\{ x\in\mathcal{X} : x_i = x_j\}$. It is easy to see that $d(\gamma(t),\Delta) = \min_{i<j} d(\gamma(t),\Delta_{ij})$.

    Now, we want to prove that the inequality \eqref{collisionless_minimal_solutions} holds also for $t\in[1,\tau]$. Suppose by contradiction that there is a sequence $({\bar x}_n)_n\subset \mathcal{U}(\bar x)$ such that $d(\gamma_n(t),\Delta)\rightarrow0$ as $n\rightarrow+\infty$ for $t\in[1,\tau]$, where, for any $n\in\N$, the curve $\gamma_n(t)$ is minimal for the value function with initial configuration ${\bar x}_n$. By Ascoli-Arezelà's Theorem, it can be proved that there is a subsequence $(\gamma_{n_k}(t))_k$ converging uniformly over the compact subsets $[1,\tau]$. This implies that the subsequence converges also pointwise over compact sets. 
    
    Let $\gamma_{n_k}(t)\rightarrow\Bar{\gamma}(t)\in\Delta$ as $k\rightarrow+\infty$ for $t\in[1,\tau]$. Then $\Bar{\gamma}(1)\in\Delta$. Since $({\bar x}_n)_n\subset \mathcal{U}(\bar x)$, if we take the limit we obtain $\Bar{\gamma}(1)\in \overline{\mathcal{U}(\bar x)}$, which is a contradiction. 
\end{proof}

\section{Hamilton-Jacobi equations}

The main goal of this section is to prove the following theorem.
\begin{thm}
\label{thm:HJ}
    Fix $x\in\Omega$. The renormalized value function
    \begin{equation}
        \label{eq:def_v}
    v(x) = \min_{\varphi\in\mathcal{D}} \mathcal{A}_x(\varphi) - \langle a,x \rangle_\mathcal{M}
      \end{equation}
    is a viscosity solution of the Hamilton-Jacobi equation
    \begin{equation}
        \label{eq:HJ_v}
       \frac{1}{2}\|\nabla v(x)\|_{\mathcal{M}^{-1}}^2 - U(x) = \frac{\|a\|_\mathcal{M}^2}{2}.
     \end{equation}
\end{thm}

The proof of Theorem \ref{thm:HJ} is in Section \ref{sec:proofHJ}, and needs a number of preliminary results, contained in this and the next subsections. Among these results, it is worth stressing out the following proposition, regarding the uniform convergence over compact sets of $v(T,x)$ (defined through the functionals $\mathcal{A}_{x,[1,T]}$) to $v(x)$, as $T\rightarrow+\infty$.

\begin{prop}
\label{prop:conv_v_T}
For $x \in\Omega$ and $T>1$, let $v(x)$ be defined as in Theorem \ref{thm:HJ}, and let $v(T,x)$ be defined as
\begin{equation}\label{eq:v(T,x)}
v(T,x):=\min_{\varphi\in\mathcal{D}^T}\mathcal{A}_{x}^T(\varphi)-\langle \dot r_0(T),x \rangle_\mathcal{M}.
\end{equation}

Then, $v(T,x)$ and $v(x)$ are continuous in $\Omega$, and, if $(T_n)_n$ is a sequence such that $T_n \to +\infty$, then
    \[
    \lim_{n\to +\infty} v(T_n,x)=v(x) \ \mbox{ uniformly on compact subsets of } \ \Omega.
    \]
\end{prop}

The proof of Proposition \ref{prop:conv_v_T} is a direct corollary of  Lemmas \ref{lem:continuity_w} and \ref{lem:unif_conv}.

\subsection{Finite horizon approximation}\label{sec:finite_horizon_approximation}

Define the functions $w:[1,+\infty)\times \Omega \to \mathbb R$ and $w_\infty:\Omega \to \mathbb R$ as the following value functions 
\begin{equation}\label{eq:min_action}
w(T,x) = \min_{\varphi\in\mathcal D^T}\mathcal{A}_{x}^T(\varphi), \ \mbox{ and } \ w_\infty(x) = \min_{\varphi\in\mathcal D}\mathcal{A}_{x}(\varphi).
\end{equation}

Given $T>1$, we say that $w(T,\cdot)$ is a \emph{finite-horizon} value function associated with $T$. Our aim, in this section, is to prove properties of these finite-horizon value functions, as well as their uniform convergence to $w_\infty$ on compacts subsets of $\Omega$, as long as $T \to +\infty$.

\begin{lem}[Structure lemma]
\label{lem:quad_pot}
For every $T\in (1,+\infty]$, we have
\[
\mathcal{A}_{x}^T(\varphi)=Q_T(\varphi,\varphi)+P_{x,T}(\varphi), \ \mbox{ for } \ \varphi \in \mathcal{D}^T,
\]
where
\begin{itemize}
    \item $Q_T$ is a positive definite quadratic form on $\mathcal{D}^T \times \mathcal{D}^T$;
    \item $P_{x,T}$ is a functional on $\mathcal{D}^T$ such that there exists $V:\Omega \times (1,+\infty) \times \mathcal{X} \to \mathbb R$ for which
    \[
    P_{x,T}(\varphi)=\int_1^T V(x,t,\varphi(t))\ \ud t.
    \]
\end{itemize}
    Moreover, for every compact subset $\mathcal K\subset\Omega$ and $\hat M>0$, there exists $\hat T>1$, $C>0$ and $\beta>1$ such that
    \begin{equation}
    \label{eq:V_bound}
    \left| V(x,t,\varphi(t))\right| \le C t^{-\beta} \ \mbox{ if } \ t \ge \hat T,
    \end{equation}
    for every $x \in \mathcal{K}$ and $\varphi$ such that $\|\varphi(t)\|_\mathcal{M} \le \hat M \sqrt{t}$.
\end{lem}
\begin{proof}
\textbf{Hyperbolic case.} Since $\|\varphi\|_{\mathcal{D}^T}=\big(\int_{1}^{T} \|\dot{\varphi}(t)\|_\mathcal{M}^2\ \ud t\big)^{1/2}$ is a norm, it is trivial that it is a definite positive quadratic form.

We can write down the second part of the action as
\begin{equation}
\label{eq:PQ_hyperbolic}
    \mathcal{A}_{x}^T(\varphi) - \frac{\|\varphi\|_{\mathcal{D}^T}^2}{2} = \int_{1}^{T} \sum_{i<j} m_i m_j \bigg(\frac{1}{|a_{ij} t +\varphi_{ij}(t) + x_{ij} - a_{ij}|} - \frac{1}{|a_{ij}t|} \bigg)\ \ud t.
\end{equation}
Since for all $s\in(0,1)$ 
\[
\begin{split}
     |a_{ij}t + s(\varphi_{ij}(t)+x_{ij} - a_{ij})| & \geq |a_{ij}|t - s(\|\varphi_{ij}\|_{\mathcal{D}^T}\sqrt{t}+|x_{ij} - a_{ij}|) \\
     & > |a_{ij}|t - (\|\varphi_{ij}\|_{\mathcal{D}^T}\sqrt{t}+|x_{ij} - a_{ij}|),
\end{split}
\]
for $x \in \mathcal K$, and $a\neq 0$, there is a $T_1=T_1(\|\varphi\|_{\mathcal{D}^T}, \mathcal K)$ such that, for all $t \ge T_1$,
\begin{equation}
    \label{eq:away_delta}
    |a_{ij}t + s(\varphi_{ij}(t) + x_{ij} - a_{ij})| > |a_{ij}|t - (|x_{ij} - a_{ij}| + \|\varphi\|_{\mathcal{D}^T}\sqrt{t})>0.
\end{equation}

Hence, for $t$ large enough, by the Fundamental Theorem of Calculus, it holds
\[
    \frac{1}{|a_{ij} t +\varphi_{ij}(t) + x_{ij} - a_{ij}|} - \frac{1}{|a_{ij}t|} = \int_{0}^{1} \frac{\ud}{\ud s}\bigg[\frac{1}{|a_{ij}t + s(\varphi_{ij}(t)+x_{ij} - a_{ij})|}\bigg]\ \ud s.
\]

By passing to the absolute value, we get
\[
\begin{split}
   \Bigg| \frac{1}{| a_{ij}t +\varphi_{ij}(t) + x_{ij} - a_{ij}|} - \frac{1}{|a_{ij}t|}\Bigg|
    &= \Bigg| \int_{0}^{1} -\frac{[a_{ij}t + s(\varphi_{ij}(t) + x_{ij} - a_{ij})](\varphi_{ij}(t) + x_{ij}-a_{ij})}{|a_{ij}t + s(\varphi_{ij}(t) + x_{ij} - a_{ij} )|^3}\ \ud s\Bigg|
    \\
    & \le \int_{0}^{1} \frac{|\varphi_{ij}(t)+x_{ij}-a_{ij}|}{|a_{ij}t + s(\varphi_{ij}(t)+x_{ij} -a_{ij} |^2}\ \ud s.
\end{split}
\]

Using the inequality
\[
    \frac{|b+c|^2}{|b|^2 - |c|^2} \geq \frac{1}{3}, \qquad \text{for each }b,c\in\mathbb{R}^d \text{ such that }|b|\geq2|c|,
\]
which can easily be proved by elementary calculus, and knowing that there is a constant $k'\in\R$ large enough such that $|x_{ij} - a_{ij}|+\hat M\sqrt{t} \leq k'\sqrt{t}$, we thus have
\[
\begin{split}
   \int_{0}^{1} \frac{|\varphi_{ij}(t)+x_{ij}-a_{ij}|}{|a_{ij}t + s(\varphi_{ij}(t)+x_{ij} -a_{ij}) |^2}\ \ud s 
    & \leq \int_{0}^{1} 3\frac{\hat M\sqrt{t}+|x_{ij}-a_{ij}|}{|a_{ij}t|^2 - s|M\sqrt{t}+x_{ij} - a_{ij}|^2}\ \ud s \\
    & \leq \int_{0}^{1} \frac{3k'\sqrt{t}}{|a_{ij}|^2 t^2 - sk't}\ \ud s,
\end{split}
\]
where the last term is dominated, at infinity, by a term $\frac{k''}{t^\beta}$ with $k''$ not depending on $T$ and $x$, and $\beta>1$. This proves the statement in the hyperbolic case, with $V(x,t,\varphi)$ given by the integrand function of the right-hand side of \eqref{eq:PQ_hyperbolic}.

\smallskip
\textbf{Parabolic case.} In the parabolic setting, we add and subtract the term $\int_{1}^{T} \frac{1}{2}\langle \nabla^2 U(r_0(t))\varphi(t),\varphi(t)\rangle\ \ud t$ to the action, obtaining
\[
    \begin{split}
   \mathcal{A}_{x}^T(\varphi) = \int_{1}^{T} &\frac{1}{2}\|\dot{\varphi}(t)\|_\mathcal{M}^2 + \frac{1}{2}\langle \nabla^2 U(r_0(t))\varphi(t),\varphi(t)\rangle\\
    &+ U(r_0(t)+\varphi(t)+\Tilde{x}) - U(r_0(t)) - \langle \nabla U(r_0(t)),\varphi(t)\rangle - \frac{1}{2}\langle \nabla^2 U(r_0(t))\varphi(t),\varphi(t)\rangle\ \ud t.
    \end{split}
\]
We prove the statement taking into account
\[
Q_T(\varphi,\varphi) = \int_1^T \frac{1}{2}\|\dot \varphi(t)\|_\mathcal{M}^2+\frac{1}{2}\langle\nabla^2 U(r_0(t))\varphi(t),\varphi(t)\rangle \, \ud t,
\]
and
\[
P_{x,T}(\varphi) = \int_1^T  U(r_0(t)+\varphi(t)+ \Tilde{x}) - U(r_0(t)) - \langle \nabla U(r_0(t)),\varphi(t)\rangle - \frac{1}{2}\langle \nabla^2 U(r_0(t))\varphi(t),\varphi(t)\rangle\ \ud t.
\]

First, exactly as observed for the case $T=+\infty$ in \cite[Section 4.2]{PolimeniTerracini}, one can see that $\varphi \to Q_T(\varphi,\varphi)$ is a norm on $\mathcal{D}^T$ equivalent to $\|\cdot\|_{\mathcal{D}^T}$. Indeed, as observed in \cite[Section 4.2]{PolimeniTerracini}, there exists $C'$ such that the following chain of inequalities is satisfied pointwise 
\[
-\frac{2}{9}\frac{\|\varphi(t)\|_\mathcal{M}^2}{t^2} \le \langle \nabla^2 U(r_0(t))\varphi(t), \varphi(t)\rangle \le C' \frac{\|\varphi(t)\|_\mathcal{M}^2}{t^2}.
\]
Then, by applying Hardy inequality on $[1,T]$, we get
\[
C''\|\varphi\|_{\mathcal{D}^T}^2 \le Q_T(\varphi, \varphi) \le \frac{1}{18}\|\varphi\|^2_{\mathcal{D}^T},
\]
for some positive $C''>0$.

    \smallskip
    Second, we prove the bound \eqref{eq:V_bound} for large values of time as follows. If $t$ is large enough to satisfy \eqref{eq:away_delta}, we have
\[
    \begin{split}
    & U(r_0(t)+\varphi(t)+\Tilde{x}) - U(r_0(t)) - \langle \nabla U(r_0(t)),\varphi(t)\rangle - \frac{1}{2}\langle \nabla^2 U(r_0(t)),\varphi(t),\varphi(t)\rangle \\
    & = \int_{0}^{1}\int_{0}^{1}\int_{0}^{1} \langle \nabla^3 U(r_0(t) + \tau_1\tau_2\tau_3 (\varphi(t)+\Tilde{x}))(\varphi(t)+\Tilde{x}),\varphi(t)+\Tilde{x},\varphi(t)+\Tilde{x}\rangle \tau_1 \tau_2^2\ \ud \tau_1\ \ud \tau_2\ \ud \tau_3\ \ud t.
     \end{split}
\]
Since $\|\varphi(t)\|_\mathcal{M}\le \hat M\sqrt{t}$,  there exists $\Tilde{t}>1$ big enough such that, for all $t\geq \tilde t$,
\begin{displaymath}
    \|r_0(t) + \tau_1\tau_2\tau_3 (\varphi(t)+\Tilde{x})\|_\mathcal{M}>0.
\end{displaymath}
This implies that for all $t\geq \tilde t$ it holds, from \cite[pag. 25]{PolimeniTerracini}, 
\[
    \bigg|\langle \nabla^3 U(r_0(t) + \tau_1\tau_2\tau_3 (\varphi(t)+\Tilde{x}))(\varphi(t)+\Tilde{x}),\varphi(t)+\Tilde{x},\varphi(t)+\Tilde{x}\rangle \bigg| \leq \hat{C}\frac{\|\varphi(t) + \Tilde{x}\|_\mathcal{M}^3}{t^{8/3}} \leq \frac{\hat{C}'(\hat M, \mathcal{K})}{t^{7/6}},
\]
where the constants $\hat C, \hat C'$ do not depend on $T$ and $x\in \mathcal{K}$. 

\smallskip
\textbf{Hyperbolic-parabolic case.}
In the hyperbolic-parabolic setting, the renormalized Lagrangian action is decomposed as in \eqref{def:decomposed_lagrangian_action} and the lemma is proved using the same arguments from the previous two settings.
\end{proof}

Lemma \ref{lem:quad_pot} is instrumental in proving the continuity of the value functions in \eqref{eq:min_action}.

\begin{lem}\label{lem:continuity_w}
For every $T>1$, $w(T,\cdot)$ is continuous in $\Omega$, and $w_\infty$ is continuous in $\Omega$. 
\end{lem}

\begin{proof}

We prove the statement for $w_\infty$. The result for the finite-horizon value functions $w(T,\cdot)$ follows a fortiori.

\textbf{Upper semicontinuity.}
Let $(x_n)_n$ be a sequence converging to $x$. Using the trivial fact that $\mathcal A_x$ is continuous with respect to $x$, it holds
\[
\mathcal{A}_x(\varphi) = \limsup_{x_n\rightarrow x}\mathcal{A}_{x_n}(\varphi) \geq \limsup_{x_n\rightarrow x} w(x_n)
\]
for all $\varphi\in\mathcal{D}$. This plainly gives
\[
w(x)\geq \limsup_{x_n\rightarrow x}w(x_n).
\]

\textbf{Lower semicontinuity.}
Let $(x_n)_n$ be a sequence converging to $x$. Consider any sequence $(\varphi^{x_n})_n\subset\mathcal{D}$ of minimizers of $\mathcal{A}_{x_n}$. By the uniform coercivity estimates of Lemma  \ref{lem:coercivity_estimates}, we know that the sequence $(\varphi^{x_n})_n$ is uniformly bounded. Besides, up to a subsequence, it converges uniformly on compact subsets of $[1,+\infty)$ and weakly in $\mathcal D$ to a function $\bar\varphi\in\mathcal{D}$. 

We decompose the action $\mathcal{A}_x$ as in Lemma \ref{lem:quad_pot}. Being a quadratic form, we already know that $\varphi\mapsto Q_\infty(\varphi,\varphi)$ is weakly lower semi-continuous on $\mathcal{D}$.

By previous arguments, we know that there exists $\bar T>1$ such that on the interval $[\bar T,\infty)$ the integrand in $P_{x,\infty}(\varphi)$, here just $P_x(\varphi)$, is dominated by an $L^1$-function. 

After dividing $P_{x_n}(\varphi^{x_n})$ in two terms
\[
P_{x_n}(\varphi^{x_n})=\int_{1}^{\bar T} V(x_n,t,\varphi^{x_n}(t))\,\ud t + \int_{\bar T}^{+\infty} V(x_n,t,\varphi^{x_n}(t))\,\ud t,
\]
we observe what follows. On the interval $[1,\bar T]$, directly by the fact that $\varphi^{x_n}$ converges uniformly on $\bar \varphi$, and, thanks to \eqref{eq:V_bound}, on the interval $[\bar T,+\infty)$ by the Dominated Convergence Theorem, we obtain
\[
P_{x}(\bar\varphi) = \lim_{n\rightarrow+\infty} P_{x_n}(\varphi^{x_n}).
\]

Putting this together with the lower semicontinuity of $Q$, we obtain
\[
\begin{split}
w_\infty(x) & \le \mathcal{A}_x(\bar\varphi) = Q(\bar\varphi,\bar\varphi) +P_{x}(\bar\varphi)
\\
&\leq\liminf_{n\to+\infty} Q(\varphi^n,\varphi^n) +\liminf_{n\to+\infty}  P_{x_n}(\varphi^n)\leq \liminf_{n\rightarrow+\infty}\mathcal{A}_{x_n}(\varphi^{x_n}) = \liminf_{n\rightarrow+\infty}w_\infty(x_n).\end{split}
\]
\end{proof}

Now, we prove the convergence of $w(T,x)$ to $w_\infty(x)$.

\begin{lem}[Finite-horizon approximation]
\label{lem:unif_conv}
Let $(T_n)_n \subset (1,+\infty)$ be a sequence of finite times such that $T_n \to +\infty$, as $n \to +\infty$. 

Then, 
\[
\lim_{n \to +\infty} w(T_n,x)=w_\infty(x) \ \mbox{ uniformly on compact subsets of $\Omega$}.
\]
\end{lem}

\begin{proof}
    Consider $\mathcal K$ a fixed compact subset of $\mathcal{X}$. For every $x \in \mathcal K$, and $n \in \mathbb N$, by the coercivity of $\mathcal{A}_{x}^{T_n}$, there exists $\varphi^n \in \mathcal{D}^{T_n}$ such that
    \[
    w(T_n,x)=\mathcal{A}_{x}^{T_n}(\varphi^n).
    \]
    Let $\tilde\varphi^n \in \mathcal{D}$ be the following extension of $\varphi^n$ to $\mathcal{D}$:
    \[
\tilde \varphi^n(t)=
\begin{cases}
    \varphi^n(t), \ &\mbox{ in } \ [1,T_n]
    \\
    \varphi^n(2T_n-t), \ &\mbox{ in }  \ (T_n,2T_n-1]
    \\
    0, \ &\mbox{ in } \ (2T_n-1,+\infty).
\end{cases}
\]
Clearly, $\mathcal{A}_{x}^{T_n}(\tilde \varphi^n)=\mathcal{A}_{x}^{T_n}(\varphi^n)$. Since $\|\tilde\varphi^n\|_\mathcal{D} = 2 \|\varphi^n\|_{\mathcal{D}_{T_n}}$, and the uniform coercivity of the family of functionals $\mathcal{A}_{x}^{T_n}$, (see Section \ref{sec:coercivity_estimates}), we deduce that $(\tilde\varphi^n)_n$ is a bounded sequence on $\mathcal{D}$. This implies that there exists $\bar\varphi \in \mathcal{D}$ such that $\tilde\varphi^n \to \bar \varphi$, weakly on $\mathcal{D}$, uniformly on compact subsets of $(1,+\infty)$, and pointwise on $(1,+\infty)$. 

We want to prove that
\[
\mathcal{A}_{x}(\bar\varphi) = \lim_{n\rightarrow+\infty}\mathcal{A}_{x}^{T_n}(\tilde\varphi^n) = \lim_{n\to+\infty}\mathcal{A}_{x}^{T_n}(\varphi^n),
\]
uniformly on compact sets with respect to $x$.

\smallskip
\textbf{Hyperbolic case.}\ Like we did for the study of the coercivity for the hyperbolic case, we write the renormalized Lagrangian action as
\[
\mathcal{A}_{x}^{T_n}(\tilde\varphi^n) = \sum_{i<j}m_i m_j \int_{1}^{T_n}\frac{1}{2M} \langle\dot{\tilde\varphi}_{ij}^n(t),\dot{\tilde\varphi}_{ij}^n(t)\rangle + U_{ij}(at + \tilde\varphi^n(t)+x-a)-U_{ij}(at)\ \ud t,
\]
where 
\[
U_{ij}(at+\tilde\varphi^n(t)+x-a) = \frac{1}{|a_{ij}t + \tilde\varphi^n_{ij}(t)+x_{ij}-a_{ij}|}.
\]

Since $\tilde\varphi^n$ is a critical point of the action, it holds
\[
\ud \mathcal{A}_{x}^{T_n}(\tilde\varphi^n)[\psi] = 0, \ \mbox{ for all } \ \psi\in\mathcal{D}^{T_n},
\]
that is, for all $\psi\in\mathcal{D}^{T_n}$,
\[
\int_{1}^{T_n} \frac1M\langle\dot{\tilde\varphi}_{ij}^n(t),\dot\psi_{ij}(t)\rangle + \langle\nabla U_{ij}(at+\tilde\varphi^n(t)+x-a),\psi_{ij}(t)\rangle\ \ud t = 0.
\]

Denoting $\bar\varphi^n = \bar\varphi\big|_{[1,T_n]}$, the above equation holds, in particular, for $\psi(t)=\tilde\varphi_{ij}^n(t)-\bar\varphi_{ij}^n(t)$. Then
\[
\int_{1}^{T_n}\frac1M\langle\dot{\tilde\varphi}_{ij}^n(t),\dot{\tilde\varphi}_{ij}^n(t)-\dot{\bar\varphi}_{ij}^n(t)\rangle + \langle\nabla U_{ij}(at+\tilde\varphi^n(t)+x-a),\tilde\varphi_{ij}^n(t)-\bar\varphi_{ij}^n(t)\rangle\ \ud t = 0.
\]

Starting with the potential term, it holds
\[
\begin{split}
\left|\langle\nabla U_{ij}(at+\tilde\varphi^n(t)+x-a),\tilde\varphi_{ij}^n(t)-\bar\varphi_{ij}^n(t)\rangle \right| \leq C\frac{|\tilde\varphi_{ij}^n(t) - \bar\varphi_{ij}^n(t)|}{|a_{ij}t + \tilde\varphi^n_{ij}(t)+x_{ij}-a_{ij}|^2} \leq C'\frac{t^{1/2}}{|a_{ij}t + \tilde\varphi^n_{ij}(t)+x_{ij}-a_{ij}|^2},
\end{split}
\]
where the constants $C,C'$ depend only on the norms $\|\varphi_{ij}\|_{\mathcal{D}^T}$, which are bounded, and, for $t\rightarrow+\infty$, the last term is dominated by the function $\frac{C'}{t^{3/2}}\in L^1(1,+\infty)$. Since $\tilde\varphi_{ij}^n(t)-\bar\varphi_{ij}^n(t)\rightarrow0$ for $t \in (1,+\infty)$, it holds, by the Dominated Convergence Theorem, that
\begin{equation*} 
\int_1^{+\infty} \langle\nabla U_{ij}(at+\tilde\varphi^n(t)+x-a),\tilde\varphi_{ij}^n(t)-\bar\varphi_{ij}^n(t)\rangle\, \ud t\rightarrow 0,
\end{equation*}
and, in particular,
\begin{equation} 
\label{eq:conv_pot}
\int_1^{+\infty} \langle\nabla U_{ij}(at+\tilde\varphi^n(t)+x-a),\tilde\varphi_{ij}^n(t)\rangle - \langle\nabla U_{ij}(at+\bar\varphi(t)+x-a),\bar\varphi_{ij}^n(t)\rangle \, \ud t\rightarrow 0.
\end{equation}
Since by the weak convergence of $(\tilde\varphi^n(t))_n$ in $\mathcal{D}$ it holds
\[
\lim_{n\rightarrow+\infty}\int_{1}^{T_n} \langle\dot{\tilde\varphi}_{ij}^n(t),\dot{\bar\varphi}_{ij}^n(t)\rangle\ \ud t = \int_{1}^{+\infty} |\dot{\bar\varphi}_{ij}(t)|^2\ \ud t, 
\]
we can conclude that 
\begin{equation} 
\label{eq:conv_norm}
\lim_{n\rightarrow+\infty} \int_{1}^{T_n} |\dot{\tilde\varphi}_{ij}(t)|^2\ \ud t = \int_{1}^{+\infty} |\dot{\bar\varphi}_{ij}(t)|^2\ \ud t.
\end{equation}

From \eqref{eq:conv_pot} and \eqref{eq:conv_norm},
since
\[
\begin{split} 
\mathcal{A}_{x}^{T_n}(\tilde\varphi^n) &= \sum_{i<j}m_i m_j \int_{1}^{T_n}\frac{1}{2M} \langle\dot{\tilde\varphi}_{ij}^n(t),\dot{\tilde\varphi}_{ij}^n(t)\rangle + U_{ij}(at + \tilde\varphi^n(t)+x-a)-U_{ij}(at)\ \ud t
\\
&=
\sum_{i<j}m_i m_j \int_{0}^{1} \int_{1}^{T_n}\frac{1}{2M} \langle\dot{\tilde\varphi}_{ij}^n(t),\dot{\tilde\varphi}_{ij}^n(t)\rangle+ \langle \nabla U_{ij}(a t+s(\tilde\varphi^n(t) + x-a)), \tilde \varphi^n_{ij}(t) +  x_{ij}-a_{ij}\rangle \ \ud t\, \ud s,
\end{split}
\]
we deduce
\[
\lim_{n\rightarrow+\infty}\mathcal{A}_{x}^{T_n}(\tilde \varphi^n) = \mathcal{A}_{x}(\bar \varphi)
\]
pointwise.

We can also state that the convergence is uniform on compact sets with respect to $x$. Indeed, writing
\[
|\mathcal{A}_{x}(\tilde\varphi^n)-\mathcal{A}_x(\bar\varphi)| = \bigg|\sum_{i<j}m_i m_j\int_{1}^{+\infty}\int_{0}^{1} \langle\nabla U_{ij}(at+\bar\varphi(t)+x-a+\theta(\tilde\varphi^n(t)-\bar\varphi(t)),\tilde\varphi_{ij}^n(t)-\bar\varphi_{ij}(t)\rangle\ \ud \theta\ \ud t\bigg|, 
\]
we know that, for any $\varepsilon>0$, there are $T_\varepsilon>1$ and $N_\varepsilon$ such that, for any $n\geq N_\varepsilon$,
\[
\int_{T_\varepsilon}^{+\infty}\int_{0}^{1} \langle\nabla U_{ij}(at+\bar\varphi(t)+x-a+\theta(\tilde\varphi^n(t)-\bar\varphi(t)),\tilde\varphi_{ij}^n(t)-\bar\varphi_{ij}(t)\rangle\ \ud \theta\ \ud t <\frac{\varepsilon}{2}.
\]
This follows from the fact that, for $t\rightarrow+\infty$, there is a constant $C''$ that does not depend on $x$ such that
\[
|\langle\nabla U_{ij}(at+\bar\varphi(t)+x-a+\theta(\tilde\varphi^n(t)-\bar\varphi(t)),\tilde\varphi_{ij}^n(t)-\bar\varphi_{ij}(t)\rangle|\leq \frac{C''}{t^{3/2}}.
\]
In addition,
the fact that, for any $n\geq N_\varepsilon$, 
\[
\int_{1}^{T_\varepsilon}\int_{0}^{1} \langle\nabla U_{ij}(at+\bar\varphi(t)+x-a+\theta(\tilde\varphi^n(t)-\bar\varphi(t)),\tilde\varphi_{ij}^n(t)-\bar\varphi_{ij}(t)\rangle\ \ud \theta\ \ud t< \frac{\varepsilon}{2},
\]
follows from the continuity of $\tilde\varphi^n$ and $\bar\varphi$ with respect to $x$, from the uniform convergence of $\tilde\varphi^n$ to $\bar\varphi$, and the fact that, if $n$ is sufficiently large, $|at+\bar \varphi+x-a+\theta(\tilde\varphi^n-\bar\varphi)|\ge \delta/2$, for some $\delta>0$.

\smallskip
\textbf{Parabolic case.}\ In the parabolic case, we consider the renormalized Lagrangian action, as in Lemma \ref{lem:quad_pot}
\[
\begin{split}
\mathcal{A}_{x}^{T_n}(\tilde\varphi^n) = \int_{1}^{T_n}&\frac{1}{2}\|\dot{\tilde\varphi}^n(t)\|_\mathcal{M} + \frac{1}{2}\langle\nabla^2U(\beta b_mt^{2/3})\tilde\varphi^n(t),\tilde\varphi^n(t)\rangle+  U(\beta b_m t^{2/3}+\tilde\varphi^n(t)+x-\beta b_m) - U(\beta b_mt^{2/3})\\& - \langle\nabla U(\beta b_m t^{2/3}),\tilde\varphi^n(t)\rangle - \frac{1}{2}\langle\nabla^2U(\beta b_mt^{2/3})\tilde\varphi^n(t),\tilde\varphi^n(t)\rangle\ \ud t
\\
&=Q_{T_n}(\tilde\varphi^n, \tilde\varphi^n)+P_{x,T_n}(\tilde\varphi^n).
\end{split}
\]

Like in the previous case, the differential of the action is zero at $\tilde\varphi^n$. It holds, in particular,
\[
\begin{split}
\int_{1}^{T_n} &\langle\dot{\tilde\varphi}^n(t),\dot{\tilde\varphi}^n(t)-\dot{\bar\varphi}^n\rangle_\mathcal{M} + \langle\nabla^2U(\beta b_mt^{2/3})\tilde\varphi^n(t),\tilde\varphi^n(t)-\bar\varphi^n(t)\rangle \\&+ \langle\nabla U(\beta b_m t^{2/3}+\tilde\varphi^n(t)+x-\beta b_m),\tilde\varphi^n(t)-\bar\varphi^n(t)\rangle - \langle\nabla U(\beta b_m t^{2/3}),\tilde\varphi^n(t)-\bar\varphi^n(t)\rangle\\& -\langle\nabla^2U(\beta b_mt^{2/3})\tilde\varphi^n(t),\tilde\varphi^n(t)-\bar\varphi^n(t)\rangle\ \ud t = 0.
\end{split}
\]
Regarding the \emph{potential} part, we get
\[
\begin{split}
&\langle\nabla U(\beta b_m t^{2/3}+\tilde\varphi^n(t)+x-\beta b_m),\tilde\varphi^n(t)-\bar\varphi^n(t)\rangle + \langle\nabla U(\beta b_m t^{2/3}),\tilde\varphi^n(t)-\bar\varphi^n(t)\rangle -\langle\nabla^2U(\beta b_mt^{2/3}\tilde\varphi^n(t),\tilde\varphi^n(t)-\bar\varphi^n(t)\rangle\\&
=\int_{0}^{1} \langle\nabla^2 U(\beta b_mt^{2/3} + \theta(\tilde\varphi^n(t)+x-\beta b_m))(\tilde\varphi^n(t)+x-\beta b_m) - \nabla^2 U (\beta b_mt^{2/3})\tilde\varphi^n(t),\tilde\varphi^n(t)-\bar\varphi^n(t)\rangle\ \ud \theta,
\end{split}
\]
and, if we do not consider the term $\langle\nabla^2 U(\beta b_mt^{2/3} + \theta(\tilde\varphi^n(t)+x-\beta b_m))(x-\beta b_m),\tilde\varphi^n(t)-\bar\varphi^n(t)\rangle$,
\[
\begin{split}
\int_{0}^{1} &\langle(\nabla^2 U(\beta b_mt^{2/3} + \theta(\tilde\varphi^n(t)+x-\beta b_m)) - \nabla^2(\beta b_mt^{2/3}))\tilde\varphi^n(t),\tilde\varphi^n(t)-\bar\varphi^n(t)\rangle\ \ud t \\
& = \int_{0}^{1} \int_{0}^{1} \langle\nabla^3 U(\beta b_mt^{2/3} + \sigma\theta(\tilde\varphi^n(t)+x-\beta b_m))(\tilde\varphi^n(t) + x + \beta b_m),\tilde\varphi^n(t),\tilde\varphi^n(t)-\bar\varphi^n(t)\rangle\theta\ \ud \sigma\ \ud\theta. 
\end{split}
\]
We have
\[
\begin{split}
    &\langle\nabla^3 U(\beta b_mt^{2/3} + \sigma\theta(\tilde\varphi^n(t)+x-\beta b_m))\tilde\varphi^n(t),\tilde\varphi^n(t),\tilde\varphi^n(t)-\bar\varphi^n(t)\rangle\theta\\
    &\leq C\frac{\|\tilde\varphi^n(t)\|_\mathcal{M}\|\tilde\varphi^n(t)\|_\mathcal{M}\|\tilde\varphi^n(t)-\bar\varphi^n(t)\|_\mathcal{M}}{\|\beta b_mt^{2/3} + \sigma\theta(\tilde\varphi^n(t)+x-\beta b_m)\|_\mathcal{M}^4}\\
    &\leq C' \frac{t^{3/2}}{\|\beta b_mt^{2/3} + \sigma\theta(\tilde\varphi^n(t)+x-\beta b_m)\|_\mathcal{M}^4},
\end{split}
\]
where the constants $C,C'$ depend only on $\|\varphi\|_{\mathcal{D}^T}$ or $\|\varphi\|_\mathcal{D}$, which are bounded, and, for $t\rightarrow+\infty$, the last term is dominated by $\frac{C'}{t^{7/6}}\in L^1(1,+\infty)$. Then, by the Dominated Convergence Theorem, 
\[
\begin{split}
&\langle\nabla U(\beta b_m t^{2/3}+\tilde\varphi^n(t)+x-\beta b_m) - \nabla U(\beta b_m t^{2/3})-\nabla^2U(\beta b_mt^{2/3})\tilde\varphi^n(t),\tilde\varphi^n(t)-\bar\varphi^n(t)\rangle\ \ud t\rightarrow0\quad\text{in }L^1.
\end{split}
\]
Arguing as in the hyperbolic case, we are able to prove that
\[
\lim_{n}
P_{x,T_n}(\tilde\varphi^n)=
P(\bar\varphi).
\]

Now we study the remaining term. We notice that
\[
\begin{split}
\int_{1}^{T_n} &\langle\dot{\tilde\varphi}^n(t),\dot{\tilde\varphi}^n(t)-\dot{\bar\varphi}^n\rangle_\mathcal{M} + \langle\nabla^2U(\beta b_mt^{2/3})\tilde\varphi^n(t),\tilde\varphi^n(t)-\bar\varphi^n(t)\rangle\ \ud t\\
& = \int_{1}^{T_n} \|\dot{\tilde\varphi}^n(t)\|_\mathcal{M}^2-\langle\dot{\tilde\varphi}^n(t),\dot{\bar\varphi}^n(t)\rangle_\mathcal{M} + \langle\nabla^2U(\beta b_mt^{2/3})\tilde\varphi^n(t),\tilde\varphi^n(t)\rangle-\langle\nabla^2U(\beta b_mt^{2/3})\tilde\varphi^n(t),\bar\varphi^n(t)\rangle\ \ud t.
\end{split}
\]

By the Dominated Convergence Theorem and the fact that $\tilde\varphi^n\rightharpoonup\bar\varphi$ in $\mathcal{D}$, we have that 
\[
\lim_{n\rightarrow+\infty} \int_{1}^{T_n} \langle\dot{\tilde\varphi}^n(t),\dot{\bar\varphi}^n(t)\rangle_\mathcal{M}  +\langle\nabla^2U(\beta b_mt^{2/3})\tilde\varphi^n(t),\bar\varphi^n(t)\rangle\ \ud t = \int_{1}^{+\infty} \|\dot{\bar\varphi}(t)\|_\mathcal{M}^2 + \langle\nabla^2U(\beta b_mt^{2/3})\bar\varphi(t),\bar\varphi(t)\rangle\ \ud t.
\]
So,
\[
\lim_{n\rightarrow+\infty} \int_{1}^{T_n} \|\dot{\tilde\varphi}^n(t)\|_\mathcal{M}^2 + \langle\nabla^2U(\beta b_mt^{2/3})\tilde\varphi^n(t),\tilde\varphi^n(t)\rangle\ \ud t = \int_{1}^{+\infty} \|\dot{\bar\varphi}(t)\|_\mathcal{M}^2 + \langle\nabla^2U(\beta b_mt^{2/3})\bar\varphi(t),\bar\varphi(t)\rangle\ \ud t.
\]

Hence, with the notations of Lemma \ref{lem:quad_pot}, we have proved
\[
\lim_{n}
Q_{T_n}(\tilde \varphi^n, \tilde \varphi^n) = Q(\bar\varphi, \bar \varphi).
\]

Since
\[
\begin{split}
\mathcal{A}_{x}^{T_n}(\tilde\varphi^n) = \int_{0}^{1}\int_{1}^{T_n}&\frac{1}{2}\|\dot{\tilde\varphi}^n(t)\|_\mathcal{M} + \frac{1}{2}\langle\nabla^2U(\beta b_mt^{2/3})\tilde\varphi^n(t),\tilde\varphi^n(t)\rangle_\mathcal{M} \\& +  \langle\nabla U(\beta b_m t^{2/3}+\theta (\tilde\varphi^n(t)+x-\beta b_m)),\tilde\varphi^n(t)+\beta b_m\rangle \\& - \langle\nabla U(\beta b_mt^{2/3}),\tilde\varphi^n(t)\rangle - \frac{1}{2}\langle\nabla^2U(\beta b_mt^{2/3})\tilde\varphi^n(t),\tilde\varphi^n(t)\rangle_\mathcal{M}\ \ud t\ \ud \theta,
\end{split}
\]
we can thus conclude that
\[
\lim_{n\rightarrow+\infty}\mathcal{A}_{x}^{T_n}(\tilde \varphi^n) = \mathcal{A}_{x}(\bar \varphi).
\]
The proof of the uniform convergence over compact sets runs as in the hyperbolic case.

\smallskip
\textbf{Hyperbolic-parabolic case.}\ As per usual, in the hyperbolic-parabolic case we use the decomposition of the renormalized Lagrangian action given by \eqref{def:decomposed_lagrangian_action}. Studying the two terms separately, we apply the same considerations that we made previously in the hyperbolic and parabolic cases and we get to the same conclusion.
\end{proof}

\subsection{Semiconcavity with linear modulus of the value function}\label{sez_semiconcavity}

By proving that the value function $v$ is semiconcave with linear modulus, we also obtain its local Lipschitz continuity on the set $\Omega$.
\begin{prop}\label{prop:semiconcavity}
    For any set $W\subset\subset\Omega$ there is a constant $C>0$ such that for any $x,z\in\mathcal{X}$ with $x,x+z,x-z\in W$ it holds
    \begin{equation}\label{slm}
        v(x+z) + v(x-z) - 2v(x) \leq C\|z\|^2.
    \end{equation}
\end{prop}
\begin{proof}
    From the uniform coercivity estimates, it follows that there is a constant $\hat C>0$ such that $\|\varphi^x\|_\mathcal{D}\leq \hat C$ for all $x\in W$.

    Fixing a set $W\subset\subset \mathcal{X}\setminus \Delta$, we know from Lemma \ref{lemma_collisionless_minimal} that there is $\delta>0$ such that
    \[
    \inf_{t\geq1} d(\gamma^x(t),\Delta)\geq\delta,
    \]
    where $\gamma^x(t) = r_0(t) + \varphi^x(t) + x - r_0(1)$ and $\varphi^x$ realizes the minimum of the functional $\mathcal{A}_x$.

    Since $v$ is continuous over the set $\overline{W}$, $v$ is bounded over $\overline{W}$: this means that we only have to prove \eqref{slm} for small values of $\|z\|$. 

    For $\|z\|$ sufficiently small, we have
    \[
    \inf_{t\geq1} d(\gamma^{x}(t)\pm sz,\Delta)\geq\frac{\delta}{2}, \ \ s \in [0,1].
    \]
    Thus
    \begin{displaymath}
        \begin{split}
            &v(x+z) + v(x-z) - 2v(x) \leq \int_{1}^{+\infty} L_{x+z}^{ren}(\varphi^x(t)) + L_{x-z}^{ren}(\varphi^x(t)) - 2L_{x}^{ren}(\varphi^x(t))\ \ud t\\
            & = \int_{1}^{+\infty} U(r_0(t) + \varphi^x(t) + x + z - r_0(1)) + U(r_0(t) + \varphi^x(t) + x - z - r_0(1)) -2U(r_0(t) + \varphi^x(t) + x - r_0(1))\ \ud t\\
            & = \int_{1}^{+\infty} \int_{0}^{1} \langle \nabla U(r_0(t) + \varphi^x(t) + x - r_0(1) + sz), z\rangle - \langle \nabla U(r_0(t) + \varphi^x(t) + x - r_0(1) - sz), z\rangle\ \ud s\ \ud t\\
            & = \int_{1}^{+\infty} \int_{0}^{1} \int_{0}^{1} -2\langle \nabla^2 U(r_0(t) + \varphi^x(t) + x - r_0(1) + (1-2\tau)sz)sz, z\rangle\ \ud\tau\ \ud s\ \ud t,
        \end{split}
    \end{displaymath}
    where we denote $L_{x}^{ren}(\varphi^x(t))=\frac{1}{2}\|\varphi^x(t)\|_\mathcal{M}^2 + U(r_0(t) + \varphi^x(t) + x - r_0(1)) - U(r_0(t)) - \langle \ddot{r}_0(t),\varphi^x(t) \rangle_\mathcal{M}$. The thesis follows from the fact that there is a constant $C'>0$ with the property that there is $\varepsilon>0$ such that for $\|z\|<\varepsilon$ and for any $x\in W$
        \[
        \|\nabla^2 U(\gamma^x(t) + z)\| \leq \frac{C'}{t^2},\quad \forall t\geq 1,
        \]
        where we used the operator norm of a matrix.
    
\end{proof}

In particular, the semiconcavity of the value function implies that the function is locally Lipschitz continuous, see e.g. \cite{MR2041617}.

\subsection{Proof of Theorem \ref{thm:HJ}}\label{sec:proofHJ}

We prove Theorem \ref{thm:HJ} by means of the finite-horizon approximation of Proposition \ref{prop:conv_v_T}, and the \emph{stability principle} for viscosity solutions. To this end, the key idea is to show that each $v(T,x)$ satisfies a proper differential equation, as proved in the following proposition.

We recall that, similarly to Definition \ref{def:viscosity}, for a function $v:[1,+\infty)\times\Omega\rightarrow\R$, we define the Fréchet superdifferential and subdifferential of $v$ at $(T,x)$ as
    \begin{displaymath}\begin{split}
        &D^- v(T,x) = \bigg\{ (p_T,p_x)\in \R\times\mathcal{X} : \liminf_{(\bar T,y)\rightarrow (T,x)}\frac{v(\bar T,y)-v(T,x) -p_T(\bar T-T) - \langle p_x,y-x\rangle_\mathcal{M}}{\sqrt{(\bar{T}-T)^2+\|y-x\|_\mathcal{M}^2}}\geq 0\bigg\},\\
        &D^+v(T,x) = \bigg\{  (p_T,p_x)\in \R\times\mathcal{X} : \limsup_{(\bar T,y)\rightarrow (T,x)}\frac{v(\bar T,y)-v(T,x) -p_T(\bar T-T) - \langle p_x,y-x\rangle_\mathcal{M}}{\sqrt{(\bar{T}-T)^2+\|y-x\|_\mathcal{M}^2}}\leq 0\bigg\}.
    \end{split}\end{displaymath}
    
Given a function $F:[1,+\infty)\times\Omega\rightarrow\R$, we consider the equation
\begin{equation}\label{eq:viscosity_solution_v(T,x)}
\frac{\partial v}{\partial T}(T,x) + H(x,\nabla v(T,x)) + F(T,x)=0.
\end{equation}
We say that a continuous function $v$ is a viscosity supersolution of \eqref{eq:viscosity_solution_v(T,x)} if, for every $(T,x)$ and $(p_T,p_x)\in D^-v(T,x)$, it holds
    \[
    p_T + H(x,\mathcal Mp_x) + F_*(T,x) \ge 0,
    \]
    where $F_*$ is the {\em lower semicontinuous envelope} of $F$; $v$ is called viscosity subsolution of \eqref{eq:viscosity_solution_v(T,x)} if, for every $(T,x)$ and $(p_T,p_x)\in D^+v(T,x)$, it holds
    \[
    p_T + H(x,\mathcal Mp_x) + F^*(T,x) \le 0,
    \]
    where $F^*$ is the {\em upper semicontinuos envelope} of $F$.
    
If $v$ is both a viscosity supersolution and a viscosity subsolution, then it is called a viscosity solution of \eqref{eq:viscosity_solution_v(T,x)}.

\begin{prop}
    \label{prop:HJ_v_T}
Let $(T,x)\in(1,+\infty)\times\Omega$. Then:
\begin{itemize}
    \item there exists a minimizer $\varphi_{x,T}$ of $\mathcal{A}_{x}^T$ on $\mathcal{D}^T$ such that $v(T,x)$ is a viscosity subsolution of
    \[
    \frac{\partial v}{\partial T}(T,x) + H(x,\nabla v(T,x)) +\langle\ddot{r}_0(T),\varphi_{x,T}(T)+x\rangle_\mathcal{M} - H(r_0(T),\dot r_0(T)) \le 0,
    \]
    \item for all $\varphi_{x,T}$ minimizers of $\mathcal{A}_{x}^T$ on $\mathcal{D}^T$, $v(T,x)$ is a viscosity supersolution of 
    \[
    \frac{\partial v}{\partial T}(T,x) + H(x,\nabla v(T,x)) +\langle\ddot{r}_0(T),\varphi_{x,T}(T)+x\rangle_\mathcal{M} - H(r_0(T),\dot r_0(T)) \ge 0.
    \]
\end{itemize}
In particular, the function $v(T,x)$ is a viscosity solution in $(1,+\infty)\times\Omega$ of
\begin{displaymath}
    \label{eq:v_T}
\frac{\partial v}{\partial T}(T,x) = -\frac{1}{2}\|\nabla v(T,x)\|^2_{\mathcal{M}^{-1}} + U(x) - \inf_{\varphi_{x,T} \in \mathcal{Z}}\langle\ddot{r}_0(T),\varphi_{x,T}(T)+x\rangle_\mathcal{M}+\frac{1}{2}\|\dot r_0(T)\|_\mathcal{M}^2 - U(r_0(T)), 
\end{displaymath}
where $\mathcal{Z}$ is the set of functions $\varphi_{x,T}$ such that $\mathcal{A}_{x}^T(\varphi_{x,T})=  \min_{\varphi\in\mathcal{D}^T}\mathcal{A}_{x}^T(\varphi)$.

Moreover, it holds
\begin{equation}
    \label{eq:pa_T_v}
\lim_{T\to +\infty} \frac{\partial v}{\partial T}(T,x)=0.
\end{equation}
\end{prop}

Taking Proposition \ref{prop:HJ_v_T}, Theorem \ref{thm:HJ} follows at once, as follows.

\begin{proof}[Proof of Theorem \ref{thm:HJ}]
    
From the uniform coercivity estimates, fixed $x$, we have that $|\varphi_{x,T}(T)|\le C \sqrt{T}$, and hence
\[
\left|\langle \ddot r_0(T), \varphi_{x,T}(T)\rangle_\mathcal{M} \right| \le \|\ddot r_0(T)\|_\mathcal{M} C \sqrt{T} \to 0, \ \mbox{ as } \ T \to +\infty.
\]
From a direct inspection,
\[
\frac{1}{2}\|\dot r_0(T)\|_\mathcal{M}^2 - U(r_0(T)) \to \frac{\|a\|^2_\mathcal{M}}{2}, \ \mbox{ as } \ T\to +\infty.
\]
Hence, from \eqref{eq:pa_T_v} and the uniform convergence of $v(T,x)$ to $v(x)$ over compact subsets of $\Omega$ (Proposition \ref{prop:conv_v_T}), using the stability principle of viscosity solutions (see e.g. \cite{Crandall}), we conclude that the uniform limit $v(x)$ satisfies in the viscosity sense the limit equation, that is, \eqref{eq:HJ_v}.

Although the stability principle stated in \cite{Crandall} is formulated for sequences of continuous equations, it still applies in our case, since the only term which may lack continuity, that is, $\inf_{\varphi_{x,T} \in \mathcal{Z}}\langle\ddot{r}_0(T),\varphi_{x,T}(T)+x\rangle_\mathcal{M}$, converges to zero, as $T\to+\infty$, uniformly over compact sets with respect to $x$.
\end{proof}

The remainder of the section is then devoted to proving Proposition \ref{prop:HJ_v_T}.

We need several lemmas, involving supplementary value functions $u(T,x)$, defined for any $x \in\Omega$ and $T>1$, by
\begin{equation}\label{eq:u(T,x)_vers1}
\begin{split}
u(T,x)&= \min_{\eta \in H^1([1,T]):\,\eta(1)=x}\int_1^T L(\eta(t),\dot \eta(t))\,\ud t - \langle \dot r_0(T), \eta(T)\rangle_\mathcal{M}
\\
&=\int_1^T L(\gamma(t),\dot{\gamma}(t))\,\ud t - \langle \dot r_0(T), \gamma(T)\rangle_\mathcal{M},
\end{split}
\end{equation}
where $\gamma$ satisfies
\begin{equation} 
\label{eq:problem_gamma_u}
\begin{cases}
\ddot{\gamma} = \nabla U(\gamma) \ \mbox{ a.e. in } \ (1,T)
\\
\gamma(1)=x
\\
\dot{\gamma}(T)=\dot r_0(T).
\end{cases}
\end{equation}
Notice that, in general, Problem \eqref{eq:problem_gamma_u} is not uniquely solvable.

Furthermore, by setting $\hat \gamma(t)=\gamma(T+1-t)$, we get a solution of
\[
\begin{cases}
\ddot{\hat\gamma} = \nabla U(\hat \gamma) \ \mbox{ a.e. in } \ (1,T)
\\
\hat\gamma(T)=x
\\
\dot{\hat\gamma}(1)=-\dot r_0(T),
\end{cases}
\]
that satisfies
\begin{equation}\label{eq:u(T,x)_vers2}
\begin{split}
u(T,x) &=\int_1^T L(\hat\gamma(t),\dot{\hat\gamma}(t))\, \ud t -\langle \dot r_0(T), \hat\gamma(1)\rangle_\mathcal{M}
\\
&= \min_{\eta \in H^1([1,T]):\, \eta(T)=x}\int_1^T L(\eta(t),\dot\eta(t))\, \ud t -
\langle \dot r_0(T), \eta(1)\rangle_\mathcal{M}.
\end{split}
\end{equation}

    Retracing the arguments made by Cannarsa and Sinestrari, we want to prove a similar inequality to the Dynamic Programming Principle (Theorem 1.2.2 in \cite{MR2041617}), which translates into the following result.
\begin{lem}
\label{lem:DDP}
    Let $T\in\R$, $T>1$, and $x\in{\Omega}$, and consider a curve $\overline{\gamma}\in H^1([1,T])$ such that $\overline{\gamma}(T)=x$. Then, for all $T'\in[1,T]$,
    \begin{equation}\label{eq:dpp}
        u(T,x) \leq u(T',\overline{\gamma}(T')) + \int_{T'}^{T} L(\overline{\gamma}(t),\dot{\overline{\gamma}}(t))\ \ud t - \langle\dot{r}_0(T)-\dot{r}_0(T'),\Tilde{\gamma}(1)\rangle_\mathcal{M},
    \end{equation}
    where $\Tilde{\gamma}\in H^1([1,T'])$ realizes the minimum in $u(T',\overline{\gamma}(T'))$ and is such that $\Tilde{\gamma}(T')=\overline{\gamma}(T')$.
\end{lem}
\begin{proof}
    Fix $T'\in[1,T]$ and let $\Tilde{\gamma}\in H([1,T'])$ such that $\Tilde{\gamma}(T') = \overline{\gamma}(T')$. Setting
    \[
    \xi(t)=\begin{cases}
        \Tilde{\gamma}(t),\quad t\in[1,T']\\
        \overline{\gamma}(t),\quad t\in[T',T]
    \end{cases}
    \]
    we have $\xi\in H^1([1,T])$ and $\xi(T)=x$. Therefore,
    \[\begin{split}
    u(T,x) &\leq \int_{1}^{T}L(\xi(t),\dot{\xi}(t))\ \ud t - \langle\dot{r}_0(T),\xi(1)\rangle_\mathcal{M}\\
    &=\int_{1}^{T'}L(\Tilde{\gamma}(t),\Tilde{\gamma}(t))\ \ud t + \int_{T'}^{T}L(\overline{\gamma}(t),\overline{\gamma}(t))\ \ud t - \langle\dot{r}_0(T),\Tilde{\gamma}(1)\rangle_\mathcal{M}\\
    &=\int_{1}^{T'}L(\Tilde{\gamma}(t),\Tilde{\gamma}(t))\ \ud t + \int_{T'}^{T}L(\overline{\gamma}(t),\overline{\gamma}(t))\ \ud t - \langle\dot{r}_0(T'),\Tilde{\gamma}(1)\rangle_\mathcal{M} - \langle\dot{r}_0(T)-\dot{r}_0(T'),\Tilde{\gamma}(1)\rangle_\mathcal{M}.
    \end{split}\]
    Taking the infimum over all $\Tilde{\gamma}\in H^1([1,T'])$, we get \eqref{eq:dpp}.    
\end{proof}

We now exploit Lemma \ref{lem:DDP} to prove the next result.
    \begin{lem}
    \label{lem:Eq_u_T}
    For every $(T,x)\in(1,+\infty)\times\Omega$, we have the following:
    \begin{itemize}
        \item for every $(p_T,p_x) \in D^+ u(T,x)$, there exists a curve $\gamma$ realizing $u(T,x)$ in \eqref{eq:u(T,x)_vers2} such that, in the viscosity sense,
        \[
        p_T+H(x,\mathcal M p_x)+\langle \ddot{r}_0(T), \gamma(1)\rangle_\mathcal{M} \le 0;        \]
        \item for every $(p_T,p_x) \in D^- u(T,x)$ and for every $\gamma$ realizing $u(T,x)$ in \eqref{eq:u(T,x)_vers2} it holds, in the viscosity sense,
        \[
        p_T + H(x, \mathcal M p_x) + \langle \ddot{r}_0(T),\gamma(1)\rangle_\mathcal{M} \geq 0.
        \]
    \end{itemize}

    In particular, the function
        $u(T,x)$ defined in \eqref{eq:u(T,x)_vers1} satisfies, in the viscosity sense,
        \begin{equation}
            \label{eq:Eq_u_T}
           -\frac{\partial u}{\partial T}(T,x)=-H(x,\nabla u(T,x)) + \inf_{\gamma \in \mathcal{Y}}\langle \ddot r_0(T),\gamma(T)\rangle_\mathcal{M},
         \end{equation}
        where $\mathcal Y$ is the set of curves $\gamma \in H^1([1,T])$, $\gamma(1)=x$ that realize $u(T,x)$.
    \end{lem}

    Before starting the rigorous proof of Lemma \ref{lem:Eq_u_T} by proving that $u(T,x)$ is a viscosity solution of \eqref{eq:Eq_u_T}, we give the heuristic idea.

    By taking a point of differentiability $x$ and formally differentiating the value function with respect to $x$, we obtain
\[
\nabla u(T,x) = -\mathcal{M}\dot{\gamma}(1).
\]
Indeed, denoting $\psi(t):=\frac{\partial \gamma}{\partial x_i}(t)$ and with $e_i$ the $i$-th element of the canonical basis, we have
\[
\begin{split}
\frac{\partial u}{\partial x_i}(T,x) &= \int_{1}^{T} \frac{\partial L}{\partial \gamma}(\gamma(t),\dot{\gamma}(t))\psi(t) + \frac{\partial L}{\partial \dot{\gamma}}(\gamma(t),\dot{\gamma}(t))\dot{\psi}(t)\ \ud t - \langle \dot{r}_0(T),\psi(T)\rangle_\mathcal{M}\\
&=\frac{\partial L}{\partial \dot{\gamma}}(\gamma(t),\dot{\gamma}(t))\psi(t)\bigg|^T_1 - \int_{1}^{T} \bigg( \frac{\ud}{\ud t}\bigg(\frac{\partial L}{\partial\dot{\gamma}}(\gamma(t),\dot{\gamma}(t))\bigg) - \frac{\partial L}{\partial \gamma}(\gamma(t),\dot{\gamma}(t)) \bigg)\psi(t)\ \ud t - \langle\dot{r}_0(T),\psi(T)\rangle_\mathcal{M}\\
&=\langle\dot{r}_0(T),\psi(T)\rangle_\mathcal{M}-\langle\dot{\gamma}(1),\psi(1)\rangle_\mathcal{M}-\langle\dot{r}_0(T),\psi(T)\rangle_\mathcal{M}\\
&=-m_i\dot{\gamma}(1)\cdot e_i,
\end{split}
\]
which follows from integration by parts. Besides, using integration by parts and the conservation of the energy, we denote $\phi(t):=\frac{\partial \gamma}{\partial T}(t)$ and compute
\[
\begin{split}
\frac{\partial u}{\partial T}(T,x) &= L(\gamma(T),\dot{\gamma}(T)) + \int_{1}^{T} \frac{\partial L}{\partial \gamma}(\gamma(t),\dot{\gamma}(t))\phi(t) + \frac{\partial L}{\partial \dot{\gamma}}(\gamma(t),\dot{\gamma}(t))\dot{\phi}(t)\ \ud t\\
&-\langle \ddot{r}_0(T),\gamma(T)\rangle_\mathcal{M} - \langle \dot{r}_0(T),\dot{\gamma}(T)\rangle_\mathcal{M} - \langle \dot r_0(T), \phi(T)\rangle_\mathcal{M}\\
&= L(\gamma(T),\dot{\gamma}(T)) + \frac{\partial L}{\partial \dot{\gamma}}(\gamma(T),\dot{\gamma}(T))\phi(T) - \frac{\partial L}{\partial \dot{\gamma}}(\gamma(1),\dot{\gamma}(1))\phi(1) \\
&- \int_{1}^{T} \bigg( \frac{\ud}{\ud t}\bigg(\frac{\partial L}{\partial\dot{\gamma}}(\gamma(t),\dot{\gamma}(t))\bigg) - \frac{\partial L}{\partial \gamma}(\gamma(t),\dot{\gamma}(t)) \bigg)\phi(t)\ \ud t \\
&-\langle \ddot{r}_0(T),\gamma(T)\rangle_\mathcal{M} - \langle \dot{r}_0(T),\dot{\gamma}(T)\rangle_\mathcal{M} - \langle \dot r_0(T), \phi(T)\rangle_\mathcal{M},
\end{split}
\]
and therefore
\[
\begin{split} 
\frac{\partial u}{\partial T}(T,x)  &=-\frac{1}{2}\|\dot{\gamma}(T)\|_\mathcal{M}^2 + U(\gamma(T)) - \langle \ddot{r}_0(T),\gamma(T)\rangle_\mathcal{M}\\
&=-\frac{1}{2}\|\dot{\gamma}(1)\|_\mathcal{M}^2 + U(\gamma(1)) - \langle \ddot{r}_0(T),\gamma(T)\rangle_\mathcal{M}\\
&=-H(x,\nabla u(T,x)) - \langle \ddot{r}_0(T),\gamma(T)\rangle_\mathcal{M}.
\end{split}
\]

Now, we prove the same result in a rigorous way.

    \begin{proof}[Proof of Lemma \ref{lem:Eq_u_T}]
To simplify the computations, we consider the formulation of $u(T,x)$ given by the last line in \eqref{eq:u(T,x)_vers2}, so that the end-point of $\gamma$ realizing $u(T,x)$ is $x$.

\smallskip
Take $T>1$, $x\in\Omega$ and $(p_T,p_x)\in D^+ u(T,x)$. Then, for every $z\in\mathcal{X}$, 
\[
\limsup_{h\rightarrow0^+}\frac{u(T-h,x-hz)-u(T,x) + h(p_T + \langle z,p_x\rangle_\mathcal{M})}{h\sqrt{1+\|z\|_\mathcal{M}^2}}\leq 0,
\]
which is equivalent to
\[
\limsup_{h\rightarrow0^+}\frac{u(T-h,x-hz)-u(T,x)}{h}\leq -p_T - \langle z,p_x\rangle_\mathcal{M}.
\]
Setting $\zeta(t) = x + (t-T)z$, with $z\in\mathcal{X}$, we can apply \eqref{eq:dpp} to get
\[
\begin{split}
    u(T,x) &\leq u(T-h,\zeta(T-h)) + \int_{T-h}^{T}L(\zeta(t),\dot{\zeta}(t))\ \ud t - \langle \dot{r}_0(T)-\dot{r}_0(T-h),\gamma_h(1)\rangle_\mathcal{M}\\
    &=u(T-h,x-hz) + \int_{T-h}^{T}L(x+(t-T)z,z)\ \ud t - \langle \dot{r}_0(T)-\dot{r}_0(T-h),\gamma_h(1)\rangle_\mathcal{M},
\end{split}
\]
where $\gamma_h$ realizes the minimum in $u(T-h,\zeta(T-h))$. This implies
\[
\begin{split}
    \limsup_{h\rightarrow 0^+} \frac{u(T-h,x-hz)-u(T,x)}{h} &\geq \lim_{h\rightarrow0^+} -\frac{1}{h} \int_{T-h}^{T}L(x+(t-T)z,z)\ \ud t + \frac{1}{h} \langle \dot{r}_0(T)-\dot{r}_0(T-h),
{\gamma_h}(1)\rangle_\mathcal{M}\\
    &= -L(x,z) + \langle \ddot{r}_0(T),\gamma(1)\rangle_\mathcal{M},
\end{split}
\]
where we used the fact that 
\begin{equation}\label{eq:limit_gamma_h}
    \gamma_h(1)\rightarrow\gamma(1)\text{ as }h\rightarrow0^+,
\end{equation}
with $\gamma$ a minimizer realizing $u(T,x)$. Notice that not every $\gamma$ minimizer is in general obtained as a limit of $\gamma_h$.

\smallskip
We can thus conclude that 
\[
p_T + \langle p_x,z \rangle_\mathcal{M} - L(x,z) + \langle \ddot{r}_0(T),\gamma(1)\rangle_\mathcal{M} \leq 0, \ \mbox{ for all } \ z \in \mathcal X,
\]
which, taking $z=p_x$, yields 
\[
p_T + H(x,\mathcal M p_x) + \langle \ddot{r}_0(T),\gamma(1)\rangle_\mathcal{M} \leq 0,
\]
for every $\gamma \in \mathcal Y$ such that \eqref{eq:limit_gamma_h} holds.

Hence, we have 
\[
p_T + H(x,\mathcal M p_x) + \inf_{\gamma\in\mathcal Y}\langle \ddot{r}_0(T),\gamma(1)\rangle_\mathcal{M} \leq 0.
\]
Now, since $p_x$ is related to $\nabla_{\mathcal M} u = \mathcal M^{-1} \nabla u$, then $u(t,x)$ is a viscosity sub-solution of
\[
\frac{\partial u}{\partial T} +H(x,{\nabla u(T,x)})+\inf_{\gamma\in \mathcal{Y}} \langle \ddot{r}_0(T), \gamma(1)\rangle_{\mathcal M} =0.
\]

To prove that $u$ is a viscosity supersolution, suppose that $\gamma$ is a minimizer for the value function $u(T,x)$ with $x=\gamma(T)$ and $w=\dot{\gamma}(T)$. Then,
\[
\begin{split}
u(T-h,{\gamma}(T-h)) - u(T,x) 
&\leq \int_{1}^{T-h} L(\gamma(t),\dot{\gamma}(t))\ \ud t - \langle \dot{r}_0(T-h), \gamma(1) \rangle_\mathcal{M} - \int_{1}^{T} L(\gamma(t),\dot{\gamma}(t))\ \ud t + \langle \dot{r}_0(T), \gamma(1) \rangle_\mathcal{M}\\
&= -\int_{T-h}^{T} L(\gamma(t),\dot{\gamma}(t))\ \ud t + \langle \dot{r}_0(T) - \dot{r}_0(T-h), \gamma(1) \rangle_\mathcal{M}.
\end{split}
\]
This implies
\[
\begin{split}
    \lim_{h\rightarrow 0^+} \frac{u(T-h,{\gamma}(T-h))-u(T,x)}{h} &\leq \lim_{h\rightarrow0^+} -\frac{1}{h} \int_{T-h}^{T}L(\gamma(t),\dot{\gamma}(t))\ \ud t + \frac{1}{h} \langle \dot{r}_0(T)-\dot{r}_0(T-h),\gamma(1)\rangle_\mathcal{M}\\
    &= -L(x,w) + \langle \ddot{r}_0(T),\gamma(1)\rangle_\mathcal{M}.
\end{split}
\]
Given $(p_T,p_x)\in D^-u(T,x)$, from the fact that $u(T,x)$ is Lipschtiz continuous (see Section \ref{sez_semiconcavity}), it holds
\[
\liminf_{h\rightarrow 0^+} \frac{u(T-h,{\gamma}(T-h))-u(T,x)}{h}\geq -p_T - \langle p_x,w \rangle_\mathcal{M}.
\]
Then
\[
p_T + \langle p_x,w \rangle_\mathcal{M} - L(x,w) + \langle \ddot{r}_0(T),\gamma(1)\rangle_\mathcal{M} \geq 0,
\]
which leads to
\[
p_T + H(x,\mathcal M p_x) + \langle \ddot{r}_0(T),\gamma(1)\rangle_\mathcal{M} \geq 0.
\]
for every $\gamma \in \mathcal Y$

As a consequence, $u(T,x)$ is a viscosity supersolution of 
\[
\frac{\partial u}{\partial T} + H(x,\nabla u(T,x)) + \inf_{\gamma \in \mathcal Y}\langle\ddot{r}_0(T), \gamma(1)\rangle_\mathcal{M} =0.
\]

We have thus proved that the value function $u$ solves, for a fixed $T>1$, the equation
\[
-\frac{\partial u}{\partial T}(T,x) = \frac{1}{2}\|\nabla u(T,x)\|^2_{\mathcal{M}^{-1}} - U(x) + \inf_{\gamma \in \mathcal Y}\langle\ddot{r}_0(T),\gamma(T)\rangle_\mathcal{M}
\]
in the viscosity sense.
\end{proof}

We finally make use of the following relation between $v(T,x)$ and $u(T,x)$.

\begin{lem}For every $T>1$ and $x \in \Omega$, it holds
    \[v(T,x)=u(T,x)-\int_1^T \frac{1}{2}\|\dot r_0(t)\|^2_\mathcal{M}+U(r_0(t))\, \ud t +\langle \dot r_0(T), r_0(T)-r_0(1)\rangle_\mathcal{M}.
    \]
    \end{lem}

    \begin{proof}
        We recall that $v(T,x)=w(T,x)-\langle \dot r_0(T),x\rangle_\mathcal{M}$, as defined in \eqref{eq:v(T,x)}, with $w(T,x)$ defined in \eqref{eq:min_action}.

        The statement is then true if we prove that
        \[
u(T,x)=w(T,x)+\int_1^T \frac{1}{2}\|\dot r_0(t)\|^2_\mathcal{M}+U(r_0(t))\, \ud t -\langle \dot r_0(T), r_0(T)+x-r_0(1)\rangle_\mathcal{M}.
        \]
        This follows from the next chain of identities:
        \[
\begin{split}
u(T,x) &= \int_{1}^{T} L(\bar\gamma(t),\dot{\bar\gamma}(t))\ \ud t - \langle\dot{r}_0(T),\bar\gamma(T)\rangle_\mathcal{M}\\
& = \int_{1}^{T} \frac{1}{2}\|\dot{\bar\varphi}(t)\|_\mathcal{M}^2 + \frac{1}{2}\|\dot r_0(t)\|_\mathcal{M}^2 + \langle \dot{\bar\varphi}(t),\dot r_0(t)\rangle_\mathcal{M}
 +U(r_0(t) + \bar\varphi(t)+x-r_0(1))\ \ud t\\
 &- \langle \dot r_0(T), r_0(T)\rangle_{\mathcal{M}}-\langle\dot{r}_0(T),\bar\varphi(T)\rangle_\mathcal{M}- \langle\dot{r}_0(T),x-r_0(1)\rangle_\mathcal{M}\\
 & = \int_{1}^{T} \frac{1}{2}\|\dot{\bar\varphi}(t)\|_\mathcal{M}^2+U(r_0(t) + \bar\varphi(t)+x-r_0(1)) - \langle \ddot r_0(t),\bar\varphi(t)\rangle_\mathcal{M} + \frac{1}{2}\|\dot r_0(t)\|_\mathcal{M}^2
 \ \ud t
 \\
 & - \langle \dot r_0(T), r_0(T)\rangle_{\mathcal{M}} -\langle\dot{r}_0(T),x-r_0(1)\rangle_\mathcal{M},
 \end{split}
 \]
which gives
\[
 \begin{split}
  u(T,x) &= \int_{1}^{T} \frac{1}{2}\|\dot{\bar\varphi}(t)\|_\mathcal{M}^2+U(r_0(t) + \bar\varphi(t)+x-r_0(1)) -U(r_0(t))- \langle \ddot r_0(t),\bar\varphi(t)\rangle_\mathcal{M}\ \ud t + \int_{1}^{T} \frac{1}{2}\|\dot r_0(t)\|_\mathcal{M}^2 \\
 &+ U(r_0(t))\ \ud t - \langle \dot r_0(T), r_0(T)\rangle_{\mathcal{M}} -\langle\dot{r}_0(T),x-r_0(1)\rangle_\mathcal{M},
 \end{split}
\]
and therefore
\[
u(T,x)= w(T,x) + \int_{1}^{T} \frac{1}{2}\|\dot r_0(t)\|_\mathcal{M}^2 + U(r_0(t))\ \ud t - \langle \dot r_0(T), r_0(T)\rangle_{\mathcal{M}} -\langle\dot{r}_0(T),x-r_0(1)\rangle_\mathcal{M}.
\]
Above, we have taken advantage of the fact that if $\bar\gamma$ realizes $u(T,x)$, then $\mathcal{D}^T \ni \bar\varphi = \bar\gamma-r_0(t)-x+r_0(1)$ realizes $w(T,x)$. This follows from the fact that $\bar\varphi$ satisfies
\[
\begin{cases}
    \ddot\varphi = \nabla U(r_0+\varphi+x-r_0(1)) \ \mbox{ a.e. in } \ (1,T)
    \\
    \varphi(1)=0
    \\
    \dot\varphi(T)=0,
\end{cases}
\]
that is, $\varphi$ is a critical point of $\mathcal{A}_{x}^T$ on $\mathcal{D}^T$.
    \end{proof}

\begin{proof}[Proof of Proposition \ref{prop:HJ_v_T}]

    This follows from \eqref{eq:Eq_u_T}, after one observes that $\nabla v(T,x)=\nabla u(T,x)$ in the viscosity sense.

We claim that
\begin{equation}
    \label{eq:partial_T_u_limit}
\lim_{T\rightarrow+\infty} -\frac{1}{2}\|\dot{\gamma}(T)\|_\mathcal{M}^2 + U(\gamma(T)) - \langle \ddot{r}_0(T),\gamma(T)\rangle_\mathcal{M} = -\frac{\|a\|_\mathcal{M}^2}{2}.
\end{equation}
If \eqref{eq:partial_T_u_limit} is proved, then, since
\[
\frac{\partial u}{\partial T}(T,x)=-\frac{1}{2}\|\dot{\gamma}(T)\|_\mathcal{M}^2 + U(\gamma(T)) - \langle \ddot{r}_0(T),\gamma(T)\rangle_\mathcal{M}
\]
in the viscosity sense, we conclude that
\[
\lim_{T\to +\infty}\frac{\partial u}{\partial T}(T,x)=-\frac{\|a\|_\mathcal{M}^2}{2},
\]
in the viscosity sense.

\smallskip
We prove our claim \eqref{eq:partial_T_u_limit}:
\begin{itemize}
    \item $\lim_{T\rightarrow+\infty}\|\dot\gamma(T)\|_\mathcal{M}^2=\lim_{T\rightarrow+\infty}\|\dot r_0(T)\|_\mathcal{M}^2=\|a\|_\mathcal{M}^2$.
    \item $\lim_{T\rightarrow+\infty}U(\gamma(T))=0$. This follows from the fact that, for $T$ sufficiently large and for $x$ belonging to a compact set, due to the uniform coercivity estimates (cfr. Section \ref{sec:coercivity_estimates}), there is a positive constant $c\in\R$ which does not depend on $T$ and $x$ such that
    \[
    \begin{split}
    U(\gamma(T)) &= \sum_{i<j}\frac{m_i m_j}{|r_{0,ij}(T) + \varphi_{ij}(T)+x_{ij}-r_{0,ij}(1)|}\\
    &\leq \sum_{i<j}\frac{m_i m_j}{|r_{0,ij}(T)| - |\varphi_{ij}(T)|-|x_{ij}-r_{0,ij}(1)|}\\
    &\leq \sum_{i<j}\frac{m_i m_j}{|r_{0,ij}(T)| - c T^{1/2}-|x_{ij}-r_{0,ij}(1)|},
    \end{split}
    \]
    and the last term goes to zero when $T\rightarrow+\infty$.
    \item $\lim_{T\rightarrow+\infty} \langle \ddot{r}_0(T),\gamma(T)\rangle_\mathcal{M}=0$. To prove this, we can study the three cases separately. In the hyperbolic case, $\ddot r_0(T)=0$ for all $T\geq1$. In the parabolic case,
    \[
    \lim_{T\rightarrow+\infty}\langle \ddot{r}_0(T),\gamma(T)\rangle_\mathcal{M} =  \lim_{T\rightarrow+\infty}-\frac{2}{9}\langle \beta b_m T^{-4/3}, \beta b_m T^{2/3}+ \varphi(T) + x - \beta b_m\rangle_\mathcal{M}=\lim_{T\rightarrow+\infty}-\frac{2}{9}\frac{1}{T^2}=0,
    \]
    and in the hyperbolic-parabolic case,
    \[
    \lim_{T\rightarrow+\infty}\langle \ddot{r}_0(T),\gamma(T)\rangle_\mathcal{M} = \lim_{T\rightarrow+\infty}-\frac{2}{9}\langle \beta b_m T^{-4/3},aT + \beta b_m T^{2/3} + \varphi(T) + x -a - \beta b_m\rangle_\mathcal{M}=0,
    \]
    since $\lim_{T\rightarrow+\infty}\langle\beta b_m T^{-4/3},\varphi(T)\rangle_\mathcal{M}=0$ by the fact that, for $T$ sufficiently large, there is a positive constant $c'\in\R$ which does not depend on $T$ and $x$ such that $\|\varphi(T)\|_\mathcal{M}\leq c'T^{1/2}$.
\end{itemize}
\end{proof}

\subsubsection{Alternative direct proof of Theorem \ref{thm:HJ}}

    Theorem \ref{thm:HJ} can also be proved with more direct computations, simply by showing that Definition \ref{def:viscosity} is satisfied by the value function. For the sake of completeness, we also show this alternative proof.

    Given $\varphi^x\in\mathcal{D}$ realizing $v(x)$, set $\gamma=r_0+\varphi^x+x-r_0(1)$. Since the conservation of the energy implies
    \[
    H(\gamma(1),\dot \gamma(1))= \frac{\|a\|_\mathcal{M}^2}{2},
    \]
    to prove Theorem \ref{thm:HJ} it is enough to show that $\|\nabla v(x)\|_{\mathcal{M}^{-1}}=\|\dot\gamma(1)\|_\mathcal{M}$, in the viscosity sense, where $\gamma(t)$ is an arbitrary expansive motion starting at $x$.

\smallskip
Let $x \in \Omega$ and $z \in \mathcal{X}$. For $h>0$, let $x_h = x+hz$, and $\varphi^h \in \mathcal{D}$ be a minimizer of $\mathcal{A}_{x_h}$. It holds, $v(x_h)=\mathcal{A}_{x_h}(\varphi^h)-\langle a, x_h\rangle_\mathcal{M}$. Since $\varphi^h \in \mathcal{D}$, we have
\[
v(x_h)-v(x) \ge \mathcal{A}_{x_h}(\varphi^h)-\mathcal{A}_x(\varphi^h) - \langle a, x_h-x\rangle_\mathcal{M}.
\]
Hence, since $r_0(t)+\varphi^h(t)+x-r_0(1)+s h z $ is free from collisions, due to the minimality of $\varphi^h$, we have
\[
\begin{split} 
v(x_h)-v(x) &\ge \int_{1}^{+\infty} U(r_0(t)+\varphi^h(t)+x_h -r_0(1))-U(r_0(t)+\varphi^h(t)+x-r_0(1)) \, \ud t -h\langle a, z\rangle_\mathcal{M} 
\\
&=\int_{1}^{+\infty} h \int_0^1 \langle \nabla U(r_0(t)+\varphi^h(t)+x-r_0(1)+shz),z\rangle_\mathcal{M} \, \ud s\, \ud t - h\langle a, z\rangle_\mathcal{M},
\end{split}
\]
that is
\[
\frac{v(x_h)-v(x)}{h} \ge \int_{1}^{+\infty} \int_0^1 \langle \nabla U(r_0(t)+\varphi^h(t)+x-r_0(1)+shz),z\rangle \, \ud s\, \ud t - \langle a, z\rangle_\mathcal{M}.
\]

As we observed in the proof of the continuity of $v$ (see Proposition \ref{prop:conv_v_T}), there exists $\bar\varphi$ minimizer of $\mathcal{A}_x$ such that $\varphi^h \to \bar \varphi$ pointwise on $[1,+\infty)$, uniformly on $[1,T]$, for every $T>1$, and weakly on $\mathcal{D}$. In particular, we have $\varphi^h +shz \to \bar \varphi$ pointwise on $[1,+\infty)$ and uniformly on every $[1,T]$. It then suffices to observe that, for every $\varepsilon>0$,  there exists $T_\varepsilon>1$ such that
\[
\left|\langle \nabla U(r_0(t)+\varphi^h(t)+x-r_0(1)+shz),z\rangle\right| \le \frac{\|z\|_\mathcal{M}}{(1-\varepsilon)\|r_0(t)\|_\mathcal{M}^2}, \ \mbox{ for } \ t \ge T_\varepsilon,
\]
to apply dominated convergence Theorem and conclude that, as $h\to0^+$, 
\begin{equation} 
\label{eq:diff_v_1}
\begin{split} 
\frac{v(x_h)-v(x)}{h} &\ge \int_{1}^{+\infty} \int_0^1 \langle \nabla U(r_0(t)+\bar\varphi(t)+x-r_0(1)), z\rangle\,\ud s\,\ud t - \langle a,z \rangle_\mathcal{M} +o(1)
\\
&
=\int_{1}^{+\infty} \langle \nabla U(r_0(t)+\bar\varphi(t)+x-r_0(1)), z\rangle \,\ud t - \langle a,z \rangle_\mathcal{M} +o(1)
\\
&=\int_{1}^{+\infty} \langle \ddot \gamma(t), z\rangle_\mathcal{M}\,\ud t +\langle a,z\rangle_\mathcal{M} +o(1)
\\
&=-\langle \dot\gamma(1),z\rangle_\mathcal{M} +o(1),
\end{split}
\end{equation}
where $\gamma(t)=r_0(t)+\bar\varphi(t)+x-r_0(1)$. Therefore, if $p \in D^+v(x)$, then the following chain holds:
\[
\langle p,z\rangle_\mathcal{M}  \ge \limsup_{h \to 0^+} \frac{v(x_h)-v(x)}{h} \ge -\langle \dot \gamma(1), z\rangle_\mathcal{M}, \ \ \mbox{ for every } \ z \in \mathcal{X}.
\]
By choosing $z= -p$, we then obtain
\begin{equation}
    \label{eq:one}
-\|p\|_\mathcal{M}^2 \ge \langle \dot \gamma(1), p\rangle_\mathcal{M},
\end{equation}

Instead, by choosing $z= \dot \gamma(1)$, we have
\begin{equation} 
\label{eq:two}
\langle \dot \gamma(1),p\rangle_\mathcal{M} \ge -\|\dot \gamma(1)\|_\mathcal{M}.
\end{equation}
Putting together \eqref{eq:one} and \eqref{eq:two}, we conclude that
\begin{equation} 
\label{eq:three}
\|p\|_\mathcal{M}^2 \le \|\dot\gamma(1)\|_\mathcal{M}^2=\|a\|_\mathcal{M}^2+2U(\gamma(1))
\end{equation}
by the conservation of the energy and Theorem \ref{thm_partially_hyperbolic}.

\smallskip
On the other hand, if $\varphi^x$ is a minimizer of $\mathcal{A}_x$, so that $v(x)=\mathcal{A}_x(\varphi^x)-\langle a,x \rangle_\mathcal{M}$, we have
\[
v(x_h)-v(x) \le \mathcal{A}_{x_h}(\varphi^x)-\mathcal{A}_x(\varphi^x) - h\langle a, z\rangle_\mathcal{M}.
\]
With similar arguments as above, we then get
\[
\begin{split} 
\frac{v(x_h)-v(x)}{h} & \le \int_{1}^{+\infty}  U(r_0(t)+\varphi^x(t)+x_h -r_0(1))-U(r_0(t)+\varphi^x(t)+x-r_0(1))\, \ud t -\langle a,z\rangle_\mathcal{M}
\\
&=\int_{1}^{+\infty} \int_0^1 \langle \nabla U(r_0(t)+ \varphi^x(t) + x+shz-r_0(1)),z\rangle\, \ud t -\langle a, z\rangle_\mathcal{M} 
\\
& \le \int_{1}^{+\infty} \langle\nabla U(r_0(t)+ \varphi^x(t)+x-r_0(1)),z\rangle\, \ud t - \langle a, z\rangle_\mathcal{M} +o(1), \ \mbox{ as } \ h \to 0^+
\\
&= \int_{1}^{+\infty} \langle \ddot \gamma^x(t), z\rangle \, \ud t -\langle a,z\rangle_\mathcal{M} +o(1),
\end{split}
\]
that is,
\begin{equation}
    \label{eq:diff_v_2}
    \frac{v(x_h)-v(x)}{h} \le -\langle \dot \gamma^x(1), z\rangle_\mathcal{M} + o(1),\ \mbox{ as }h\to0^+.
\end{equation}

Hence, if $p \in D^-v(x)$, then
\begin{equation}
    \label{eq:four}
\langle p, z\rangle_\mathcal{M} \le \liminf_{h \to 0^+} \frac{v(x_h)-v(x)}{h} \le  -\langle \dot \gamma^x(1), z\rangle_\mathcal{M}, \ \mbox{ for every } \ z \in \mathcal{X}.
\end{equation}

As in the arguments to obtain \eqref{eq:three}, by putting together the inequalities from choosing $z=-p$ and $z=\dot \gamma^x(1)$ in \eqref{eq:four}, we get
\[
\|p\|_\mathcal{M}^2\ge \|\dot\gamma^x(1)\|_\mathcal{M}^2= \|a\|_\mathcal{M}^2+2U(\gamma^x(1)),
\]
by the conservation of the energy and Theorem \ref{thm_partially_hyperbolic}.

\smallskip
To conclude, we proved that, in the viscosity sense,
\[
\|\nabla_{\mathcal M} v(x)\|_\mathcal{M}^2 = \|a\|_\mathcal{M}^2+2U(x),
\]
that is
\begin{equation}
\label{eq:HJ-v}
H(x,\nabla v(x)) = \frac{\|a\|_\mathcal{M}^2}{2},
\end{equation}
thanks to $\|\nabla v(x)\|_{\mathcal M^{-1}}^2=
\|\nabla_{\mathcal M} v(x)\|_\mathcal{M}^2$.

\begin{rem}
\label{rem:regular_points}
    {
\rm
\emph{En passant}, in the proof of \eqref{eq:HJ-v}, we also proved that if for $x \in \Omega$ there is a unique $\varphi^x$ minimizing $\mathcal{A}_x$ on $\mathcal D$, then $v$ is differentiable at $x$, and 
\[
\nabla v(x)=-\mathcal{M}\dot \gamma^x(1),
\]
where $\gamma^x(t)=r_0(t)+\varphi^x(t)+x-r_0(1)$. 

Indeed, \eqref{eq:diff_v_1} and \eqref{eq:diff_v_2} together, in this case, imply
\[
-\langle \dot \gamma^x(1), z\rangle_\mathcal{M} \le \lim_{h\to0^+}\frac{v(x+hz)-v(x)}{h} \le -\langle \dot \gamma^x(1),z\rangle_\mathcal{M}, \ \mbox{ for every } \ z \in \mathcal{X}.
\]
    }
\end{rem}

\section{Fine regularity results on $v$}

In this section, which concerns regularity results on the value function $v(x)$, we treat the \emph{size} of two special sets of points: the set of irregular points and the set of the so-called conjugate points.

\subsection{Irregular and conjugate points}\label{reg_irregular}

Below, we introduce the definitions that we adopt for irregular and conjugate points associated with the variational problem defining $v(x)$. These notions are inspired by Definitions \ref{def:CS_irregular_points} and \ref{def:CS_conjugate_points}, suitably adapted to the present non-compact setting, where both the spatial diameter and the time interval are unbounded.

It is worth mentioning that these definitions are closely related to classical notions arising in the context of (smooth) Riemannian geometry. In that framework, the notions of {\em cut locus} and {\em conjugate locus} are classical and well studied (see, e.g., the books \cite{Bergerbook, Gallot_etal} and the references therein). More precisely, given a point $p$ in a Riemannian manifold, the cut locus of $p$ is defined as the set of endpoints of geodesics issuing from $p$ at which the minimizing property ceases to hold. Such a loss of minimality may occur for two distinct reasons: either because multiple minimizing geodesics reach the same point (corresponding to a failure of injectivity of the exponential map $\exp_p$), or because the differential $\ud\exp_p$ becomes noninvertible, giving rise to conjugate points. These notions are intimately related to the regularity properties of the distance function from $p$, viewed as the value function associated with the minimization of the metric length functional.

With this analogy in mind, we introduce Definitions \ref{def:irregular} and \ref{def:conjugate}.

\begin{defn}[Irregular points]\label{def:irregular}
    We say that a configuration $x\in\Omega$ is \emph{regular} if $\mathcal{A}_x$ admits a \emph{unique} minimum $\varphi^x$ on $\mathcal{D}$. All other points are called \emph{irregular}. We denote by $\Sigma$ the set of irregular points.
\end{defn}
As observed in Remark \ref{rem:regular_points}, the set $\Omega \setminus \Sigma$ corresponds to points in which $v$ is $C^1$, and the Hamilton-Jacobi equation \eqref{eq:HJ_v} is satisfied in the classical sense.

\medskip
Retracing the work by Cannarsa and Sinestrari \cite{MR2041617}, we want to define the set of \emph{conjugate} points in a similar way to what has been done in \cite{MR2041617}, taking into consideration that we are working with a renormalized Lagrangian action that is defined over a set of curves whose dominium is the half-line $[1,+\infty)$. 

We recall from \cite{PolimeniTerracini} that over the set of non-collisional configurations, the differential $\ud\mathcal{A}_x(\varphi)$ is continuous to the dual space $\mathcal{D}^*$, and it takes the following form:
\begin{equation}
    \label{eq:diffA}
    \ud \mathcal{A}_x(\varphi)[\psi] 
    = \int_{1}^{+\infty} \langle \dot{\varphi}(t),\dot{\psi}(t)\rangle_\mathcal{M} + \langle \nabla U(r_0(t)+\varphi(t)+x-r_0(1)),\psi(t)\rangle - \langle\ddot{r}_0(t),\psi(t)\rangle_\mathcal{M}\ \ud t.
\end{equation}
By composing $\ud \mathcal{A}_x(\varphi)$ with the inverse of Riesz's isomorphism $\mathcal{R}_i^{-1}:\mathcal{D}^* \to \mathcal{D}$, we can define a map $F:U\subset \Omega\times \mathcal{D} \to \mathcal{D}$
as:
\[
\begin{split}
F(x,\varphi)&:=\mathcal{R}_i^{-1}\circ\ud\mathcal{A}_x(\varphi)
\\
&= \varphi+ \mathcal{R}_i^{-1}[\nabla U(r_0 + \varphi + x - r_0(1)) - \mathcal{M}\ddot{r}_0].
\end{split}
\]

Two comments on $F$. First, zeros of $F$ correspond to couples $(\bar x,\bar \varphi)$ in which $\bar\varphi$ is a critical point of $\mathcal{A}_{\bar x}$. Second, if the differential $D_{\varphi}F(\bar x,\bar\varphi)$ is invertible, then the Implicit Function Theorem implies that there is $\mathcal{\bar U}\times \mathcal{\bar V}$ a neighborhood of $(\bar x, \bar\varphi) \in \Omega \times \mathcal{D}$, such that for every $x \in \mathcal{\bar U}$ there is one and only one critical point of $\mathcal{A}_x$ on $\mathcal{\bar V}$, say $\varphi^{x}\in\mathcal{\bar V}$, which depends on $x$ with maximal regularity. 

A sufficient condition to prove that $D_\varphi F(x,\varphi)$ is invertible is to prove that $\ud^2 \mathcal{A}_x(\varphi)$ is coercive. In contrast, the set of conjugate points heuristically consists of points such that the linearized problem degenerates. Hence, the following definition.

\begin{defn}[Conjugate points]\label{def:conjugate}
    We define the set of \emph{conjugate points} $\Gamma$ as the subset of points $x \in \Omega$ such that 
        \begin{itemize}
        \item the set of minimizers of $\mathcal A_x$ is isolated on the set of critical points,
            \item $\ud^2\mathcal A_{x}(\varphi^x) \in \mathcal{L}(\mathcal{D}; \mathcal{D}^*)$ is not invertible, for every $\varphi^x$ minimizer of $\mathcal A_x(\varphi^x)$,
        \end{itemize}
    where $\ud^2  \mathcal{A}_x$ is the bilinear form corresponding to $D_\varphi F(x,\varphi^x)$, given by
    \begin{equation}
        \label{eq:hessian_A}
    \ud^2 \mathcal{A}_x(\varphi^x)[\psi,\zeta] = \int_{1}^{+\infty} \langle\dot{\psi}(t),\dot{\zeta}(t)\rangle_\mathcal{M} + \langle\nabla^2 U(r_0(t) + \varphi^x(t) + x - r_0(1))\psi(t),\zeta(t)\rangle_\mathcal{M}\ \ud t.
    \end{equation}
\end{defn}

Our goal is to obtain an estimate on the dimension of the set of conjugate points $\Gamma$.

\subsection{Rectifiability of $\Sigma \setminus \Gamma$}

In this and the next sections, we make use of the following right-tail $T$-actions:
\begin{equation}
\label{eq:def_right_tale_func}
\tilde{\mathcal{A}}_{x}^{T}(\varphi) = \int_{T}^{+\infty} \frac{1}{2}\|\dot{\varphi}(t)\|_\mathcal{M}^2 + U(r_0(t) + \varphi(t) + x - r_0(T)) - U(r_0(t)) - \langle\ddot{r}_0(t),\varphi(t)\rangle_\mathcal{M}\ \ud t,
\end{equation}
where  $\varphi \in \tilde{ \mathcal D}_T$, where 
\[
\tilde{\mathcal D}_T = \{\varphi\in H^{1,2}([T,+\infty))\ :\ \varphi(T)=0,\ \int_{T}^{\infty}\|\dot\varphi(t)\|^2_\mathcal{M}\ \ud t<+\infty\},
\]
with $\|\varphi\|_{\tilde{\mathcal D}_T}$ defined as expected as the $L^2(T,+\infty)$ norm, with respect to $\mathcal M$, of $\dot \varphi$.

\begin{prop}
\label{prop:dA}
    For all $T>1$, and for any $x\in\Omega$ and $\varphi\in \tilde{\mathcal D}_T$, the differential $\ud \tilde{\mathcal A}_{x}^{T}(\varphi)$ is a compact perturbation of an invertible operator.
\end{prop}
\begin{proof}
We observe that $\ud \tilde{\mathcal A}_{x}^{T}$ is as in  \eqref{eq:diffA}, but with the integral evaluated in $[T,+\infty)$. Using Riesz isomorphism, we can write
\[
\ud\tilde{\mathcal{A}}_{x}^{T}(\varphi) = \mathbf{1} + \mathcal{R}_i^{-1}[\nabla U(r_0 + \varphi + x - r_0(1)) - \mathcal{M}\ddot{r}_0].
\]
We will divide the proof accordingly to the class of expansive motion under consideration.

\smallskip
{\bf Hyperbolic case.} The statement follows once one observes that $P_{x,T}^H\in \mathcal{L}(\tilde{\mathcal D}_T; {\tilde{\mathcal D}_T}^*)$, defined by
\[
P_{x,T}^H(\varphi)\ :\ \psi \rightarrow \int_{T}^{+\infty} \langle \nabla U(\gamma(t)),\psi(t)\rangle\ \ud t,
\]
is compact. Consider $(\varphi_n)_n$ a bounded sequence in $\tilde{\mathcal D}_T$. As we already observed throughout the paper, up to subsequences, $\varphi_n \to \bar\varphi$, for some $\bar\varphi \in \tilde{\mathcal D}_T$, where the convergence is, in particular, pointwise in the whole half-line $[T,+\infty)$. Hence, for every fixed $\psi \in \tilde{\mathcal D}_T$  such that $\|\psi\|_{\tilde{\mathcal D}_T} = 1$, we have
\[
\begin{split}
    P_{x,T}^H(\varphi_n)[\psi]-P_{x,T}^H(\bar\varphi)[\psi] & = \int_T^{+\infty} \langle\nabla U(\gamma_n(t)-\nabla U(\bar\gamma(t)),\psi(t)\rangle\, \ud t
    \\
    &=\int_T^{+\infty} \int_0^1 \langle \nabla^2 U(\gamma_n(t) +s(\varphi_n(t)-\bar\varphi(t)))(\varphi_n(t)-\bar\varphi(t)),\psi(t)\rangle\,\ud s\,\ud t.
\end{split}
\]
Since, for $t$ large enough and for every $s \in [0,1]$,
\[
\left| \langle \nabla^2 U(\gamma_n(t) +s(\varphi_n(t)-\bar\varphi(t)))(\varphi_n(t)-\bar\varphi(t)),\psi(t)\rangle\right| \le \frac{C\,\|\varphi_n(t)-\bar\varphi(t)\|\,\|\psi(t)\|}{t^3} \le \frac{C'}{t^{5/2}}\in L^1(T,+\infty),
\]
by the Dominated Convergence Theorem and the pointwise convergence of $\varphi_n$ to $\bar \varphi$, we deduce that
\[
\|P_{x,T}^H(\varphi_n)-P^H_{x,T}(\bar\varphi)\|_{\tilde{\mathcal D}_T^*}=
\sup_{\|\psi\|=1}\left|P_{x,T}^H(\varphi_n)[\psi]-P_{x,T}(\bar\varphi)[\psi]\right| \to 0, \ \mbox{ as } \ n\to \infty,
\]
that is, $P_{x,T}^H(\varphi_n)$ is convergent in $\tilde{\mathcal D}_T^*$, which means that $P_{x,T}^H$ is compact.

\smallskip
{\bf Parabolic case.} In this case, we need to start from $\tilde{\mathcal A}_{x}^T(\varphi)$ written as follows
\[
\begin{split}
\tilde{\mathcal A}_{x}^T(\varphi)&=\int_{T}^{+\infty} \frac{\|\dot\varphi(t)\|_\mathcal{M}^2}{2}+\langle \nabla^2 U(r_0(t))\varphi(t),\,\varphi(t)\rangle \,\ud t 
\\
&\quad \quad +\int_T^{+\infty} U(\gamma(t))-U(r_0(t))-\langle \nabla U(r_0(t)),\varphi(t)\rangle -\langle \nabla^2 U(r_0(t)) \varphi(t),\,\varphi(t)\rangle\,\ud t,
\end{split}
\]
where 
\[
\varphi \mapsto
\int_{T}^{+\infty} \frac{\|\dot\varphi(t)\|_\mathcal{M}^2}{2}+\langle \nabla^2 U(r_0(t))\varphi(t),\,\varphi(t)\rangle \,\ud t 
\]
is a norm, coming from a positive quadratic form which is equivalent to $\|\cdot\|_{\tilde{\mathcal D}_T}$, as shown in \cite[Section 4.2]{PolimeniTerracini}. Hence, by Lax-Milgram Theorem, its differential is invertible. Setting
\[
\mathcal{F}(\varphi) = \int_T^{+\infty} U(\gamma(t))-U(r_0(t))-\langle \nabla U(r_0(t)),\varphi(t)\rangle -\langle \nabla^2 U(r_0(t)) \varphi(t),\,\varphi(t)\rangle\,\ud t,
\]
it holds 
\[
\begin{split}
\ud \mathcal{F}(\varphi)[\psi] &= \int_T^{+\infty} \langle \nabla U(\gamma(t))-\nabla U(r_0(t)),\psi(t)\rangle-\langle \nabla^2 U(r_0(t))\varphi(t),\,\psi(t)\rangle\,\ud t
\\
&=\int_T^{+\infty} \int_0^1 \langle \nabla^2 U(\gamma(t)+\sigma_1(\varphi(t)+\tilde x))(\varphi(t)+\tilde x),\psi(t)\rangle 
-
\langle \nabla^2 U(r_0(t))\varphi(t),\,\psi(t)\rangle \, \ud t
\\
&=\int_T^{+\infty} \int_0^1 \int_0^1 \langle \left[\nabla^3 U(\gamma(t)+(\sigma_1+\sigma_2+\sigma_1\sigma_2)(\varphi(t)+\tilde x))(\varphi(t)+\tilde x+\sigma_1(\varphi(t)+\tilde x)\right]\,\varphi(t),\,\psi(t)\rangle \, \ud\sigma_1\ \ud \sigma_2\ \ud t
\\
&\qquad \qquad + \mbox{ lower order terms }.
\end{split}
\]
Since, for $t$ sufficiently large, we have
\[
\left| \langle \left[\nabla^3 U(\gamma(t)+(\sigma_1+\sigma_2+\sigma_1\sigma_2)(\varphi(t)+\tilde x))(\varphi(t)+\tilde x+\sigma_1(\varphi(t)+\tilde x)\right]\,\varphi(t),\,\psi(t)\rangle \right|
 \le \frac{C\,\|\varphi(t)\|^2\,\|\psi(t)\|}{\|r_0(t)\|^4},
\]
with the same arguments used in the proof of the hyperbolic case for the operator $P_{x,T}^H$, we conclude that $\ud \mathcal{F} \in \mathcal{L}(\tilde{\mathcal D}_T;\tilde{\mathcal D}_T^*)$ is compact.

\smallskip
{\bf Hyperbolic-Parabolic case.} The proof exploits the Lagrangian action decomposition that follows from the cluster partition of the bodies that has been presented in \cite{PolimeniTerracini}. In particular, the renormalized Lagrangian action is written as the sum of two terms, where one describes the movement of the bodies inside each cluster and the other describes the motions of couples of bodies that belong to different clusters. Since the first term describes a parabolic motion and the second describes a hyperbolic one, the computation will be similar to the two cases above.
\end{proof}

The following is the companion of Theorem 6.4.9 in \cite{MR2041617}.

\begin{thm}
\label{th:finiteness of d*}
If the set $\{\varphi\in\mathcal{D}:\, \varphi \ \mbox{ is a minimizer of the renormalized value function v(x)}\}$ is not finite, then $x\in \Gamma$.  
\end{thm}

\begin{proof}
    Assuming $x \notin \Gamma$ and $\varphi\in\mathcal{D}$, we have that $\ud^2 \mathcal A_x$ is invertible, that is, for every $J \in \mathcal{D}^*$ there exists a unique $\psi \in \mathcal{D}$ such that $J[\cdot]=\ud^2 \mathcal{A}_x(\varphi)[\psi,\cdot]$ in $\mathcal{D}^*$. This implies that the Implicit Function Theorem applies to the $0-$level of 
    \[     F(x,\varphi) := \ud\,\mathcal A_x(\varphi) \in \mathcal{D}^*,     \]
    and, therefore, there exists a neighborhood $\mathcal U_x\times\mathcal U_\varphi$ of $(x,\varphi)$ in $\Omega\times\mathcal D$ such that, for every $z \in  \mathcal U_x$, there exists a unique $\bar\varphi = \bar\varphi^z = \bar\varphi(z) \in \mathcal U_\varphi$ critical point of $\mathcal A_z$. Thus, there are at most countably many minimizers, and they are isolated. 

    Suppose that the number of such minimizers is not finite, and let $(\varphi_n)_n$ be the sequence of minimizers. Then, up to subsequences, there exists $\bar\varphi \in \mathcal{D}$ such that $\varphi_n \to \bar\varphi$ weakly and pointwise. By Proposition \ref{prop:dA}, we can conclude that $\bar\varphi$ must be a critical point and, in particular, it is a minimizer. It follows that 
    \[
    0=\|\varphi_n\|_{\mathcal{D}}-\|\bar\varphi\|_{\mathcal{D}}+2\,\int_1^{+\infty} \int_0^1 \langle \nabla U(r_0(t)+\varphi_n(t)+x-r_0(1)+s(\bar\varphi(t)-\varphi_n(t)), \bar\varphi(t)-\varphi_n(t)\rangle\,\ud s\, \ud t
    \]
    and then, by the Dominated Convergence Theorem, we obtain $\|\varphi_n\|_{\mathcal{D}} \to \|\bar\varphi\|_{\mathcal{D}}$. This in turn implies $\|\varphi_n - \bar\varphi\|_{\mathcal{D}}\to 0$, which contradicts the fact that $(\varphi_n)_n$ is discrete in $\mathcal{D}$.

\end{proof}

\begin{thm}\label{thm:finite_number_vi}
Let $x \in \Sigma \setminus \Gamma$ be given. Then there exists a neighborhood $\mathcal U_{x}$ of $x$ and a finite number of $v_1, \dots, v_k : \mathcal U_{x} \to \mathbb R$ such that $\nabla v_i \neq \nabla v_j$ if $i \neq j$ and $u=\min\{v_i\}_i$ in $\mathcal U_{x}$.
\end{thm}
\begin{proof}
Let $\varphi_i$, for $i=1, \dots, k$, be the minimizers for $x \in \Sigma \setminus \Gamma$. We define
\begin{equation}\label{eq:min_v_i}
v_i(z) = \min_{\xi \, \in \,  \mathcal U_{\varphi_i} \subset \mathcal{D}} \mathcal A_z(\xi), \ \mbox{ for } \ z \in \mathcal U_x \ \mbox{ for every } \ i =1, \dots, k.
\end{equation}

We observe that every $z \in \mathcal U_x$ is regular for $\mathcal A_z|_{\mathcal U_{\varphi_i}}$, and that (up to take a smaller neighborhood of $x$ instead of $\mathcal U_x$) $z \notin \Gamma$ (since $\Gamma$ is closed). Then $v_i \in C^1(\mathcal U_x)$. 

Defining the functions $v_i$, $i=1,...,k$, as in \eqref{eq:min_v_i}, the proof follows directly from Theorem \ref{th:finiteness of d*} and IFT applied to $F$.
\end{proof}

Recall that, since we work in a space of configurations with fixed center of mass, $\dim\mathcal{X}=d(N-1)$. We have $v_i(x) = v_j(x)$, and, since $\nabla v_i(x) = -\mathcal M\, \dot \varphi_i(1) \neq -\mathcal M\, \dot \varphi_j(1) = \nabla v_j(x)$, the set $\{y:\, v_i(y)=v_j(y)\}$ is a $(d(N-1)-1)$-dimensional $C^\infty$ hypersurface for any pair $i\neq j$. Then, we also have the following corollary.

\begin{cor}
Let $ x \in \Sigma \setminus \Gamma$ be given. Then, there exists $\mathcal U_x$, a neighborhood of $x$, such that $\mathcal U_x \cap \Sigma$ is contained in a finite union of $(d(N-1)-1)$-dimensional hypersurfaces of class $C^\infty$.
\end{cor}

\subsection{Bound on the dimension of the set of conjugate points}

\subsubsection{Definition of a local flow}

Given $x_0 \in \Gamma$, consider $\varphi^{x_0}$ a corresponding minimizer of $\mathcal A_{x_0}$ on $\mathcal D$. Associated to $\varphi^{x_0}$ is $\gamma^{x_0}$, defined by $\gamma^{x_0}(t)=r_0(t)+\varphi^{x_0}(t)+x_0-r_0(1)$. Extend $\gamma^{x_0}$ to the interval $(1-\varepsilon, 1]$ as the unique smooth solution of \eqref{eq_newton}. Still denote by $\gamma^{x_0}$ this extension, defined in $(1-\varepsilon, +\infty)$. 

Let $z_0=\gamma^{x_0}(T)$. The Hessian of the action is defined as the bilinear form $\ud^2 \mathcal{A}_{x_0}(\varphi)$ given in \eqref{eq:hessian_A}. Following $\varphi(t)$, and consequently $\gamma(t)$, we consider a set of bilinear forms on $\tilde{\mathcal D}_T\times\tilde{\mathcal D}_T$
\[
\begin{split} 
\ud^2 \tilde{\mathcal{A}}_{z_0}^{T}(\tilde\varphi)[\psi,\eta]&=\int_T^{+\infty} \langle\dot\psi(t),\dot \eta(t)\rangle_\mathcal{M}+\langle\nabla^2 U(r_0(t)+\tilde\varphi(t)+\gamma^{x_0}(T)-r_0(T))\psi(t),\,\eta(t)\rangle \,\ud t
\\
&=\int_T^{+\infty} \langle\dot\psi(t),\dot \eta(t)\rangle_\mathcal{M}+\langle \nabla^2 U(r_0(t)+\varphi^{x_0}(t)+x_0-r_0(1))\psi(t),\,\eta(t)\rangle\,\ud t,
\end{split}
\]
where $\tilde\varphi = \varphi^{x_0}|_{[T,+\infty)}-\varphi^{x_0}(T)$. Notice that $\ud^2 \mathcal{A}_{x_0}(\varphi)$ is not invertible, as, by hypothesis, $x\in\Gamma$.
\begin{rem}
\label{rem:T_0}
    We recall that $\ud^2 \mathcal{A}_{x_0}(\varphi)$ is invertible if $\forall\ J\in \mathcal{D}^*$ there is a unique $\psi\in\mathcal{D}$ such that $\ud^2 \mathcal{A}_{x_0}(\varphi)[\psi,\cdot]=J[\cdot]$ in $\mathcal{D}^*$ (or, using Riesz's isomorphism, $\mathcal{R}_i^{-1}\circ \ud^2 \mathcal{A}_{x_0}(\varphi)[\psi,\cdot]=\mathcal{R}_i^{-1}\circ J[\cdot]$ in $\mathcal{D}$).
\end{rem}

By Lax-Milgram's Theorem, $\ud^2 \tilde{\mathcal{A}}_{z_0}^{T}(\tilde\varphi)$ is invertible if it is coercive, that is, if there is a constant $\alpha>0$ such that
\begin{equation}\label{coercivity_LxT}
   \ud^2 \tilde{\mathcal{A}}_{z_0}^T(\tilde\varphi)[\psi,\psi] \geq \alpha\|\psi\|^2_{\tilde{\mathcal D}_T},    
\end{equation}
for all $\psi\in\tilde{\mathcal{D}}_T$.

Now, we give the following lemma.
\begin{lem}
    \label{lem:T_hess_A_coerc}
    Given $x_0 \in \Gamma$ there exist ${T}_0>1$ and $C>0$ such that \eqref{coercivity_LxT} holds for every $\psi\in\tilde{\mathcal{D}}_{T_0}$.
\end{lem}
\begin{proof}
    We distinguish the hyperbolic case from the parabolic one to prove the claim.

    Starting with the hyperbolic case, we notice that
\[
\langle \nabla^2 U(at)\psi(t),\psi(t)\rangle\geq -C\frac{\|\psi(t)\|_\mathcal{M}^2}{\|\gamma^{x_0}(t)\|_\mathcal{M}^3} \geq - \frac{C}{\|a\|_\mathcal{M}^3 T}\frac{\|\psi(t)\|_\mathcal{M}^2}{t^2},
\]
for some constant $C>0$ and for $t\geq T$. Then
\[
\int_{T}^{+\infty} \|\dot{\psi}(t)\|_\mathcal{M}^2 + \langle \nabla^2 U(\gamma^{x_0}(t))\psi(t),\psi(t)\rangle\ \ud t \geq \int_{T}^{+\infty} \|\dot{\psi}(t)\|_\mathcal{M}^2 - \frac{C'}{ T}\frac{\|\psi(t)\|_\mathcal{M}^2}{t^2}\ \ud t \geq \bigg(1-\frac{4C'}{T}\bigg)\int_{T}^{+\infty}\|\dot{\psi}(t)\|_\mathcal{M}^2\ \ud t.
\]
Choosing $T$ big enough so that $1-\frac{4C'}{T}>0$, we obtain the coercivity of $L_{x_0,T}$. 

To prove \eqref{coercivity_LxT} in the parabolic case, we use the inequality
\[
\langle \nabla^2 U(\gamma^{x_0}(t))\psi(t),\psi(t)\rangle \geq -U_{min}\frac{\|\psi(t)\|_\mathcal{M}^2}{\|\gamma^{x_0}(t)\|_\mathcal{M}^3}, 
\]
which implies
\[
\begin{split}
    \langle \nabla^2 U(\gamma^{x_0}(t))\psi(t),\psi(t)\rangle &\geq -U_{min}\frac{\|\psi(t)\|_\mathcal{M}^2}{\|r_0(t)\|_\mathcal{M}^3}\frac{1}{\big(1+\frac{\|\varphi(t)\|_\mathcal{M}}{\|r_0(t)\|_\mathcal{M}} + \frac{\|x - r_0(1)\|_\mathcal{M}}{\|r_0(t)\|_\mathcal{M}}\big)^3}\\
    & \geq -U_{min}\frac{\|\psi(t)\|_\mathcal{M}^2}{\|\beta b_m t^{2/3}\|_\mathcal{M}^3}\\
    & = -\frac{2}{9}\frac{\|\psi(t)\|_\mathcal{M}^2}{t^2},
\end{split}
\]
being $r_0(t) = \beta b_m t^{2/3}$. Then,
\[
\begin{split}
    \ud^2 \tilde{\mathcal{A}_{\gamma^{x_0}(T)}^T}(\tilde\varphi)[\psi,\psi] &\geq \|\psi\|_{\tilde{\mathcal D}_T}^2 - \frac{2}{9}\int_{T}^{+\infty}\frac{\|\psi(t)\|_\mathcal{M}^2}{t^2}\ \ud t\\
    &\geq \|\psi\|_{\tilde{\mathcal D}_T}^2 - \frac{8}{9} \|\psi\|_{\tilde{\mathcal D}_T}^2\\
    &= \frac{1}{9} \|\psi\|_{\tilde{\mathcal D}_T}^2.
    \end{split}
\]
\end{proof}

Let $T_0>1$ be a time such that $\ud^2 \tilde{\mathcal{A}}_{\gamma^{x_0}(T_0)}^{T_0}$ is coercive, as proved in Lemma \ref{lem:T_hess_A_coerc}. Denote $z_0=\gamma^x(T_0)$. By the Implicit Function Theorem, there are two neighborhoods $\mathcal{U}_{z_0}\subset\R^{d(N-1)}$ and $\mathcal U_{\varphi^{z_0}}\subset\tilde{\mathcal {D}}_{T_0}$, and there is a function $\varphi:\mathcal{U}_{x_0}\rightarrow \mathcal U_{\varphi^{z_0}}$ such that $\varphi(z)=\varphi^z$ is the unique critical point of the functional $\tilde{\mathcal A}_{z}^{T_0}$ in $\tilde{\mathcal{D}}_{T_0}$, that turns out to be its minimum. Moreover, $z \mapsto \varphi^z$ is smooth, and the function $u_{T_0}(z)$, defined by
\[
u_{T_0}(z)=\tilde{\mathcal{A}}_{z}^{T_0}(\varphi^z),
\]
is smooth and such that 
$\nabla u_{T_0}(z)=-\mathcal M\,\dot \varphi^z(T_0)$. 

Up to consider a ball $B_R=B_{R}(z_0) \subset \mathcal U_{z_0}$, by continuity, there exists $\varepsilon>0$ such that it is well defined 
$\xi=\xi(z,t)$, for every $z \in  B_R$, $t\in (1-\varepsilon, T_0]$, as the unique solution of
\begin{equation}\label{eq:system_xi}
\begin{cases}
    \ddot\xi = \nabla U(\xi),\quad 1-\varepsilon<t \le T_0\\
    \xi(z,T_0) = z \in {B}_R\\
    \dot\xi(z,T_0) = -\mathcal{M}^{-1}\,\nabla u_{T_0}(z) + \dot r_0(T_0).
\end{cases}
\end{equation}
We briefly comment on the existence of $\varepsilon$. If $z=z_0$,  then $\xi(z_0,1)=x_0 \notin \Delta$. Hence, there exists $\varepsilon>0$ such that we can extend $\xi$ backward up to the time $1-\varepsilon$. By continuity, it is possible to extend the arguments in a neighborhood of $x_0$.

\smallskip
Fix $x$ such that there exist $z \in B_R$ and $s \in (1-\varepsilon,T_0)$, for which $\xi(z,s)=x$. By translating by $r_0$ the flow $\xi$, we can define $\varphi(z,t) = \xi(z,t) - r_0(t)$, that satisfies
\begin{equation}
\label{eq:NDsystemT0}
\begin{cases}
\ddot{\varphi} = \nabla U(r_0+ \varphi)-\ddot r_0, \quad 1-\varepsilon < t \le T_0
    \\
    \dot {\varphi}(T_0) = -\mathcal{M}^{-1}\,\nabla u_{T_0}(z)
    \\
    \varphi(s) = x- r_0(s).
\end{cases}
\end{equation}
Clearly, $\varphi(z,T_0)=\varphi^z(T_0)=0$, so that $\varphi(z,\cdot)$ given in \eqref{eq:NDsystemT0}
 can be extended as the unique solution of $\eqref{eq:NDsystemT0}_1$ and $\eqref{eq:NDsystemT0}_2$ such that $\varphi(z,T_0)=\varphi^z(T_0)$.
 
It results that $\varphi(z,\cdot)$ is a critical point of 

\begin{equation}
\begin{cases}
    \label{finite-time problem}
    &\mathcal{J}_{x,s} (\varphi)= \displaystyle{\int_{s}^{T_0} \frac{\|\dot\varphi(t)\|_\mathcal{M}^2}{2}+U(r_0(t)+\varphi(t)) - U(r_0(t))-\langle \ddot r_0(t), \varphi(t)\rangle_\mathcal{M}\ \ud t +u_{T_0}(\varphi(T_0)+r_0(T_0))},
    \\
    &\varphi \in H^1(s,T_0),\, \varphi(s)=x-r_0(s), \,\varphi(T_0)+r_0(T_0) \in B_R.
    \end{cases}
\end{equation}

\begin{rem}
\label{rem:uniqueness}
    If $\mathcal{A}_x$ admits a unique minimizer $\varphi^x$, over $\mathcal D$, then $\mathcal J_{x,s}$ admits a unique minimizer, which coincides with $\varphi^x\big|_{(s,T_0)} - x+r_0(s)$.
\end{rem}

\subsubsection{Spectral theory along the flow}

Thanks to the definition of $\xi(z,T)$ and $\varphi(z,T)$ \eqref{eq:system_xi}-\eqref{eq:NDsystemT0}, we can consider, for every $z \in \mathcal B_0$ and $T \in (1-\varepsilon,+\infty)$, the functional 
$\tilde{\mathcal A}_{\xi(z,T)}^{T}$, over $\tilde{\mathcal D}_T$, defined in \eqref{eq:def_right_tale_func}. The Hessian of this functional evaluated at $\tilde\varphi(z,\cdot) = \varphi(z,\cdot)|_{[T,+\infty)}-\varphi(z,T)$ is the bilinear form over $\tilde{\mathcal D}_T\times \tilde{\mathcal D}_T$ given by
\[
\ud^2 \tilde{\mathcal A}_{\xi(z,T)}^{T}(\tilde\varphi(z,\cdot))[\psi,\eta]=\int_T^{+\infty} \langle\dot\psi(s),\dot\eta(s)\rangle_\mathcal{M}^2+\langle \nabla^2 U(\xi(z,T))\psi(s),\,\eta(s)\rangle\,\ud s.
\]

Given $z$ and $T$,  we observe that $\ud^2 \tilde{\mathcal{A}}_{\xi(z,T)}^T(\tilde\varphi(z))[\psi,\psi]$ corresponds to the numerator of the Raileigh quotient
\begin{equation*}
\mathcal R_a(\psi):=
\frac{\int_{T}^{+\infty}\|\dot{\psi}(t)\|_\mathcal{M}^2+\langle \nabla^2 U(\xi(z,t))\psi(t),\psi(t)\rangle\ \ud t}{\int_{T}^{+\infty}\frac{\|\psi(t)\|_\mathcal{M}^2}{t^3}\ \ud t}, \ \ \psi \in \tilde{\mathcal D}_T,
\end{equation*}
of the Sturm-Liouville problem
\begin{equation}
\label{eq:eigenfunction}
\begin{cases}
    -\ddot{\psi}(t) +  \nabla^2 U(\xi(z,t))\psi(t) = -\frac{\lambda}{t^3}\psi(t) \ \mbox{ a.e. in } \ (T,+\infty)
    \\
    \psi \in \tilde{\mathcal D}_T.
\end{cases}
\end{equation}

Fix $z,T$. We already observed, since (roughly speaking) $\nabla^2 U(\xi)\approx \nabla^2 U(r_0)$ for large values of $t$,  that there exists $\hat T >T$ such that $\ud^2 \tilde{\mathcal{A}}_{\xi(z,\hat T)}^{\hat T}(\tilde \varphi(z,\cdot))$ is coercive on $\tilde{\mathcal{D}}_{\hat T}$. Moreover, since $\xi(z,\cdot)$ is far from collisions, there exists $C>0$ such that
\[
\inf_{\eta:\, \|\eta\|=1} \langle \nabla^2 U(\xi(z,t))\eta(t),\,\eta(t)\rangle \ge - \frac{C}{t^3}, \ \mbox{ for every } \ t \in [T, \tilde T].
\]
As a consequence, by  using Lemma \ref{lem:T_hess_A_coerc}, there exist $\mu_0 >0$ and $\tilde\alpha>0$ such that
\[
\begin{split}
\int_T^{+\infty} &\|\dot \psi(t)\|_\mathcal{M}^2+\langle \nabla^2 U(\xi(z,t))\psi(t),\,\psi(t)\rangle + \mu_0 \frac{\|\psi(t)\|_\mathcal{M}^2}{t^3}\, \ud t \\
&= \int_{T}^{\tilde T} \|\dot \psi(t)\|_\mathcal{M}^2+\langle \nabla^2 U(\xi(z,t))\psi(t),\,\psi(t)\rangle + \mu_0 \frac{\|\psi(t)\|_\mathcal{M}^2}{t^3}\, \ud t \\
& \quad + \int_{\tilde T}^{+\infty} \|\dot \psi(t)\|_\mathcal{M}^2+\langle \nabla^2 U(\xi(z,t))\psi(t),\,\psi(t)\rangle + \mu_0 \frac{\|\psi(t)\|_\mathcal{M}^2}{t^3}\, \ud t
\\
&\ge \int_{T}^{\tilde T}  \|\dot \psi(t)\|_\mathcal{M}^2-C\frac{\|\psi(t)\|_\mathcal{M}^2}{t^3} +\mu_0 \frac{\|\psi(t)\|_\mathcal{M}^2}{t^3}\, \ud t  +\alpha\,\int_{\tilde T}^{+\infty}\|\dot \psi(t)\|^2_\mathcal{M}\, \ud t
\\
&\ge \int_{T}^{\tilde T}\|\dot \psi(t)\|^2_\mathcal{M}\, \ud t+\alpha\,\int_{\tilde T}^{+\infty}\|\dot \psi(t)\|^2_\mathcal{M}\, \ud t
\\
&\ge \tilde\alpha \|\psi\|^2_{\tilde{\mathcal D}_T} \ \mbox{ for every } \ \psi \in \tilde{\mathcal D}_T,
\end{split}
\]
where $\alpha$ is given in Lemma \ref{lem:T_hess_A_coerc}.

Hence, for every $f \in L^2((T,+\infty);\ud t/t^3)$ there exists a unique $\psi$ such that
\begin{equation}
\label{eq:coerc_modified}
\ud^2 \tilde{\mathcal{A}}_{\xi(z,T)}^T(\tilde \varphi(z,\cdot))[\psi, \eta] + \int_T^{+\infty} \frac{\langle \mu_0\psi(t), \eta(t)\rangle}{t^3}\,\ud t = \int_T^{+\infty} \frac{\langle f(t), \eta(t)\rangle}{t^3}\,\ud t, \ \mbox{ for every } \ \eta \in \tilde{\mathcal D}_T.
\end{equation}
As a consequence, it results defined an operator $\mathcal{F}:L^2((T,+\infty);\ud t/t^3)\to L^2((T, \infty); \ud t/t^3)$ such that $\mathrm{im}(\mathcal{F})=\tilde{\mathcal D}_T$, and also such that $\mathcal{F}(f)=\psi$ solving \eqref{eq:coerc_modified}. From Lemma \ref{prop:compact_emb}, $\mathcal{F}$ is compact (and positive), and hence standard techniques imply that \eqref{eq:eigenfunction} admits a sequence of eigenvalues
\[
\lambda_1 \le \lambda_2 \le \dots \lambda_n\le \cdots \to +\infty,
\]
such that

\begin{equation}\label{lambda_1}
\lambda_1(\xi(z,T),T) = \min_{\psi\in\tilde{\mathcal{D}}_{T},\ \psi\not\equiv0} \frac{\int_{T}^{+\infty}\|\dot{\psi}(t)\|_\mathcal{M}^2+\langle  \nabla^2 U(\xi(z,t))\psi(t),\psi(t)\rangle\ \ud t}{\int_{T}^{+\infty}\frac{\|\psi(t)\|_\mathcal{M}^2}{t^3}\ \ud t}.
\end{equation}

We define the function $\Lambda$ as:
\[
\Lambda(z,t)= \lambda_1(\xi(z,t),t).
\]

\smallskip
\begin{rem}
By the variational formulation of eigenvalues as minmax, we can write the $k$-th eigenvalue as
\[
\Lambda_k(z,T) = \min_{M\subseteq \tilde{\mathcal D}_T,\ \dim(M)=k} \max_{M\setminus\{0\}} \mathcal{R}_a(\psi),
\]
where $\mathcal{R}_a(\psi)$ is the quotient in \eqref{lambda_1}.
\end{rem}

\begin{rem}
Given $z \in \mathcal B_0$, we have that $t \mapsto \Lambda(z,t)$ is strictly decreasing. Indeed, for every fixed $\delta >0$, if we extend by zero in $(t,t+\delta)$  an eigenfunction associated with $\Lambda(z,t+\delta)$, we obtain a function $\bar \psi$ in $\tilde{\mathcal D}_T$ such that $\mathcal R(\bar \psi)=\Lambda(z,t+\delta)$. Hence, by the characterization of $\Lambda(z,t)$ given by \eqref{lambda_1} we have $\Lambda(z,t)\le \Lambda(z,t+\delta)$. Moreover,
if $\Lambda(z,t+\delta)=\Lambda(z,t)=\lambda$, for some $\delta>0$, one can extend by zero an eigenfunction of \eqref{eq:eigenfunction} associated with $\lambda$ in $\tilde{\mathcal{D}}_{t+\delta}$ to an eigenfunction of \eqref{eq:eigenfunction} associated with $\lambda$ in $\tilde{\mathcal D}_T$. This is impossible since the eigenfunctions of \eqref{eq:eigenfunction} cannot be zero in an open set of their domain, thanks to unique continuation due merely to the continuity of $\nabla^2 U(\xi(z,t))$. Hence, for every $z \in \mathcal B_0$, there exists a unique $t$ such that $\Lambda(z,t)=0$.
\end{rem}

\begin{rem}
    We notice that there exists $\mathcal{U}_x$, a neighborhood of $x$, such that 
    \[
    \Gamma \cap \mathcal{U}_x = \{y \in \mathbb R^n:\,  y = \xi(z,t) , \, \mbox{ for some } \, (t,z) \in (1-\varepsilon, T_0)\times \mathcal{B}_0\, \mbox{ such that } \, \Lambda(z,t)=0 \}.
    \]
    For every $x \in \Gamma \cap \mathcal U_x$, if $(z,t)$ is such that $\xi(z,t) =x$, then $\xi(\xi(z,t-t'),t')=x$. 
\end{rem}

\begin{rem}
For every $\mathcal W$ transversal section to $\xi$ containing $z_0$, we can define 
\[
\xi_\mathcal W (z,t) = \xi(z,t), \ \mbox{ for every } \ z\in \mathcal W, t \in (1-\varepsilon, T_0).
\]
Consequently, we can define 
\[
\Lambda_\mathcal W(z,t)=\lambda_1(\xi_\mathcal W(z,t),t),  \ \mbox{ for every } \ z\in \mathcal W, t \in (1-\varepsilon, T_0).
\]
Again, $t\mapsto \Lambda_\mathcal W(z,t)$ is strictly decreasing. Hence, for every $z \in \mathcal W$, there exists a unique $t$ such that $\Lambda_\mathcal W(z,t)=0$.
\end{rem}

\subsubsection{Case $\Lambda(z_0,1)$ simple}
\label{ssec:simple}

Since $\lambda_1(x_0,1)<\lambda_2(x_0,1)$, then there exists $\delta$ such that $\Lambda(z_0,t)$ is simple (i.e. $\lambda_1(\xi(z_0,t),t)<\lambda_2(\xi(z_0,t),t)$) for $t\in(1-\delta,1+\delta)$. Let $t\mapsto \psi(t,T)$ be the first eigenfunction in \eqref{eq:eigenfunction} associated with $\Lambda(T):=\Lambda(z_0,T)$ such that 
\[
\int_{T}^{+\infty} \frac{\|\dot{\psi}(t,T)\|_\mathcal{M}^2}{t^2}\ \ud t = 1.
\]

Denoting $\psi_T(t,T) := \frac{\partial}{\partial T}\psi(t,T)$, we use \eqref{lambda_1} to compute $\Lambda(T)$:
\[
\begin{split}
    \frac{\ud}{\ud T} \Lambda(T)
&= -\|\dot{\psi}(T,T)\|_\mathcal{M}^2 + \int_{T}^{+\infty} 2\langle \dot{\psi}_T(t,T),\dot{\psi}_T(t,T)\rangle_\mathcal{M} + 2\langle V(t)\psi_T(t,T),\psi_T(t,T)\rangle_\mathcal{M}\ \ud t\\
& = -\|\dot{\psi}(T,T)\|_\mathcal{M}^2 + 2\langle \psi_T(t,T),\dot{\psi}_T(t,T)\rangle_\mathcal{M}\bigg|_{T}^{+\infty} + 2 \int_{T}^{+\infty} -\langle\ddot{\psi}(t,T),\psi_T(t,T)\rangle_\mathcal{M} + \langle V(t)\psi_T(t,T),\psi(t,T)\rangle_\mathcal{M}\ \ud t\\
& = -\|\dot{\psi}(T,T)\|_\mathcal{M}^2 - 2\langle \psi_T(T,T),\dot{\psi}_T(T,T)\rangle_\mathcal{M} + 2 \lambda_1(T)\int_{T}^{+\infty} \frac{\langle\psi(t,T),\psi_T(t,T)\rangle_\mathcal{M}}{t^2}\ \ud t.
\end{split}
\]
From the boundary condition and the fact that we supposed $\int_{T}^{+\infty} \frac{\|\dot{\psi}(t)\|_\mathcal{M}^2}{t^2}\ \ud t \equiv 1\ \forall T$, it follows
\[
\frac{\ud}{\ud T}\int_{T}^{+\infty} \frac{\|\dot{\psi}(t,T)\|_\mathcal{M}^2}{t^2}\ \ud t = 0,
\]
which implies
\[
\int_{T}^{+\infty} \frac{2\langle\psi(t,T),\psi_T(t,T)\rangle_\mathcal{M}}{t^2}\ \ud t = 0.
\]
Besides, from the boundary condition, we have
\[
\frac{\partial}{\partial T}\psi(T,T) = \dot{\psi}(T,T) + \psi_T(T,T) = 0.
\]
We can thus conclude that 
\[
\frac{\ud}{\ud T}\Lambda(T) = -\|\dot{\psi}(T,T)\|_\mathcal{M}^2 - 2\langle \psi_T(T,T),\dot{\psi}(T,T)\rangle_\mathcal{M} = - \|\dot{\psi}(T,T)\|_\mathcal{M}^2 + 2\|\dot{\psi}(T,T)\|_\mathcal{M}^2 = \|\dot{\psi}(T,T)\|_\mathcal{M}^2 > 0.
\]

This implies that we can apply the Implicit Function Theorem to obtain that the function $t(z)$, defined for all $z \in \mathcal U_{z_0}$ such that
\[
\Lambda(z,t)=0 \ \mbox{ if and only if } \ (z,t)=(z,t(z)),
\]
is smooth. The same holds if one considers $\Lambda_\mathcal W$ instead of $\Lambda$, implying that for every transversal section $\mathcal W$, $\Lambda_\mathcal W(z,t) =0$ if and only if $t=t_\mathcal W(z)$ is smooth on a neighborhood of $z_0$ in $\mathcal W$, that is, $\mathcal V_{z_0}=\mathcal U_{z_0}\cap \mathcal W$.

Therefore, $\Gamma' \cap \mathcal{U}_x$ is the image through $\xi_\mathcal W$ of a graph of a smooth function.
Since $\Gamma' = \cup_{k \in \mathbb N} \Gamma_k$, where 
\begin{equation}
    \label{eq:gammaone}
\Gamma_k=\{x=\xi(z,t):\, z \in \mathcal W, \, \lambda_1(\xi(z,t),t) \le \lambda_2(\xi(z,t),t)+1/k\},
\end{equation}
 then  $\Gamma' \cap \mathcal{U}_x$ is countably $\mathcal{H}^{d(N-1)-1}$-rectifiable. Indeed, if we define
 \[
 \hat \xi(z)=\xi(z,t_{z_0}(z)), \ \mbox{ for } \ z \in \mathcal V_{z_0},
 \]
 we have
  \[
 \Gamma'\subset \bigcup_{x_0 \in \Gamma'} \hat\xi(\mathcal V_{z_0})=\bigcup_{k \in \mathbb N} \bigcup_{x_0 \in \Gamma_k} \hat\xi(\mathcal V_{z_0})=\bigcup_{k \in \mathbb N} \hat\xi(\mathcal V_{z_k}).
 \]

\subsubsection{Case $\Lambda$ not simple}
\label{ssec:multiple}

Let $x \in \Gamma$, for which a local flow \eqref{eq:NDsystemT0} is defined. Suppose that for some $z \in B_R$ and $s \in (1-\varepsilon,T_0]$, $\xi(z,s)\in \Gamma$ and  $\lambda_1(z,s)=\dots = \lambda_k(z,s)=0$, with $k \ge 2$. This implies that there exist $w_1, \cdots, w_k$ linearly independent in $\tilde{\mathcal D}_{s}$ such that 
\begin{equation}
    \label{eq:Hessian_zero}
\int_{s}^{+\infty}\|\dot{w_i}(t)\|_\mathcal{M}^2+\langle \nabla^2 U(\xi(z,t))\,w_i(t),\ w_i(t)\rangle_\mathcal{M}\ \ud t = 0,\quad \forall i=1,\dots, k.
\end{equation}
The functions $w_i$ are independent solutions of
\begin{equation}
    \begin{cases}
\label{eq:linearpb}
    \ddot w = \nabla^2U (\xi(z,t))\,w, \ \ t >1
    \\
    w\in \tilde{\mathcal D}_{s}.
\end{cases}
\end{equation}

Consider now the fundamental solution $\Phi=\Phi(z,t) \in \mathbb R^{d(N-1)}$, satisfying
\begin{equation}\label{eq:system_Phi}
\begin{cases}
\ddot{\Phi} = \nabla^2U(\xi(z,t))\,\Phi, \ \  t > 1-\varepsilon
\\
\Phi(z,T_0) = I_d,
\\
\Phi \in \tilde{\mathcal D}_{1-\varepsilon}.
\end{cases}
\end{equation}

It is plain to verify that $w_i(t)=\Phi(z,t)\,w_i(T_0)$. Hence, $w_i(T_0)$ are linear independent elements of $\mathrm{Ker}\,\Phi(z,s)$. Therefore, $\mathrm{dim}(\mathrm{Ker}\,\Phi(z,s))\ge k$.

Vice versa, let $\theta_i \in \mathrm{Ker}\,\Phi(z,s)$. Hence, the function $v_i$ defined by
\[
v_i(t):=\Phi(z,t)\,\theta_i
\]
is a solution of \eqref{eq:linearpb}. Therefore, $v_i$ satisfies \eqref{eq:Hessian_zero}, and hence $\mathrm{dim}(\ud^2 \tilde{\mathcal{A}}_{\xi(z,s)}^s(\varphi(z,\cdot))) \ge k$. 

Hence, we have
\[
\mathrm{dim}(\mathrm{Ker}\,\ud^2 \tilde{\mathcal{A}}_{\xi(z,s)}^s(\tilde\varphi(z,\cdot))) = \mathrm{dim}(\mathrm{Ker}\,\Phi(z,s)).
\]

Therefore, if $x \in \mathcal U_{x_0}$ is such that $x=\xi(z,s)$, with $\lambda_1(\xi(z,s),s)$ not simple, then 
\[
\mathrm{dim}(\mathrm{Ker}\,\Phi(z,s)) \ge 2.
\]

\subsection{$\mathcal{H}^{d(N-1)-1}$-rectifiability of $\Gamma$}

\begin{thm}
    $\Gamma$ is countably $\mathcal{H}^{d(N-1)-1}$-rectifiable.
\end{thm}

\begin{proof}
We can split $\Gamma$ in two components:
\[
 \Gamma'=\{x \in \Gamma:\, \mathrm{dim}(\mathrm{Ker}(\ud^2 \mathcal A_{x}(\varphi^{x})))=1\},
\]
\[
\Gamma''=\{x \in \Gamma:\, \mathrm{dim}(\mathrm{Ker}(\ud^2 \mathcal A_{x}(\varphi^{x})))\ge 2\}.
\]

For every $x \in \Gamma$, as in \ref{eq:system_xi}, we can define a local flow $\xi=\xi(z,t)$, for $z \in \mathcal B$ and $t >1-\varepsilon$, whose image includes $\Gamma \cap \mathcal U_{x}$, with $\mathcal U_{x}$ a proper neighborhood of $x$.
 
We proved in Section \ref{ssec:simple} that for every $\mathcal W$  transversal section to $\dot \xi$ passing by $z=\gamma^x(T_0)$, we have
\begin{multline*}
    \Gamma'\cap \mathcal U_{x} =
    \{x=\xi(z,t),\,\mbox{ for some } \, z \in \mathcal W,\, \mbox{ and } \, t \in (1-\varepsilon, T_0),\, \mbox{ such that } \, \lambda_1(\xi(z,t),t)<\lambda_2(\xi(z,t),z)\},
\end{multline*}
and $\Gamma' = \bigcup_{k \in \mathbb N}\Gamma'\cap \mathcal U_{x_k}$, for (at most) countable many $x_k$. From \eqref{eq:gammaone} we proved that $\Gamma'$ is (at most) countably $(d(N-1)-1)$-rectifiable.

On the other hand, fixed $t=1$, thanks to the monotonicity of $t \mapsto \Lambda(z,t)$, there is a convenient way to rephrase $\Gamma''\cap \mathcal U_{x}$. We have
\begin{equation}
\label{eq:gammatwo}
\Gamma''\cap \mathcal U_{x}=\{x=\xi(z,1),\, \mbox{ for some }\, z \in \mathcal B,\, \mbox{ and }\, \lambda_1(\xi(z,1),1)=\lambda_2(\xi(z,1),1)\},
\end{equation}
resulting in a one-to-one correspondence between a subset of $\mathcal B$ and $\Gamma'' \cap \mathcal U_{x}$. This specific set of $z \in \mathcal B$ defined by \eqref{eq:gammatwo} is not (a priori) smooth, but we can still define the function $F:\mathcal{B}\rightarrow\R^{d(N-1)}$,
\[
F(z) = \xi(z,1).
\]

Differentiating \eqref{eq:system_xi}, we obtain \eqref{eq:system_Phi}, that is $\xi_z(z,1)=\Phi(z,1)$. There exists $\mathcal U_{x_0}'$ such that $x \in \Gamma'' \cap \mathcal U_{x_0}' = F(z)$, for $z \in \mathcal{B}$ such that $\Phi(z,1)$ has rank at most $d(N-1)-2$.  

Hence, by Sard-type lemma applied as in \cite[Lemma 6.6.1]{MR2041617} to the function $F$, we have that $\mathcal{H}^{d(N-1)-2+\varepsilon}(\Gamma''\cap \mathcal{U}_x) \le \mathcal H^{d(N-1)-2+\varepsilon}(F(\{z:\, \mathrm{rk}\,F_z(z) \le d(N-1)-2 \}))= 0$, for every $\varepsilon>0$. This, in particular, implies that $\Gamma''\cap \mathcal U_x$ is (at least) $(d(N-1)-1)$-rectifiable. Again, since $\Gamma''$ is a closed set, then there are (at most) countably many $x_k$ such that $\Gamma'' = \bigcup_{k \in \mathbb N} \Gamma''\cap \mathcal U_{x_k}$, where each $\mathcal U_{x_k}$ is a proper neighborhood of such $x_k$.
\end{proof}

\subsection{Estimate on the Hausdorff measure of $\Gamma \setminus \Sigma$}

Fix $x_0 \in \Gamma \setminus \Sigma$, and let $z_0=\gamma^{x_0}(T_0)$. Then, for every $z \in B_R(z_0)$, they are defined
\[
\theta(z) \ \mbox{generator of }\mathrm{Ker}\,\Phi(z,t(z)), \quad 
v(z) = \dot \xi(z,T_0).
\]
Observe that $\theta(z)$ and $v(z)$ are not parallel. Since $\eta(\cdot)=\dot \xi(z,\cdot)$ satisfies
\begin{equation*}
    \begin{cases}
        \ddot \eta = \mathcal M^{-1}\nabla^2 U(\xi)\eta
        \\
        \eta(T_0)=v(z)
        \\
        \eta \in {\mathcal D^{1,2}(1,T_0)},
    \end{cases}
\end{equation*}
then it must hold $\eta(t)=\Phi(z,t)\,v(z)$. Here, $\mathcal{D}^{1,2}(1,T_0)$ is defined similarly to $\mathcal D$ but on the time interval $(1,T_0)$ and without the requirement that $\eta(0)=0$.  If $v(z)= k\,\theta(z)$, for some $k \in\mathbb R$, then $\eta(t(z))=0$, and hence $\dot \xi(z,t(z))=0$. This contradicts at $t=t(z)$ the conservation of the mechanical energy, since $h\ge 0$ and $U>0$.

\smallskip
Now, consider the vector space of dimension $d(N-1)-2$, orthogonal to $\theta_0, v_0$, and denote it by $W_0$. For each $w \in W_0$, we can define $\eta_w=\eta_w(s)$ by
\[
\begin{cases}
    \dot \eta_w = \theta(\eta_w(s))
    \\
    \eta_w(0)=z_0+w.
\end{cases}
\]

We can then define $\mathcal W$ a $(d(N-1)-1)$-dimensional sub-manifold of $\mathcal B_0$, directly through the parameterization $U\ni (w,s)\mapsto \sigma(w,s)=\eta_w(s)$. It is plain to verify that $\partial_s \sigma(w,s)=\theta(\eta_w(s))$ and that $\partial_w \sigma(w,0) \in W_0$. This in turn implies that $v(\eta_w(0))\notin T_{\eta_w(0)}\mathcal W$. As a consequence, up to defining $\sigma$ in a subset of $U$, $v(\eta_w(s)) \notin T_{\eta_w(s)}\mathcal W$, for every $(w,s) \in \mathcal W$.

 Starting from the chart $(U,\sigma)$, we compute
\begin{multline*}
    \partial_s \hat \xi(w,s,t(w,s))=\partial_s(\xi(\sigma(w,s,t(w,s))))=\partial_z \xi(\sigma(w,s,t(w,s)))\,\partial_s\sigma(w,s,t(w,s))=
    \\
    \partial_z \xi(z,t)\,\theta(z) + \partial_t \xi(z,t(z))\,\langle \nabla t(z), \theta(z)\rangle.
\end{multline*}

For the second addendum of the last line above, we can adapt \cite[Proposition 6.6.8]{MR2041617} to our case by proving that for every $z\in \mathcal V_{z_0}\subset \mathcal B_0$ such that $\xi(z,t(z))\notin \Sigma$, it holds $\langle \nabla t(z), \theta(z)\rangle =0$. Indeed, as in \cite{MR2041617}, if $\langle \nabla t(\bar z), \theta(\bar z)\rangle \neq 0$, in view of \cite[Lemmas 6.6.6 and 6.6.11]{MR2041617}, it holds that for a $\rho>0$, the equation
\[
\xi(z,t(\bar z)) = \bar x  + s\,w
\]
does not admit any solution $z \in B_\rho(\bar z)$, if $s<0$ is small enough, $w$ is a generator of $\mathrm{Im}\,\Phi(\bar z, t(\bar z))^\perp$, and $\bar x = \xi(\bar z, t(\bar z))$. Hence, as in the proof of \cite[Proposition 6.6.8]{MR2041617}, but involving $\mathcal J_{\xi(\bar z,t(\bar z)),t(\bar z)}$ given in \eqref{finite-time problem} instead of $J_{\bar t}$, there exists a sequence of arcs $(\varphi_n)_n$ where every $\varphi_n$ minimizes $\mathcal J_{x_n,t(\bar z)}$ with $x_n=\xi(\bar z, t(\bar z))\pm (1/n)\,w \to \xi(\bar z,t(\bar z))$, and is such that $\varphi_n(T_0)+r_0(T_0) \notin B_\rho(\bar z)$. Up to a subsequence, passing to the limit for $n\to \infty$, we obtain another minimizer of $\mathcal J_{\xi(\bar z,t(\bar z)),t(\bar z)}$ given in \eqref{finite-time problem}, that is a contradiction.

This proves that $\ud \xi_{\mathcal W}(z,t(z)) \theta(z) = 0$, and hence $\mathrm{rk}(\ud\xi_{\mathcal W}(z,t(z))) \le d(N-1)-2$.

The application of the Sard-type Lemma on Banach manifolds  \cite{Sard} implies that
\[
\mathcal H^{
d(N-1)-2+\varepsilon}(\Gamma '\cap \mathcal U_x)\le \mathcal H^{d(N-1)-2+\varepsilon}(\{\xi_{\mathcal W}(z,t(z)):\,z \in \mathcal W\})=0,
\]
for every $\varepsilon >0$. 

Now, for $k\in \mathbb N$, consider
\[
\Gamma_k = \{x \in \Gamma': \, \lambda_2(x)\ge \lambda_1(x)+1/k\},
\]
that is a closed subset of $\Gamma$. Hence, for a fixed $n\in\N$
\[
\Gamma_k \cap B(0,n) = \bigcup_{x \in \Gamma_k} (\Gamma_k \cap \mathcal{U}_x\cap B(0,n)) = \bigcup_{x_1,\dots,x_M\in \Gamma_k} (\Gamma_k \cap \mathcal{U}_{x_i} \cap B(0,n)),
\]
where $B(0,n)$ is the ball centered at the origin with ray $n$. Therefore,
\[
\mathcal H^{d(N-1)-2+\varepsilon}(\Gamma_k\cap B(0,n)) \le \sum_{i=1}^{M}\mathcal H^{d(N-1)-2+\varepsilon}(\Gamma_k \cap \mathcal U_{x_i}\cap B(0,n))=0.
\]
It follows
\[
\mathcal H^{d(N-1)-2+\varepsilon}(\Gamma_k)\le \sum_{n=1}^{+\infty}\mathcal H^{d(N-1)-2+\varepsilon}(\Gamma_k \cap B(0,n)) =0
\]
and
\[
\mathcal{H}^{d(N-1)-2+\varepsilon}(\Gamma') \le \sum_{k=1}^{+\infty}\mathcal{H}^{d(N-1)-2+\varepsilon}(\Gamma_k)=0.
\]
Since
we already proved that $\mathcal H^{
d(N-1)-2+\varepsilon}(\Gamma '')=0$, for every $\varepsilon > 0$, we obtain
\[
\mathrm{dim}_{\mathcal H}(\Gamma \setminus \Sigma) \le d(N-1)-2.
\]

\end{document}